\documentclass[10pt, reqno]{amsart}

\usepackage{amsmath,amssymb,amsthm,amsfonts,verbatim}
\usepackage{microtype}
\usepackage[all,2cell]{xy}
\usepackage{mathtools}
\usepackage{graphicx}
\usepackage{pinlabel}
\usepackage{hyperref}
\usepackage{mathrsfs}
\usepackage{color}
\usepackage{enumerate}
\usepackage{bm}

\CompileMatrices

\usepackage[top=1.3in,bottom=1.9in,left=1.4in,right=1.4in]{geometry}

\theoremstyle{plain}
\newtheorem{theorem}{Theorem}[section]
\newtheorem{maintheorem}{Theorem}

\newtheorem{question}[theorem]{Question}

\newtheorem{proposition}[theorem]{Proposition}
\newtheorem{lemma}[theorem]{Lemma}
\newtheorem{conjecture}[theorem]{Conjecture}

\newtheorem{corollary}[theorem]{Corollary}

\theoremstyle{definition}
\newtheorem{definition}[theorem]{Definition}

\newtheorem{remark}[theorem]{Remark}

\newtheorem{construction}[theorem]{Construction}
\newtheorem{convention}[theorem]{Convention}

\newcommand{\nc}{\newcommand}
\nc{\dmo}{\DeclareMathOperator}

\nc{\Q}{\mathbb{Q}}
\nc{\F}{\mathbb{F}}
\nc{\R}{\mathbb{R}}
\nc{\Z}{\mathbb{Z}}
\nc{\C}{\mathbb{C}}
\nc{\Ell}{\mathcal{L}}
\nc{\M}{\mathcal{M}}
\nc{\K}{\mathcal{K}}
\nc{\I}{\mathcal{I}}
\nc{\U}{\mathcal U}
\nc{\disk}{\mathbb{D}}
\nc{\hyp}{\mathbb{H}}

\nc{\CP}{\mathbb{CP}}
\nc{\cS}{\mathcal{S}}
\dmo{\Mod}{Mod}
\dmo{\PMod}{PMod}
\dmo{\Diff}{Diff}
\dmo{\Homeo}{Homeo}
\dmo{\dist}{dist}
\dmo\BDiff{BDiff}
\dmo\SO{SO}
\dmo\Hom{Hom}
\dmo\SL{SL}
\dmo\Sp{Sp}
\dmo\rank{rank}
\dmo\sig{sig}
\dmo\Out{Out}
\dmo\Aut{Aut}
\dmo\Inn{Inn}
\dmo\GL{GL}
\dmo\PSL{PSL}
\dmo\BHomeo{BHomeo}
\dmo\EHomeo{EHomeo}
\dmo\EDiff{EDiff}
\nc\Sig{\Sigma}
\dmo\Teich{Teich}
\dmo\Fix{Fix}
\nc{\pair}[1]{\langle #1 \rangle}
\nc{\abs}[1]{\left| #1 \right|}
\nc{\action}{\circlearrowright}
\nc{\norm}[1]{\left | \left | #1 \right | \right |}
\nc{\abcd}[4]{\left(\begin{array}{cc} #1 & #2 \\ #3 & #4 \end{array}\right)}
\nc{\into}\hookrightarrow
\dmo{\Isom}{Isom}
\nc{\normal}{\vartriangleleft}
\dmo{\Vol}{Vol}
\dmo{\im}{Im}
\dmo{\Push}{Push}
\dmo{\Conf}{Conf}
\dmo{\PConf}{PConf}
\dmo{\id}{id}
\dmo{\Jac}{Jac}
\dmo{\Pic}{Pic}
\dmo{\Stab}{Stab}
\dmo{\Arf}{Arf}

\renewcommand{\epsilon}{\varepsilon}
\nc{\coloneq}{\mathrel{\mathop:}\mkern-1.2mu=}
\nc{\margin}[1]{\marginpar{\scriptsize #1}}
\nc{\para}[1]{\medskip\noindent\textbf{#1.}}
\nc{\red}[1]{\textcolor{red}{#1}}

\title{Monodromy and vanishing cycles in toric surfaces}

\author{Nick Salter}
\thanks{This material is based upon work supported by the National Science Foundation under Award No. DMS-1703181.}
\email{nks@math.columbia.edu}
\date{December 5, 2018}

\address{Department of Mathematics\\ Columbia University\\ 2990 Broadway, New York, NY 10027}

\begin{document}
\maketitle
\begin{abstract}
Given an ample line bundle on a toric surface, a question of Donaldson asks which simple closed curves can be vanishing cycles for nodal degenerations of smooth curves in the complete linear system. This paper provides a complete answer. This is accomplished by reformulating the problem in terms of the mapping class group-valued monodromy of the linear system, and giving a precise determination of this monodromy group.
\end{abstract}

\section{Introduction}
Let $X$ be a smooth toric surface and $\mathcal L$ an ample line bundle on $X$. In the complete linear system $\abs{\mathcal L}$, there is a hypersurface $\mathcal D$ known as the {\em discriminant locus} consisting of the singular curves $C \in \abs{\mathcal L}$. The complement 
\[
\mathcal M(\mathcal L) : = \abs{\mathcal L} \setminus \mathcal D
\]
therefore supports a tautological family of closed Riemann surfaces of some genus $g(\mathcal L)$. Topologically, this is a fiber bundle $\pi: \mathcal E(\mathcal L) \to \mathcal M(\mathcal L)$ with fiber $\Sigma_{g(\mathcal L)}$. Consequently, there is a {\em monodromy representation}
\[
\mu_{\mathcal L}: \pi_1(\mathcal M(\mathcal L),C_0) \to \Mod(C_0).
\]
Here, $C_0 \in \mathcal M(\mathcal L)$ is a fixed curve, and $\Mod(C_0):= \pi_0(\Diff^+(C_0))$ denotes the {\em mapping class group} of $C_0$ (see Section \ref{section:basics}). Under $\mu_\mathcal L$, a based loop $\gamma \in \pi_1(\mathcal M(\mathcal L),C_0)$ is mapped to (the isotopy class of) the diffeomorphism $\mu_\mathcal L(\gamma) \in \Diff(C_0)$ obtained by ``parallel transport'' of $C_0$ along $\gamma$. For details, see, e.g., \cite[Section 5.6.1]{FM}.

In this paper, we give a nearly complete answer to the following fundamental question. Define
\[
\Gamma_{\mathcal L} := \im(\mu_{\mathcal L}) \leqslant \Mod(\Sigma_{g(\mathcal L)}).
\]
\begin{question}\label{question:mod}
What is $\Gamma_{\mathcal L}$? When is it a finite-index subgroup of $\Mod(\Sigma_{g(\mathcal L)})$? Can one give a precise characterization of $\Gamma_{\mathcal L}$?
\end{question}

Question \ref{question:mod} is closely related to a question posed by Donaldson \cite{donaldson}. Fix a curve $C_0 \in \mathcal M(\mathcal L)$ and an identification $C_0 \cong \Sigma_{g(\mathcal L)}$. Define a {\em vanishing cycle} for $\mathcal L$ as a simple closed curve $\gamma$ on $C_0$ for which there is a degeneration of $C_0$ to a curve $C'$ with a single node, such that $\gamma$ becomes null-homotopic on $C'$. If $c$ is a vanishing cycle, then necessarily the Dehn twist $T_c$ lies in $\Gamma_{\mathcal L}$; it arises from a loop in $\mathcal M(\mathcal L)$ encircling the nodal curve in $\abs{\mathcal L}$. 

\begin{question}[Donaldson]\label{question:donaldson}
For $\mathcal L$ an ample line bundle on a smooth toric surface $X$, which curves (on a fixed $C_0$) are vanishing cycles?
\end{question}

A first insight into Questions \ref{question:mod} and \ref{question:donaldson} is to observe the presence of an invariant ``higher spin structure''. Let $K_X$ denote the canonical bundle of $X$. The {\em adjoint line bundle} of $\mathcal L$ is the line bundle $\mathcal L \otimes K_X$. Define $r \in \mathbb N$ to be the highest root of $\mathcal L \otimes K_X$ in $\Pic(X)$. As explained in Proposition \ref{proposition:invtspin}, associated to $\mathcal L \otimes K_X$ is a {\em $\Z/r\Z$-valued spin structure} $\phi_{\mathcal L}$, and the associated stabilizer subgroup $\Mod(\Sigma_{g(\mathcal L)})[\phi_{\mathcal L}]$ (see Definition \ref{definition:stabilizer}). Proposition \ref{proposition:invtspin} asserts that necessarily $\Gamma_\mathcal L \leqslant \Mod(\Sigma_{g(\mathcal L)})[\phi_{\mathcal L}]$. The function $\phi_\mathcal L$ gives rise to a notion of {\em admissible curve} and the associated subgroup $\mathcal T_{\phi_{\mathcal L}} \leqslant \Mod(\Sigma_{g(\mathcal L)})[\phi_{\mathcal L}]$ of {\em admissible twists} (see Definition \ref{definition:admissible}). If a curve $c$ is a vanishing cycle, it is necessarily admissible; see Lemma \ref{lemma:twists}. Our main theorem asserts that these necessary conditions are also sufficient (at least ``virtually'' so, in the case $r$ is even).

\begin{maintheorem}\label{theorem:toric}
Let $\mathcal L$ be an ample line bundle on a smooth toric surface $X$ for which the generic fiber is not hyperelliptic. Assume $r > 1$ or else $g(\mathcal L) \ge 5$.
\begin{itemize}
\item If $r$ is odd, then $\Gamma_{\mathcal L} = \Mod(\Sigma_{g(\mathcal L)})[\phi_{\mathcal L}]$. 
\item If $r$ is even, then $\Gamma_{\mathcal L} \leqslant \Mod(\Sigma_{g(\mathcal L)})$ is a finite-index subgroup that contains $\mathcal T_{\phi_{\mathcal L}}$.
\end{itemize}
In either case, $[\Mod(\Sigma_{g(\mathcal L)}):\Gamma_{\mathcal L}]$ is finite. Moreover, Question \ref{question:donaldson} admits the following complete answer: {\em a curve $\gamma$ is a vanishing cycle if and only if $\gamma$ is an admissible curve.}
\end{maintheorem}

We remark that many familiar algebraic surfaces such as $\CP^2$ and $\CP^1 \times \CP^1$ are smooth toric surfaces. For instance, as a special case of Theorem \ref{theorem:toric} we obtain the following theorem concerning smooth plane curves. The case $d = 5$ was addressed in \cite{saltermonodromy}, while the cases $d \le 4$ are either classical or trivial.

\begin{theorem}
Set $g = {d-1 \choose 2}$, and define
\[
\Gamma_d \leqslant \Mod(\Sigma_g)
\]
to be the monodromy group of the family of smooth curves in $\CP^2$ of degree $d$, i.e. the group $\Gamma_\mathcal L$ for the line bundle $\mathcal L = \mathcal O(d)$ on $\CP^2$. Then there exists a $\Z/(d-3)\Z$-valued spin structure $\phi_d$ such that the following hold.
\begin{itemize}
\item If $d$ is even, then $\Gamma_d = \Mod(\Sigma_g)[\phi_d]$.
\item If $d$ is odd, then $\Gamma_d$ is of finite index in $\Mod(\Sigma_g)[\phi_d]$, where $\Gamma_d$ contains the subgroup $\mathcal T_{\phi_d}$ of admissible twists.
\end{itemize}

\end{theorem}

Theorem \ref{theorem:toric} also addresses a conjecture that was independently formulated by the author in \cite{saltermonodromy} in the case of $X = \CP^2$, and in full generality by  Cr\'etois--Lang \cite{CL}.

\begin{conjecture}\label{conjecture:image}
For any pair $(X, \mathcal L)$ as above, there is an equality
\[
\Gamma_{\mathcal L} = \Mod(\Sigma_{g(\mathcal L)})[\phi_{\mathcal L}].
\]
\end{conjecture}
Theorem \ref{theorem:toric} resolves Conjecture \ref{conjecture:image} in the affirmative whenever $r$ is odd, and shows that in the case $r$ even, $\Gamma_{\mathcal L}$ is at least of finite index in $\Mod(\Sigma_{g(\mathcal L)})[\phi_{\mathcal L}]$. 

Theorem \ref{theorem:toric} is proved using a combination of methods from toric geometry and the theory of the mapping class group. On the toric end of the spectrum, we make essential use of the powerful results developed by  Cr\'etois--Lang in \cite{CL}. The centerpiece of their theory is a combinatorial model for a curve $C_0 \in \mathcal M(\mathcal L)$ based around a convex lattice polygon. Their results give a description of vanishing cycles in terms of lattice points and line segments, and allow one to produce many elements of $\Gamma_{\mathcal L}$.  Cr\'etois--Lang developed their methods in order to address Question \ref{question:donaldson} and Conjecture \ref{conjecture:image} in the case $r \le 2$, and obtained complete answers in these cases. See \cite{CL} for the case $r = 1$, and \cite{CL2} for the case $r = 2$, as well as the case where the general fiber is hyperelliptic.

On the mapping class group side, we carry out an extensive investigation of the groups $\Mod(\Sigma_g)[\phi]$ and $\mathcal T_\phi$ mentioned above. We remark here that the theory of higher spin structures does not require the presence of a specific ample line bundle $\mathcal L$, and so we adjust notation accordingly and refer to Riemann surfaces $\Sigma_g$, spin structures $\phi$, etc. Our main result here is a general criterion for a collection of Dehn twists to generate (a finite-index subgroup of) $\Mod(\Sigma_g)[\phi]$, given in Theorem \ref{theorem:networkgenset}.

\para{Outline of the paper} The bulk of the paper (Sections \ref{section:mcg} - \ref{section:network}) is devoted to developing the mapping class group technology necessary to show that the vanishing cycles investigated by Cr\'etois--Lang generate a finite-index subgroup of the mapping class group. This culminates in Theorem \ref{theorem:networkgenset}. Portions of Theorem \ref{theorem:networkgenset} are established earlier in Proposition \ref{proposition:twistgen} and Proposition \ref{proposition:twistgeneven}. 

Sections \ref{section:mcg} - \ref{section:action} contain preliminary results that are used throughout the paper. Section \ref{section:mcg} collects the necessary background on mapping class groups; these results are standard and are included so as to fix notation and terminology, and to serve as a guide to the reader approaching the paper from a background in toric geometry. Section \ref{section:spin} presents the basic theory of higher spin structures, building off the foundational work of Humphries--Johnson \cite{HJ}. Section \ref{section:action} describes the action of the mapping class group on the set of higher spin structures. This yields several crucial corollaries (Corollaries \ref{corollary:allvalues}, \ref{corollary:allvalueseven}, and \ref{corollary:dneven}) concerning the existence of configurations of curves with prescribed properties which are used extensively in subsequent sections.

Theorem \ref{theorem:networkgenset} gives a criterion for a collection of Dehn twists to generate the so-called {\em admissible subgroup} $\mathcal T_\phi$ associated to a higher spin structure $\phi$. A study of the admissible subgroup is sufficient to answer Question \ref{question:donaldson}. The reader interested only in this portion of Theorem \ref{theorem:toric} can skip Sections \ref{section:dodd} and \ref{section:deven} and jump directly from Section \ref{section:action} to Section \ref{section:trick}. 

The proof of Theorem \ref{theorem:networkgenset} is carried out in Sections \ref{section:trick} - \ref{section:network}. Section \ref{section:trick} establishes the connectivity of certain simplicial complexes acted on by the stabilizer subgroup of a higher spin structure. These results are used in the argument of Section \ref{section:push}, and also underlie the method by which the admissible subgroup is used to study the set of vanishing cycles. Section \ref{section:push} is devoted to a study of certain subgroups of the admissible subgroup; the main result Proposition \ref{proposition:submakesT} furnishes a generating set for $\mathcal T_\phi$ in terms of these subgroups. Section \ref{section:network} introduces the notion of a {\em network}; ultimately a network is a technical device used to factor the generators given in Proposition \ref{proposition:submakesT} into products of Dehn twists. Theorem \ref{theorem:networkgenset} gives a sufficient condition, formulated in the language of networks, for a collection of Dehn twists to generate a subgroup containing the admissible subgroup.

The portion of Theorem \ref{theorem:toric} that goes beyond Question \ref{question:donaldson} concerns establishing that the admissible subgroup is finite-index in the mapping class group. This is the content of Sections \ref{section:dodd} and \ref{section:deven}, which treat the case where the $\Z/r\Z$-valued spin structure under study has $r$ odd or even, respectively. The arguments for these two cases are substantially different, owing to the fact that in the case of $r$ even, the higher spin structure has an Arf invariant which must be accounted for in various guises. 

The net result of Sections \ref{section:mcg} - \ref{section:network} is a criterion for a finite collection of Dehn twists to generate a finite-index subgroup of the mapping class group. In the final two sections, these results are applied in the setting of monodromy groups of linear systems on toric surfaces. Section \ref{section:CL} contains the necessary background material on toric surfaces, concentrating on the work of Cr\'etois--Lang describing a particular finite collection of vanishing cycles. Section \ref{section:proof} exhibits a network amongst the set of vanishing cycles discussed in Section \ref{section:CL} and verifies that this network satisfies the hypotheses of Theorem \ref{theorem:networkgenset} in order to obtain Theorem \ref{theorem:toric}.

\para{Acknowledgements} The author would like to extend his warmest thanks to R. Cr\'etois and L. Lang for helpful discussions of their work. He would also like to acknowledge C. McMullen for some insightful comments on a preliminary draft, and M. Nichols for a productive conversation. A special thanks is due to an anonymous referee for a very careful reading of the preprint and for many useful suggestions, both mathematical and expository.

\section{Mapping class groups}\label{section:mcg}
This section collects background material on mapping class groups that will be used throughout the arguments in Sections \ref{section:spin} through \ref{section:network}. Most of the material can be found in \cite{FM} and so will only be touched on briefly. The exception to this is the $D_n$ relation of Section \ref{subsection:relations}, which will consequently be dealt with in greater detail.

\subsection{Basics}\label{section:basics} The material in this section is almost certainly well-known to a reader conversant in mapping class groups, but is included so as to fix notation and terminology.

\para{Genus, boundary, punctures} All surfaces under consideration are oriented and of finite type. A surface of genus $g$ with $n$ punctures and $b$ boundary components is denoted by $\Sigma_{g,b}^n$. When one or more of $b, n = 0$, the corresponding decoration will be omitted. 

\para{Intersection numbers} Let $a,b$ be simple closed curves on a surface $S$. Often we will confuse the distinction between a simple closed curve and its isotopy class. The {\em geometric intersection number} between $a,b$ will be notated $i(a,b)$ (see \cite[Section 1.2.3]{FM}). For {\em oriented} simple closed curves $a,b$, the algebraic intersection number is denoted $\pair{a,b}$. Of course, algebraic intersection depends only on the homology classes $[a],[b] \in H_1(S; \Z)$.

\para{Mapping class groups} Let $\Sigma_{g,b}^n$ be a surface. The {\em mapping class group} of $\Sigma_{g,b}^n$, written $\Mod(\Sigma_{g,b}^n)$, is defined as
\[
\Mod(\Sigma_{g,b}^n) := \pi_0(\Diff^+(\Sigma_{g,b}^n, \partial \Sigma_{g,b}^n)),
\]
where $\Diff^+(\Sigma_{g,b}^n, \partial \Sigma_{g,b}^n)$ denotes the group of orientation-preserving diffeomorphisms of $\Sigma_{g,b}^n$ that restrict to the identity on the boundary of $\Sigma_{g,b}^n$ and fix the punctures {\em pointwise} (not merely setwise, as some authors adopt).

\begin{figure}
\labellist
\small
\pinlabel $a_1$ [l] at 53.6 80.8
\pinlabel $a_2$ [l] at 96 80.8
\pinlabel $b_1$ [t] at 44 40
\pinlabel $b_2$ [t] at 87.2 40
\pinlabel $b_3$ [t] at 133.6 40
\pinlabel $b_g$ [tr] at 262.4 40
\pinlabel $c_1$ [b] at 66.4 62.4
\pinlabel $c_2$ [b] at 110.4 62.4
\pinlabel $d_1$ [bl] at 341.6 92.8
\pinlabel $d_2$ [bl] at 354.4 71.2
\pinlabel $d_3$ [tl] at 357.6 35.6
\pinlabel $d_4$ [tl] at 347.2 17.6
\pinlabel $d_5$ [tl] at 329.6 0.8
\endlabellist
\includegraphics{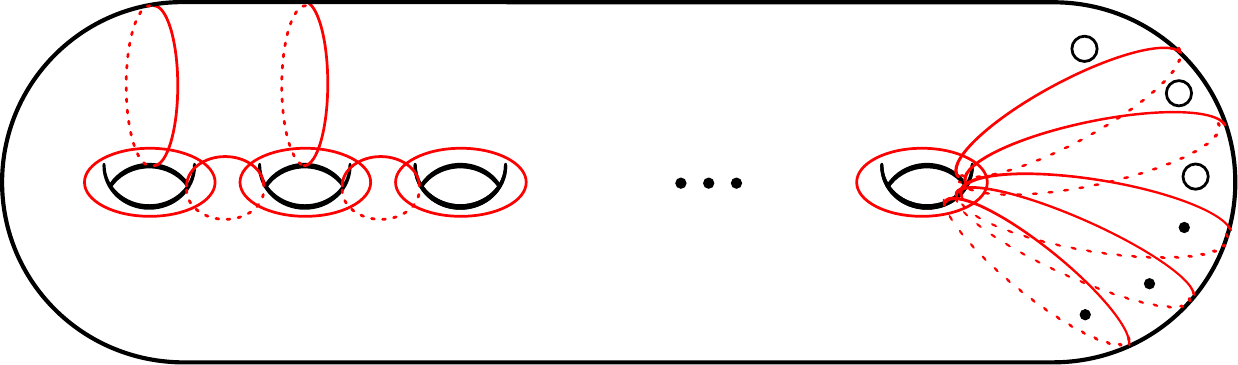}
\caption{The standard generators for $\Sigma_{g,3}^3$.}
\label{figure:humph}
\end{figure}
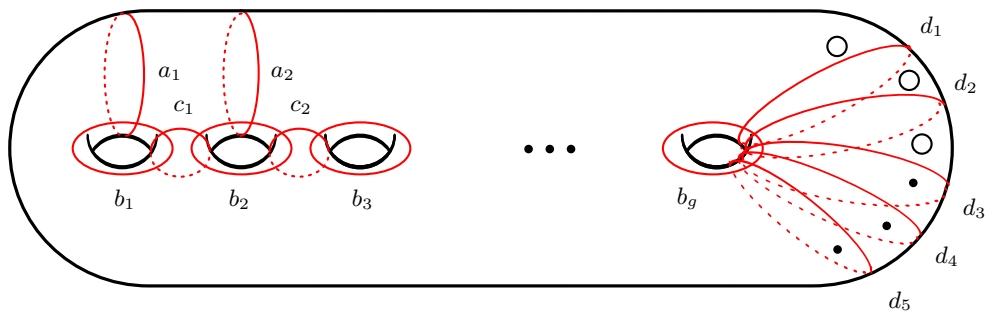

\para{The standard generators} For a simple closed curve $a$ on $\Sigma_{g,b}^n$, the {\em left-handed} Dehn twist about $a$ is written $T_a$. For $g \ge 2$, the {\em standard generators} form a generating set for $\Mod(\Sigma_{g,b}^n)$ consisting of the Dehn twists about the curves $a_1, a_2, b_1, \dots, b_g, c_1, \dots, c_{g-1}, d_1, \dots, d_{b+n-1}$ shown in Figure \ref{figure:humph}. 

\para{The change-of-coordinates principle} The classification of surfaces theorem asserts that if $S, S'$ are two (connected and orientable) surfaces of finite type with the same genus, number of punctures, and number of boundary components, then there is a diffeomorphism $f: S \to S'$. This is often exploited in the study of mapping class groups in the guise of the ``change-of-coordinates principle''. It is difficult to write down a single, all-encompassing statement of the change-of-coordinates principle, but informally, it states that {\em any configuration of curves, arcs, and/or subsurfaces of a surface $S$ is determined up to diffeomorphism by combinatorial information alone.} In the present paper, the change-of-coordinates principle will often be invoked tacitly. The reader interested in a more thorough discussion of the change-of-coordinates principle is referred to \cite[Section 1.3]{FM}. 

One consequence of the change-of-coordinates principle is that it becomes easy to understand the $\Mod(S)$ orbits of many different kinds of configurations. As an example, we discuss here the action on {\em geometric symplectic bases} for $S$.
\begin{definition}\label{definition:gsb}
Let $S$ be a surface of genus $g \ge 0$ with $n\ge 0$ boundary components and $b \ge 0$ punctures. A {\em geometric symplectic basis} for $S$ is a collection of oriented simple closed curves $\mathcal B = \{\alpha_1, \beta_1, \dots, \alpha_g, \beta_g\}$ satisfying the following properties:
\begin{enumerate}
\item $i(a_i, b_i) = 1$ for each $i = 1, \dots, g$, and all other pairs of elements of $\mathcal B$ are disjoint,
\item $\pair{[a_i], [b_i]} = 1$ for each $i = 1, \dots, g$.
\end{enumerate}
\end{definition}

\begin{remark}
The (homology classes of the) curves in a geometric symplectic basis form a basis for $H_1(S;\Z)$ in the sense of linear algebra only when $n + b \le 1$. In this paper, geometric symplectic bases are used to study $\Z/r\Z$-valued spin structures. Proposition \ref{proposition:HCC} and Theorem \ref{theorem:gsb} together imply that a $\Z/r\Z$-valued spin structure is determined by its ``signature'' (Definition \ref{definition:signature}) in combination with its values on a geometric symplectic basis. 
\end{remark}

The following is a typical statement that is proved using the change-of-coordinates principle.

\begin{lemma}\label{lemma:gsb}
Let $\mathcal B$ and $\mathcal B'$ be two geometric symplectic bases for $S$. Then there is a diffeomorphism $f: S \to S$ such that $f(\mathcal B) = \mathcal B'$. 
\end{lemma}

\subsection{The Birman exact sequence}\label{subsection:birman}
A reference for this subsection is \cite[Section 4.2]{FM}. Consider a surface $\Sigma_{g,b}^n$ with $n \ge 1$ and $2g+b+n \ge 4$. There is an inclusion $\Sigma_{g,b}^n \into \Sigma_{g,b}^{n-1}$ obtained by filling $p$ in. This induces the {\em Birman exact sequence}
\begin{equation}\label{equation:birman1}
1 \to \pi_1(\Sigma_{g,b}^{n-1}, p) \to \Mod(\Sigma_{g,b}^n) \to \Mod(\Sigma_{g,b}^{n-1}) \to 1. 
\end{equation}

There is a slight variation on the Birman exact sequence where one fills in a boundary component with a closed disk, originally due to Johnson. In order to formulate this, we recall that the unit tangent bundle to a surface $S$ is written $UTS$. Then the inclusion $\Sigma_{g,b}^n \to \Sigma_{g,b-1}^n$ induces the short exact sequence
\begin{equation}\label{equation:birman2}
1 \to \pi_1(UT\Sigma_{g,b-1}^n, \tilde p) \to \Mod(\Sigma_{g,b}^n) \to \Mod(\Sigma_{g,b-1}^n) \to 1,
\end{equation}
where $\tilde p$ is a unit tangent vector based at $p$. In both situations, the kernels admit descriptions in terms of Dehn twists. Consider first the case of (\ref{equation:birman1}). Let $\alpha$ be an {\em embedded, oriented} simple closed curve based at $p$, corresponding to an element $\alpha \in \pi_1(\Sigma_{g,b}^{n-1}, p)$. Let $\alpha_L$ (resp. $\alpha_R$) denote the left (resp. right) side of a neighborhood of $\alpha$. Both $\alpha_L, \alpha_R$ are simple closed curves on $\Sigma_{g,b}^n$. Then $\alpha \in \pi_1(\Sigma_{g,b}^{n-1}, p)$ corresponds to $T_{\alpha_L} T_{\alpha_R}^{-1} \in \Mod(\Sigma_{g,b}^n)$. The embedding $P: \pi_1(\Sigma_{g,b}^{n-1}, p) \to \Mod(\Sigma_{g,b}^n)$ is known as the {\em point-pushing map}, and $\pi_1(\Sigma_{g,b}^{n-1})$ is often referred to as the {\em point-pushing subgroup} of $\Mod(\Sigma_{g,b}^n)$.

It is a basic topological fact that for any surface $\Sigma_{g,b}^{n-1}$, there exists a collection of simple closed curves $\alpha_1, \dots, \alpha_k$ based at $p$, such that $\{\alpha_1, \dots, \alpha_k\}$ generates $\pi_1(\Sigma_{g,b}^{n-1},p)$. In practice, this means that to exhibit $\pi_1(\Sigma_{g,b}^n, p)$ as a subgroup of some group $H \leqslant \Mod(\Sigma_{g,b}^n)$, it suffices to exhibit this finite collection of multitwists.

In the case of (\ref{equation:birman2}), everything is much the same. Let $\Sigma_{g,b}^n \into \Sigma_{g,b-1}^n$ be an inclusion corresponding to capping off a boundary component $\Delta$ of $\Sigma_{g,b}^n$. Let $p \in \Sigma_{g,b-1}^n$ be a point on the interior of this new disk, and $\tilde p$ a tangent vector at $p$. Suppose that $\tilde \alpha \in \pi_1(UT\Sigma_{g,b-1}^n,\tilde p)$ corresponds to a {\em framed} simple closed curve $\alpha$ based at $\tilde p$. We define $\alpha_L$ and $\alpha_R$ as before. Then 
\[
P(\tilde \alpha) = T_{\alpha_L} T_{\alpha_R}^{-1} T_\Delta^k
\]
where $k \in \Z$ is the winding number of the tangent vector field specified by $\tilde \alpha$, relative to the tangential framing of the underlying curve $\alpha$. The subgroup $\pi_1(UT\Sigma_{g,b-1}^n, \tilde p)$ is known as the {\em disk-pushing subgroup} of $\Mod(\Sigma_{g,b}^n)$. 

There is an analogous set of ``geometric'' generators for $\pi_1(UT\Sigma_{g,b-1}^n, \tilde p)$. Let $\alpha_1, \dots, \alpha_k$ be a collection of $C^1$-embedded simple closed curves on $\Sigma_{g,b-1}^n$ based at $p$ such that $\pi_1(\Sigma_{g,b-1}^n,p) = \pair{\alpha_1, \dots, \alpha_k}$ as above. Each $\alpha_i$ determines an element $\tilde \alpha_i \in \pi_1(\Sigma_{g,b-1}^n, \tilde p)$ via the so-called {\em Johnson lift}, whereby $\alpha_i$ is framed using the forward-pointing tangent vector. Suppose that each $\tilde \alpha_i$ is based at some common tangent vector $\tilde p$. Then $\pi_1(UT\Sigma_{g,b-1}^n, \tilde p)$ has a generating set of the following form:
\[
\pi_1(UT\Sigma_{g,b-1}^n, \tilde p) = \pair{\tilde \alpha_1, \dots, \tilde \alpha_k, \zeta},
\]
where $\zeta$ is the loop around the $S^1$ fiber in the fibration $S^1 \to UT\Sigma_{g,b-1}^n \to \Sigma_{g,b-1}^n$. In terms of Dehn twists, the Johnson lifts $\tilde \alpha_i$ correspond to mapping classes $T_{\alpha_{i,L}} T_{\alpha_{i,R}}^{-1}$ as before, while $\zeta$ corresponds to $T_\Delta$. 

\subsection{Relations}\label{subsection:relations}
In this subsection we collect various relations in the mapping class group that will be used throughout the paper.

\para{The braid relation} Suppose $a,b$ are simple closed curves satisfying $i(a,b) = 1$. Then the corresponding Dehn twists satisfy the {\em braid relation}:
\[
T_a T_b T_a = T_b T_a T_b.
\] 
We will also employ the following alternative form, formulated in terms of the curves $a,b$ themselves:
\[
T_a T_b (a) = b.
\]

\para{The chain relation} A {\em chain} of simple closed curves is a sequence $(a_1, \dots, a_n)$ of simple closed curves such that $i(a_i, a_{i+1}) = 1$ and $i(a_i, a_j) = 0$ otherwise. Let $\nu$ denote a regular neighborhood of a chain of length $n$, where the representative curves $a_1, \dots, a_n$ are in minimal position. When $n$ is odd, $\partial \nu$ has two components $\Delta_1$ and $\Delta_2$; for $n$ even, $\partial \nu = \Delta$ is a single (necessarily separating) curve. Abusing terminology, we will speak of the boundary of a chain itself, by which we mean the boundary of $\nu$. Given a subsurface $S$ with $1$ or $2$ boundary components, a chain $a_1, \dots, a_n$ of curves on $S$ is {\em maximal} if there is a deformation retraction of $S$ onto $a_1 \cup \dots \cup a_n$. The following appears as  \cite[Proposition 4.12]{FM}.

\begin{proposition}[Chain relation]\label{proposition:chain}
For $n$ odd, 
\[
(T_{a_1} \dots T_{a_n})^{n+1} = T_{\Delta_1} T_{\Delta_2},
\]
and for $n$ even,
\[
(T_{a_1} \dots T_{a_n})^{2n+2} = T_\Delta.
\]
\end{proposition}

\begin{remark}The intersection pattern of a chain of $n$ simple closed curves is recorded by the Dynkin diagram of type $A_n$, where vertices in the graph are adjacent if the corresponding curves intersect, and are nonadjacent if the curves are disjoint. Such a chain of curves determines a homomorphism from the Artin group $A(A_n)$ of type $A_n$ into the mapping class group $\Mod(\nu)$, where generators of $A(A_n)$ are sent to Dehn twists about the corresponding curves. 

Under this homomorphism, the chain relation is a consequence of the fact that $A(A_n)$ has nontrivial center. The twist(s) about the boundary component(s) appearing on the right-hand side of the expressions in Proposition \ref{proposition:chain} are elements of the center of $\Mod(\nu)$, while the left-hand side merely gives the expression for a generator of $Z(A(A_n))$ as a word in the standard generators of $A(A_n)$. In \cite[Section 2.4]{matsumoto}, Matsumoto explains how to determine the precise expression for this central element as a Dehn multitwist; this is the principle underlying the ``$D_n$ relation'' given in Proposition \ref{proposition:dnrel} below.
\end{remark}

\para{The $\bm{D_n}$ relation} There is an analogous (though less ubiquitous) relation that arises from a configuration of curves whose intersection pattern is modeled on the Dynkin diagram of type $D_n$. Proposition \ref{proposition:dnrel} below is the specialization of \cite[Proposition 2.4]{matsumoto} to the case of an Artin group of type $D_n$. The case of $n$ odd is treated explicitly in \cite[Theorem 1.5]{matsumoto}, while the case of $n$ even is given an alternate proof in \cite[Proposition 4.5]{saltermonodromy}.

\begin{figure}
\labellist
\small
\pinlabel $\Delta_0$ [tl] at 22.4 13.6
\pinlabel $\Delta_1$ [l] at 332 96
\pinlabel $\Delta_1'$ [l] at 332 24
\pinlabel $\Delta_2$ [bl] at 292 100
\pinlabel $a$ [l] at 68.8 96
\pinlabel $a'$ [l] at 68.8 24
\pinlabel $c_1$ [tr] at 40.8 48
\pinlabel $c_2$ [bl] at 93 76.8
\pinlabel $c_3$ [tr] at 150 46
\pinlabel $c_4$ [bl] at 168 76.8
\pinlabel $c_{2g-1}$ [tr] at 260 48
\pinlabel $c_{2g}$ [br] at 288 76
\endlabellist
\includegraphics{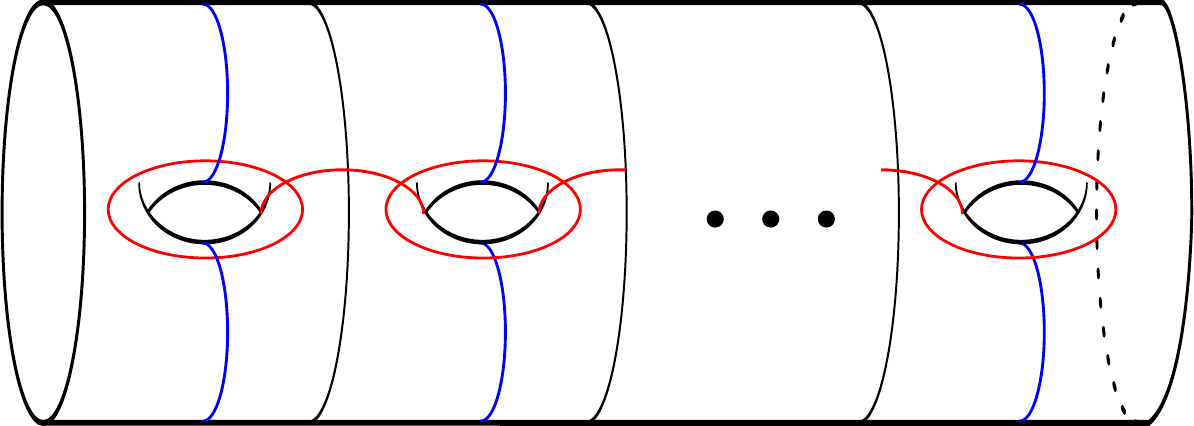}
\caption{The configuration of curves used in the $D_n$ relation. Note the presence of the unlabeled curve $c_{2g+1}$ on the far right side.}
\label{figure:dnrel}
\end{figure}
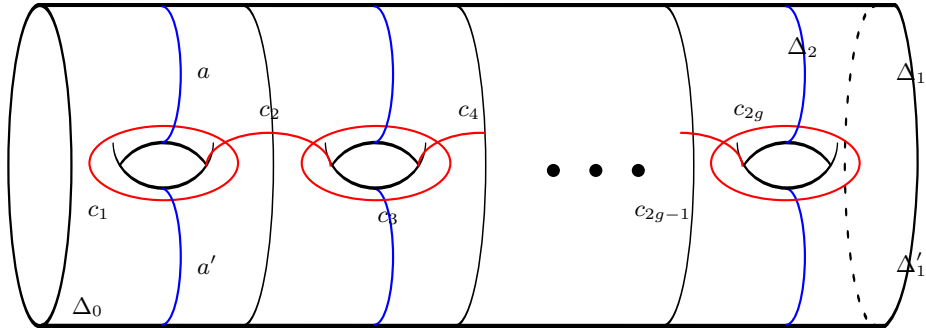

\begin{proposition}[$D_n$ relation]\label{proposition:dnrel}
Let $n \ge 3$ be given, and express $n = 2g+1$ or $n = 2g+2$ according to whether $n$ is odd or even. With reference to Figure \ref{figure:dnrel}, let $H_n$ be the group generated by elements of the form $T_x$, with $x \in \mathscr D_n$ one of the curves below:
\begin{align*}
\mathscr D_{n} &= \{a, a', c_1, \dots, c_{n-2}\}.
\end{align*}
Then for $n = 2g+1$ odd,
\[
T_{\Delta_0}^{2g-1} T_{\Delta_2} \in H_n,
\]
and for $n = 2g+2$ even,
\[
T_{\Delta_0}^{g} T_{\Delta_1} T_{\Delta_1'}\in H_n.
\]
\end{proposition}

The $D_n$ relation has some useful consequences which we record in Corollary \ref{corollary:dn} below. It is necessary to first describe the curves $C_k$ that will appear in the statement. For $1 \le k \le g+1$, let $\nu_k$ be a regular neighborhood of the subconfiguration $\mathscr D_{2k+1} \subset \mathscr D_n$. Each such $\nu_k$ is a surface of genus $k$ with two boundary components. One of these is $\Delta_0$; the other is defined to be the curve $C_k$. Note in particular that $C_g = \Delta_2$ and that $C_{g+1}$ is the unlabeled boundary component of the ambient surface on the far right side of Figure \ref{figure:dnrel}.
\begin{corollary}\label{corollary:dn} Fix notation as in Proposition \ref{proposition:dnrel}, and for $1 \le \ell \le 2g+3$, consider the configurations
\[
\mathscr D_{\ell} = \{a, a', c_1, \dots, c_{\ell-2}\}
\]
as in Figure \ref{figure:dnrel}. Let $H_{2g+3}^+$ be the group generated by $H_{2g+3}$ and the Dehn twist $T_{\Delta_1}$. Then the following assertions hold:
\begin{enumerate}
\item $T_{\Delta_1'} \in H_{2g+3}^+$,
\item $T_{C_k}^m \in H_{2g+3}^+$ for any $1 \le k \le g+1$ and any $m$ such that $(2k - 1)m$ divides $g$.
\end{enumerate}
\end{corollary}
\begin{proof}
The proof of (1) follows from an important simple principle. Given a mapping class $f$ and a simple closed curve $d$, there is a relation
\[
f T_d f^{-1} = T_{f(d)}.
\]
It follows that if $f, T_d \in H_{2g+3}^+$, then also $T_{f(d)} \in H_{2g+3}^+$. To establish (1), we will find $f \in H_{2g+3}^+$ such that $f(c_{2g+1}) = \Delta_1'$. This will be accomplished by means of the braid relation.

The curves $a, a', c_1, \dots, c_{2g}$ are arranged in the configuration of the $D_{2g+2}$ relation; the boundary components correspond to $\Delta_0, \Delta_1, \Delta_1'$. By the $D_{2g+2}$-relation (Proposition \ref{proposition:dnrel}),
\[
T_{\Delta_0}^{g}T_{\Delta_1} T_{\Delta_1'} \in H_{2g+3}^+,
\]
and since $T_{\Delta_1} \in H_{2g+3}^+$ by assumption, also $T_{\Delta_0}^{g} T_{\Delta_1'} \in H_{2g+3}^+$. Since $\Delta_0$ is disjoint from both $c_{2g+1}$ and $\Delta_1'$, the braid relation implies that 
\[
T_{c_{2g+1}} T_{\Delta_0}^{g} T_{\Delta_1'}(c_{2g+1}) = T_{c_{2g+1}} T_{\Delta_1'} (c_{2g+1}) = \Delta_1'.
\]
Since $(T_{\Delta_0}^{g} T_{\Delta_1'})\in H_{2g+3}^+$, this shows $T_{\Delta_1'} \in H_{2g+3}^+$ as required. 

We observe that (2) follows from the $D_{2k-1}$ relation (as applied to the subconfiguration $\mathscr D_{2k-1}$) and the claim that $T_{\Delta_0}^{g} \in H_{2g+3}^+$; this latter assertion follows from the $D_{2g+2}$ relation (applied to $\mathscr D_{2g+2}$) and (1).
\end{proof}

\subsection{The Torelli group}
Most of the material in this subsection can be found in \cite[Chapter 6]{FM}, but see also \cite{johnsonsurvey}. We begin by observing that the action of $\Mod(\Sigma_g)$ on $H_1(\Sigma_g; \Z)$ preserves the algebraic intersection pairing $\pair{\cdot, \cdot}$, leading to the {\em symplectic representation}
\begin{equation}\label{equation:sprep}
\Psi: \Mod(\Sigma_g) \to \Sp(2g, \Z).
\end{equation}
This is classically known to be a surjection. The {\em Torelli group}, notated $\mathcal I_g$, is the kernel of this representation:
\[
\mathcal I_g := \ker(\Psi). 
\]

\para{Bounding pairs and separating twists} There are two types of elements in $\mathcal I_g$ that will be of particular importance. Suppose that $c,d$ are simple closed curves such that $c \cup d$ bounds a subsurface $S \cong \Sigma_{h,2}$. Then $T_c T_d^{-1} \in \mathcal I_g$ is known as a {\em bounding pair map}. The {\em genus} of a bounding pair map is slightly ambiguous: if $c \cup d$ bounds a surface $\Sigma_{h,2}$, then also $c \cup d$ bounds a surface $\Sigma_{g-h-1,2}$ on the other side. One defines the genus of $T_c T_d^{-1}$ as $\min\{h, g-h-1\}$. The second important class of elements is the class of {\em separating twists} - these are Dehn twists $T_c$ for $c$ a separating curve. The {\em genus} of a separating twist $T_c$ that bounds a subsurface of genus $h$ is defined as $g(c) = \min\{h, g-h\}$.\

\para{The Johnson homomorphism} A fundamental tool in the study of $\mathcal I_g$ is the {\em Johnson homomorphism}, due to D. Johnson in \cite{johnsonhom}. This is a surjective homomorphism
\begin{equation}\label{equation:johnsonhom}
\tau: \mathcal I_g \to \wedge^3 H_\Z / H_\Z,
\end{equation}
where for convenience we define $H_A:= H_1(\Sigma_g; A)$ for some abelian group $A$. The embedding $H_\Z \into \wedge^3 H_\Z$ is defined via
\[
z \mapsto z \wedge (x_1 \wedge y_1 + \dots + x_g \wedge y_g),
\]
where $\{x_1, \dots, y_g\}$ is a symplectic basis for $H_\Z$. Recall that a symplectic basis must satisfy $\pair{x_i, y_i} = 1$ and $\pair{x_i,x_j} = \pair{x_i,y_j} = 0$ for $i \ne j$. 

We will not need to know a precise definition of $\tau$, but it will be useful to know some basic properties of $\tau$, including how to compute $\tau$ on bounding pair maps and separating twists. 

\begin{lemma}[Johnson, \cite{johnsonhom}]\label{lemma:jhom}\ 
\begin{enumerate}
\item $\tau$ is $\Sp(2g; \Z)$-equivariant, with respect to the conjugation action on $\mathcal I_g$ and the evident action on $\wedge^3 H_\Z/H_\Z$.
\item $\tau(T_c) = 0$ for any separating twist $T_c$.
\item Let $c\cup d$ bound a subsurface $\Sigma_{h,2}$. Choose any further subsurface $\Sigma_{h,1} \subset \Sigma_{h,2}$, and let $\{x_1, y_1, \dots, x_h, y_h\}$ be a symplectic basis for $H_1(\Sigma_{h,1};\Z)$. Then 
\[
\tau(T_c T_d^{-1}) = (x_1 \wedge y_1 + \dots + x_h \wedge y_h)\wedge [c],
\]
where $c$ is oriented with $\Sigma_{h,2}$ to the left. In the case $h = 1$, if $\alpha,\beta,\gamma$ is a maximal chain on $\Sigma_{1,2}$, then 
\[
\tau(T_cT_d^{-1}) = [\alpha] \wedge [\beta] \wedge [\gamma].
\]
\end{enumerate}
\end{lemma}

\para{The Johnson kernel} The {\em Johnson kernel}, written $\mathcal K_g$, is the kernel of the Johnson homomorphism:
\[
\mathcal K_g := \ker(\tau).
\]
A fundamental theorem of Johnson gives an alternate characterization of $\mathcal K_g$ in terms of separating twists. 

\begin{theorem}[Johnson, \cite{johnson2}]\label{theorem:johnson2}
Let $\mathcal T_g$ be the subgroup of $\mathcal K_g$ generated by all separating twists of genus at most two. Then for all $g \ge 3$,
\[
\mathcal T_g = \mathcal K_g.
\]
\end{theorem}

\section{Spin structures}\label{section:spin}
In this section we introduce and study higher spin structures and their stabilizer subgroups. Section \ref{subsection:spinbasics} defines higher spin structures and presents the work of Humphries--Johnson that gives a cohomological formulation of a higher spin structure. Section \ref{subsection:operations} discusses some cut-and-paste operations on simple closed curves and how these operations interact with higher spin structures. Section \ref{subsection:firstproperties} defines spin structure stabilizer groups and some important elements of these groups. Finally Section \ref{subsection:classical} explains the connection between higher spin structures and the classical theory of spin structures as quadratic forms on vector spaces over $\Z/2\Z$.

\subsection{Spin structures}\label{subsection:spinbasics}
Let $S$ be a surface of genus $g \ge 0$. For simplicity, we assume in this section that $S$ can have boundary components but not punctures; for surfaces with puncture, one can simply remove an open neighborhood of the puncture to produce a surface with boundary. Let $\mathcal S$ denote the set of isotopy classes of oriented simple closed curves on $S$. In keeping with standard practice, the term ``curve'' will often be used to refer to an isotopy class of curves. Crucially, curves are not required to be essential (see property (2) of Definition \ref{definition:wnf}). The following definition is due to Humphries--Johnson \cite{HJ}; see Remark \ref{remark:reconcile} for a discussion of how to reconcile their definition with the one given here.

\begin{definition}[spin structure]\label{definition:wnf} A {\em $\Z/r \Z$-valued spin structure} on $S$ is a function $\phi: \mathcal S \to \Z/r\Z$ satisfying the following two properties. 
\begin{enumerate}
\item (Twist-linearity) Let $c,d \in \mathcal S$ be arbitrary. Then 
\[
\phi(T_c(d)) = \phi(d) + \pair{d,c}\phi(c) \pmod{r}.
\]
\item (Normalization) For $\zeta$ the boundary of an embedded disk $D \subset S$, oriented with $D$ to the left, $\phi(\zeta) = 1$.
\end{enumerate}
\end{definition}

\begin{remark}\label{remark:reconcile}
The definition of a $\Z/r\Z$-valued spin structure presented in Definition \ref{definition:wnf} is superficially different from that given by Humphries--Johnson \cite{HJ} in several respects. First, it should be noted that Humphries--Johnson study a more general notion of ``twist-linear function''; only spin structures are needed in the present paper. Secondly, in Definition \ref{definition:wnf}, simple closed curves are considered up to the equivalence relation of isotopy. This is an {\em a priori} different equivalence relation than the notion of ``$L$-direct homotopy'' defined in \cite[p. 366]{HJ}. The precise definition of $L$-directness is cumbersome, but if two simple closed curves $c$ and $d$ are $L$-directly homotopic, then they are in particular homotopic in the ordinary sense. It is well-known that homotopy and isotopy determine the same equivalence relation on simple closed curves, see e.g. \cite[Proposition 1.10]{FM}. Moreover, an isotopy is an instance of an $L$-direct homotopy, so that these notions coincide in our setting.
\end{remark}

\begin{remark}
In the literature, higher spin structures go by various names and have various definitions; the term ``$r$-spin structure'' is especially common. It is not {\em a priori} clear how to reconcile the definition given here with others. See Remark \ref{remark:roots} for a brief discussion, or \cite[Sections 2-3]{saltermonodromy} for a fuller treatment.
\end{remark}

\begin{convention}\label{convention:orientation}
Often we will speak of the value $\phi(c)$ where $\phi$ is some $\Z/r\Z$-valued spin structure and $c$ is a curve {\em without a specified orientation}. Such a statement should be understood to mean that there is some unspecified orientation on $c$ for which $\phi(c)$ has the stated value.
\end{convention}

\para{The Johnson lift} Recall from the discussion in Section \ref{subsection:birman} the notion of the {\em Johnson lift}. In \cite{HJ}, Johnson-Humphries use the Johnson lift to give a homological formulation of a $\Z/r\Z$-valued spin structure. The following is an amalgamation of the Remark following Theorem 2.1 and Theorem 2.5 of \cite{HJ}.

\begin{theorem}[Humphries--Johnson]\label{theorem:HJhomological}

Let $S$ be a surface. An element $\psi \in H^1(UTS; \Z/r\Z)$ determines a $\Z/r\Z$-valued spin structure via
\[
\alpha \mapsto \psi(\tilde \alpha),
\]
where $\alpha$ is a simple closed curve on $S$ and $\tilde \alpha$ is the Johnson lift. This determines a $1-1$ correspondence
\[
\{\phi \mbox{ a $\Z/r\Z$-valued spin structure on $S$}\} \leftrightarrow \{ \phi \in H^1(UTS; \Z/r\Z) \mid \phi(\zeta) = 1\}.
\]
\end{theorem}

\begin{remark}\label{remark:admits}
From the standard presentation 
\[
\pi_1(UT\Sigma_g) = \pair{a_1, b_1,\dots, a_g, b_g, \zeta \mid \prod_{i =1}^g [a_i, b_i] = \zeta^{2-2g}}
\]
 and the Universal Coefficient Theorem, one sees that
\[
H^1(UT\Sigma_g; A) \cong \Hom(\pi_1(UT\Sigma_g), A) \cong \Hom(\Z^{2g}\oplus \Z/(2g-2)\Z, A).
\]
The factor $\Z/(2g-2)\Z$ in $H_1(UT\Sigma_g;\Z) = \Z^{2g} \oplus \Z/(2g-2)\Z$ is generated by the class of $\tilde \zeta$, the Johnson lift of the non-essential curve $\zeta$. In the case $A = \Z/r \Z$, it follows that there exists a spin structure if and only if $r \mid (2g-2)$.
\end{remark}

\begin{remark}\label{remark:roots}
Via covering space theory, $\Z/r\Z$-valued spin structures on $\Sigma_g$ are in correspondence with cyclic $r$-fold coverings $\widetilde {UT\Sigma_g} \to UT \Sigma_g$ that restrict to connected coverings of the fiber $S^1$. In the setting of linear systems on toric surfaces, such coverings arise from the presence of roots of the canonical line bundle of the generic fiber. See Proposition \ref{proposition:invtspin} and the references mentioned therein for more details.
\end{remark}

An important consequence of Theorem \ref{theorem:HJhomological} is the fact that  $\Z/r\Z$-valued spin structures satisfy a property known as the {\em homological coherence criterion}. This follows by combining Theorem \ref{theorem:HJhomological} with \cite[Lemma 2.4]{HJ}.

\begin{proposition}[Homological coherence criterion]\label{proposition:HCC}
Let $\phi$ be a $\Z/r\Z$-valued spin structure on a surface $S$, and let $S' \subset S$ be a subsurface with Euler characteristic $\chi(S') = m$. Suppose $\partial(S') = c_1 \cup \dots \cup c_k$, and all $c_i$ are oriented so that $S'$ is to the left. Then $\sum \phi(c_i) = m$.  
\end{proposition}

Theorem \ref{theorem:HJhomological} shows that $\Z/r\Z$-valued spin structures are determined by a finite amount of data. In the sequel it will be useful to have an explicit criterion for the equality of two $\Z/r\Z$-valued spin structures. The following appears as \cite[Corollary 2.6]{HJ}.

\begin{theorem} [Humphries--Johnson] \label{theorem:gsb}
Let $S$ be a surface of genus $g \ge 0$. Let $\mathcal B = \{\gamma_1, \dots, \gamma_{k}\}$ be a set of oriented simple closed curves such that the set $\{[\gamma_1], \dots, [\gamma_k]\}$ forms a basis for $H_1(\Sigma_g;\Z)$. Suppose $\phi$ and $\psi$ are $\Z/r\Z$-valued spin structures on $S$. Then $\phi = \psi$ if and only if $\phi(\gamma_i) = \psi(\gamma_i)$ for each $\gamma_i \in \mathcal B$.
\end{theorem}

\begin{figure}
\labellist
\Huge
\pinlabel $\rightsquigarrow$ at 180 45
\endlabellist
\includegraphics{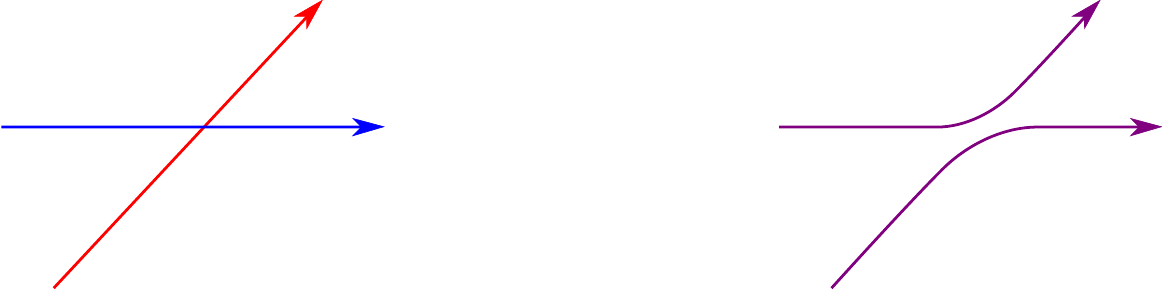}
\caption{The smoothing operation.}
\label{figure:smoothing}
\end{figure}

\begin{figure}
\labellist
\Huge
\pinlabel $\rightsquigarrow$ at 180 45
\endlabellist
\includegraphics{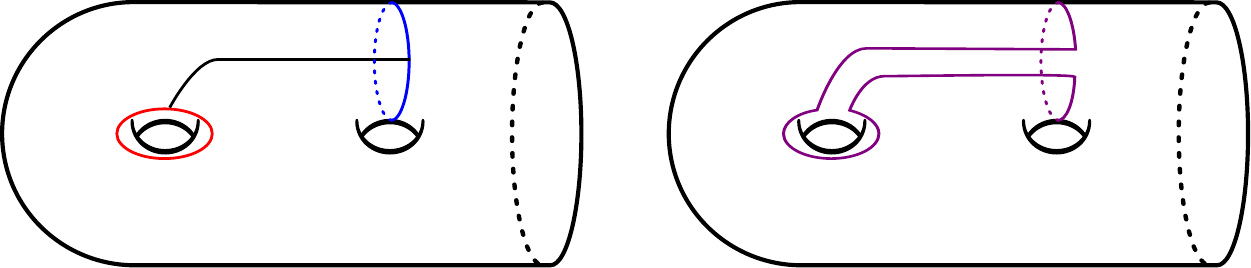}
\caption{The curve-arc sum operation.}
\label{figure:casum}
\end{figure}

\subsection{Operations on curves}\label{subsection:operations} In what follows, we will make use of two procedures for constructing new simple closed curves from old. Here, we define these operations and collect some facts about how they interact with spin structures.

\begin{definition} [Smoothing, curve sum]\label{definition:curvesum}
Let $\mathcal{C} = \{c_1, \dots, c_n\}$ be a collection of oriented embedded simple closed curves on a surface $S$. Suppose that all intersections between elements of $\mathcal C$ are transverse. The {\em smoothing} of $\mathcal{C}$ is the embedded multicurve obtained from $\mathcal C$ by smoothly resolving all intersections in the unique orientation-preserving way. See Figure \ref{figure:smoothing}.

Now suppose $\alpha$ and $\beta$ are oriented simple closed curves. For natural numbers $m,n$, define the {\em curve sum} $m \alpha + n \beta$ as the smoothing of $m$ parallel copies of $\alpha$ with $n$ copies of $\beta$. In case $m < 0$ or $n < 0$, the curve sum $m \alpha + n \beta$ can be defined as before, with the orientation on $\alpha$ (resp. $\beta$) reversed if $m < 0$ (resp. $n < 0$). See Figure \ref{figure:casum}.

 By choosing arbitrary representatives in minimal position, both of these operations are well-defined on the level of isotopy classes.
\end{definition}

\begin{lemma}\label{lemma:smoothing}
Let $\alpha, \beta$ be oriented simple closed curves in minimal position, and let $\phi$ be a $\Z/r\Z$-valued spin structure. Then for any integers $m,n$,
\[
\phi(m\alpha + n \beta) = m \phi(\alpha) + n \phi(\beta).
\]
If in addition, $i(\alpha, \beta) = 1$ and $\gcd(m,n) =1$, then $m \alpha + n \beta$ has a single component.
\end{lemma}
\begin{proof}
The first assertion follows directly from the identification of $\phi$ with an element of $H^1(UTS; \Z/r\Z)$ given in Theorem \ref{theorem:HJhomological}, while the second is straightforward to verify.
\end{proof}

\begin{definition}[Curve-arc sum]
Let $\alpha$ and $\beta$ be disjoint oriented simple closed curves on $S$, and let $\epsilon$ be an arc connecting $\alpha$ to $\beta$ whose interior is disjoint from $\alpha \cup \beta$. A regular neighborhood $\nu$ of $\alpha \cup \epsilon \cup \beta$ is homeomorphic to $\Sigma_{0,3}$. Two of the boundary components of $\nu$ are homotopic to $\alpha$ and $\beta$, respectively. The {\em curve-arc-sum} $\alpha +_{\epsilon} \beta$ is the third boundary component of $\nu$. Again, the curve-arc sum descends to the level of isotopy classes. 
\end{definition}

\begin{lemma}\label{lemma:curvearcsum}
Let $\alpha, \beta, \epsilon, \nu$ be as above. Orient $\alpha, \beta$ so that $\epsilon$ connects the left sides of $\alpha, \beta$, and orient $\alpha +_{\epsilon} \beta$ so that the subsurface $\nu$ is to the right. Then for $\phi$ a $\Z/r\Z$-valued spin structure,
\[
\phi(\alpha+_{\epsilon}\beta) = \phi(\alpha) + \phi(\beta) + 1.
\]
In addition, on the level of homology, $[\alpha+_{\epsilon} \beta] = [\alpha] + [\beta]$.
\end{lemma}
\begin{proof}
Observe that $\chi(\nu) = -1$. By the homological coherence criterion (Proposition \ref{proposition:HCC}),
\[
-1 = \phi(\alpha) + \phi(\beta) + \phi(-(\alpha+_\epsilon \beta)),
\]
where $-(\alpha+_\epsilon \beta)$ denotes the curve $\alpha +_\epsilon \beta$ with orientation opposite to that specified above. By the case $(m,n) = (-1,0)$ of Lemma \ref{lemma:smoothing}, it follows that $\phi(-(\alpha+_\epsilon \beta)) = - \phi(\alpha+_\epsilon \beta)$, from which the first claim follows. The second claim is an immediate consequence of the orientation conventions. 
\end{proof}

\subsection{The group \boldmath$\Mod(S)[\phi]$; first examples of elements}\label{subsection:firstproperties} For any surface $S$, there is an obvious (left) action of $\Mod(S)$ on the set of $\Z/r\Z$-valued spin structures: for $f \in \Mod(S)$ and $c \in \mathcal S$, define $(f\cdot \phi)(c) = \phi(f^{-1}(c))$. Similarly, if $f: S \to S'$ is a diffeomorphism and $\phi$ is a $\Z/r\Z$-valued spin structure on $S'$, there is a pullback $f^*(\phi)$ defined on $S$ via $(f^*\phi)(c) = \phi(f(c))$. 
 
\begin{definition}[Stabilizer subgroup]\label{definition:stabilizer} Let $\phi$ be a spin structure on a surface $S$. The {\em stabilizer subgroup} of $\phi$, written $\Mod(S)[\phi]$, is defined as
\[
\Mod(S)[\phi] = \{ f\in \Mod(S) \mid (f \cdot \phi) = \phi\}.
\]
\end{definition}

Let $\phi$ be a $\Z/r \Z$-valued spin structure on a surface $S$. Below we discuss some fundamental examples of elements in $\Mod(S)[\phi]$. 

\para{Dehn twist powers and admissible twists} The twist-linearity formula of Definition \ref{definition:wnf} immediately implies the following  characterization of Dehn twists in $\Mod(S)[\phi]$.
\begin{lemma}\label{lemma:twists}
Let $c$ be a simple closed curve on $S$. If $c$ is separating, then $T_c \in \Mod(S)[\phi]$. If $c$ is nonseparating, then $T_c^k \in \Mod(S)[\phi]$ if and only if $k \phi(c) \equiv 0 \pmod r$. In particular, for $c$ nonseparating, $T_c \in \Mod(S)[\phi]$ if and only if $\phi(c) = 0$. 
\end{lemma}

\begin{definition}[Admissible]\label{definition:admissible}
A nonseparating curve $c$ with $\phi(c) = 0$ is called an {\em admissible curve}. The associated element $T_c \in \Mod(S)[\phi]$ is called an {\em admissible twist}. The group generated by all admissible twists is written $\mathcal T_\phi$, and is called the {\em admissible subgroup}. 
\end{definition}

\para{Fundamental multitwists} Let $P\cong \Sigma_{0,3}$ be a pair of pants with boundary curves $\alpha,\beta,\gamma$. Suppose that $\phi(\alpha) = a, \phi(\beta) = b,$ and that $\phi(\gamma) = c$, with all curves oriented so that $P$ lies to the left.  By the homological coherence property, $a + b + c = -1$. 
\begin{definition}\label{definition:multi}
Let $P$ and $\phi$ be as above. A {\em $\phi$-bounding multitwist} associated to $P$, denoted $T_P(x,y,z)$, is given by
\[
T_P(x,y,z) = T_\alpha^{x}T_\beta^y T_\gamma^z
\]
for any choice of integers $x,y,z$ such that $T_P \in \Mod(S)[\phi]$. 
\end{definition}
By the above, $T_P(r, r, r)$ is a $\phi$-bounding multitwist for any $P$ and $\phi$, but for special values of $a,b,c$, there are much simpler examples. 
\begin{lemma}\label{lemma:fundtwist}
Let $P$ be as above, and suppose that $b = -a$, so that $c = -1$. Then $T_P(1,-1, b) = T_\alpha T_\beta^{-1} T_\gamma^{b}$ is a $\phi$-bounding multitwist. The element $T_P(1,-1,b)$ is called a {\em fundamental multitwist} for $P$ and is denoted $T_P$. 
\end{lemma}
\begin{proof}
Let $d$ be any curve on $S$; we must show that $\phi(d) = \phi(T_\alpha T_\beta^{-1} T_\gamma^b(d))$. As $\alpha, \beta, \gamma$ are all disjoint, the twist-linearity property, in combination with the fact that $[\alpha + \beta + \gamma] = 0$ in $H_1(S)$, gives
\begin{align*}
\phi(T_\alpha T_\beta^{-1} T_\gamma^b(d)) &= \phi(d) + \pair{d,\alpha} a - \pair{d,\beta} b - \pair{d, \gamma}b\\
	&= \phi(d) - \pair{d, \alpha + \beta + \gamma}b\\
	&= \phi(d).
\end{align*}
\end{proof}
\begin{remark} Of course, if $T_P(1,-1,b)$ is a fundamental multitwist, then so is $T_P(1,-1, b+ k r)$ for any $k \in \Z$. An important special case is when $\phi(\alpha) = \phi(\beta) = 0$. Then $T_\alpha T_\beta^{-1}$ is a fundamental multitwist. 
\end{remark}

\subsection{``Classical'' spin structures}\label{subsection:classical} Spin structures in the sense of Definition \ref{definition:wnf} generalize the more familiar notion of a ``classical'' spin structure. In our setting, a classical spin structure is a spin structure valued in $\Z/2\Z$. We pause here to briefly review the theory of classical spin structures and the connection with our definition. These results, especially the theory of the Arf invariant, will play a crucial role in the analysis of $\Z/r\Z$-valued spin structures for $r$ even to be begun in Proposition \ref{proposition:onlyarf} and Corollary \ref{corollary:allvalueseven}, and returned to in Section \ref{section:deven}.

Let $V$ be a vector space over the field $\Z/2\Z$ equipped with a nondegenerate symplectic pairing $\pair{\cdot,\cdot}$ (i.e. a nondegenerate bilinear pairing satisfying $\pair{x,x} = 0$ for all $x \in V$). The motivating example is $V = H_1(\Sigma_g; \Z/2\Z)$ with the intersection pairing. A {\em $\Z/2\Z$ quadratic form relative to $\pair{\cdot, \cdot}$} is a function $q: V \to \Z/2\Z$ such that for any $x,y \in V$, the equation
\begin{equation}\label{equation:q}
q(x+y) = q(x) + q(y) + \pair{x,y}
\end{equation}
holds. 

Let $\mathcal B = \{x_1, y_1, \dots, x_g, y_g\}$ be a symplectic basis for $V$. It is clear that $q$ is determined by its values on $\mathcal B$. Define $Q(V, \pair{\cdot, \cdot})$ as the set of $\Z/2\Z$ quadratic forms on $V$ relative to $\pair{\cdot, \cdot}$; then a choice of $\mathcal B$ provides a bijection 
\[
Q(V, \pair{\cdot,\cdot}) \cong (\Z/2\Z)^{2g}.
\]
There is an evident action of the group $\Sp(V, \pair{\cdot, \cdot})$ of $\pair{\cdot,\cdot}$-preserving automorphisms on $Q(V, \pair{\cdot,\cdot})$. 

To understand the set of orbits, we introduce the {\em Arf invariant}. The Arf invariant of $q$ is the element of $\Z/2\Z$ defined by the following formula:
\[
\Arf(q) := \sum_{i = 1}^g q(x_i)q(y_i).
\]
$q$ is said to be {\em even} or {\em odd} according to whether $\Arf(q) = 0,1$ respectively; in this way we will speak of the {\em parity} of a spin structure. The following records some well-known properties of the Arf invariant.

\begin{lemma}\label{lemma:arfprops}
Let $(V, \pair{\cdot, \cdot})$ be a symplectic vector space over $\Z/2\Z$, and let $q,q' \in Q(V, \pair{\cdot,\cdot})$ be quadratic forms.
\begin{enumerate}
\item\label{item:arfwd} $\Arf(q)$ is well-defined independently of the choice of symplectic basis,
\item\label{item:arforbit} $q$ and $q'$ are in the same orbit of $\Sp(V, \pair{\cdot, \cdot})$ if and only if $\Arf(q) = \Arf(q')$.
\end{enumerate}
\end{lemma}

Suppose now that $\phi$ is a $\Z/2d\Z$-valued spin structure in the sense of Definition \ref{definition:wnf}. The reduction $\Z/2d\Z \to \Z/2\Z$ associates to $\phi$ an underlying $\Z/2\Z$-valued spin structure which we denote $\bar \phi$. {\em A priori}, $\bar \phi$ is defined on the set $\mathcal S$ of isotopy classes of oriented curves on $\Sigma_g$. It follows from \cite[Theorem 1A]{johnsonspin} that $\bar \phi$ factors through the map $[\cdot]: \mathcal S \to H_1(\Sigma_g;\Z/2\Z)$. The induced map 
\[
\bar{\bar\phi}: H_1(\Sigma_g;\Z/2\Z)\to \Z/2\Z
\] 
is not quite a classical spin structure, but it follows from \cite[Theorem 1A]{johnsonspin} that the function
\begin{equation}\label{equation:classicalspin}
q_\phi := \bar{\bar \phi} + 1
\end{equation}
{\em does} determine a classical spin structure.

In the remainder of this paper we will exclusively use the term ``spin structure'' in the sense of Definition \ref{definition:wnf}. The reader versed in classical spin structures should be aware that certain formulas appear different in this setting. For instance, a Dehn twist about some nonseparating curve $c$ preserves a $\Z/2\Z$-valued spin structure $\phi$ if and only if $\phi(c) = 0$, whereas a transvection about some nonzero $v\in V$ preserves a quadratic form $q$ if and only if $q(v) = 1$. Likewise, if $\phi$ is a $\Z/2d\Z$-valued spin structure, the formula for the Arf invariant $\Arf(\phi)$ of the underlying classical spin structure is given by
\begin{equation}\label{equation:arf}
\Arf(\phi) = \sum_{i = 1}^g (\bar \phi(x_i) + 1)(\bar \phi(y_i) + 1) \pmod 2.
\end{equation}

\section{The action of the mapping class group on spin structures} \label{section:action}
In what follows, we will need to understand the action of $\Mod(\Sigma_g)$ on the set of $\Z/r\Z$-valued spin structures. Following the discussion in Section \ref{subsection:classical}, when $r$ is even, the Arf invariant shows there are at least two orbits of $\Mod(S)$ on the set of $\Z/r \Z$-valued spin structures, but it is not clear what happens for odd $r$, nor whether there are further invariants leading to more orbits. The goal of this section is to give a complete description of this action. In the case of $r$ odd, the mapping class group action on the set of $\Z/r\Z$-valued spin structures is described in Proposition \ref{proposition:one}, and for $r$ even it is described in Proposition \ref{proposition:onlyarf}. Both results can be understood as asserting that there are no ``higher Arf invariants''.

\subsection{Odd \boldmath$r$} In the case of $r$ odd, we will need to consider surfaces with multiple boundary components. Before formulating the results, we define the notion of the {\em signature} of a $\Z/r\Z$-valued spin structure.

\begin{definition}[Signature of a $\Z/r\Z$-valued spin structure]\label{definition:signature}
Let $S$ be a surface equipped with a $\Z/r\Z$-valued spin structure $\phi$. Enumerate the boundary components as $\Delta_1, \dots, \Delta_n$, each one oriented so that $S$ is to the {\em left}. The {\em signature of $S$ rel $\phi$} is defined as the $n$-tuple of values $\sig(S, \phi) = (\phi(\Delta_1), \dots, \phi(\Delta_n))$. We will also speak of the signature of an individual $\Delta_k$, defined as $\phi(\Delta_k)$.
\end{definition}

\begin{proposition}\label{proposition:one}
Fix an odd integer $r$. Let $S$ be a surface, and let $\phi$ and $\psi$ be $\Z/r \Z$-valued spin structures on $S$ satisfying $\sig(\phi) = \sig(\psi)$. Suppose that either $g(S) \ne 1$ or else $g = 1$ and there is at least one boundary component with signature $\phi(c_1) = \psi(c_1) = k$ for some $k$ such that $k+1$ generates $\Z/r \Z$. Then there is a mapping class $f \in \Mod(S)$ such that $f^*(\psi) = \phi$. 
\end{proposition}

\begin{proof}
The proof is by induction on the genus $g(S)$. If $g(S) = 0$, then every curve $c$ on $S$ is separating, so that the homological coherence criterion (Proposition \ref{proposition:HCC}) implies that $\phi(c)$ and $\psi(c)$ can be computed just from the signature. In this case, it follows that in fact $\phi = \psi$.

For $g(S) \ge 1$, let $\alpha_0,\beta_0$ be curves on $S$ satisfying $i(\alpha_0,\beta_0) = 1$. Choose nonzero integers $a,b \in \Z$ such that $a \equiv \phi(\alpha_0)$ and $b \equiv \phi(\beta_0) \pmod r$. Let $d = \gcd(a,b)$, and define $x = a/d, y=b/d$; by construction, $x,y$ are coprime. Define the curve $\alpha_1  = y\alpha_0 -x \beta_0$ in the sense of Definition \ref{definition:curvesum}. By Lemma \ref{lemma:smoothing}, $\phi(\alpha_1) = 0$.

Choose any curve $\gamma_0$ satisfying $i(\alpha_1, \gamma_0) = 1$. We claim there exists some {\em separating} oriented curve $c$ on $S$ that is disjoint from $\gamma_0 \cup \alpha_1$ and such that $\phi(c) = k$ for $k$ such that $k+1$ generates $\Z/r \Z$. In the case $g(S) = 1$ this is true by hypothesis, while for $g(S) \ge 2$, the curve $c$ can be taken to be the neighborhood of some subsurface $T \subset S$ with $T \cong \Sigma_{1,1}$ and $T$ disjoint from $\alpha_1 \cup \gamma_0$. In this case, orient $c$ so that $T$ lies to the right. By the homological coherence property, such a $c$ satisfies $\phi(c) = 1$, and since $r$ is odd, the claim follows.

Either $c$ is isotopic to a boundary component of $S$ and is oriented with $S$ lying to the right, or else (by the change-of-coordinates principle), there exists an arc $\epsilon_0$ from the left side of $\gamma_0$ to the left side of $c$ that is disjoint from $\alpha_1$.  In the former case, there exists an arc $\epsilon_0$ from the right side of $\gamma_0$ to the right side of $c$ that is disjoint from $\alpha_1$. Via Lemma \ref{lemma:curvearcsum}, the curve-arc sum $\gamma_1 = \gamma_0 +_{\epsilon_0} c$ satisfies $\phi(\gamma_1) = \phi(\gamma_0) - (k+1)$ in the former case, and $\phi(\gamma_1) = \phi(\gamma_0) + (k+1)$ in the latter case. Since the curve $c$ is null-homologous, there is an equality $[\gamma_1] = [\gamma_0]$. A further appeal to the change-of-coordinates principle shows that there is another arc $\epsilon_1$ from the left side of $\gamma_1$ to the left of $c$, again disjoint from $\alpha_1$. This process can therefore be repeated indefinitely, giving rise to curves $\gamma_m$ satisfying $\phi(\gamma_m) = \phi(\gamma_0) + m(k+1)$. By hypothesis, $k+1 \in \Z/r \Z$ is a generator, so that $\phi(\gamma_m) = 0$ for some $m$. Set $\beta_1 = \gamma_m$ for such an $m$. By construction, $i(\alpha_1, \beta_1) = 1$. 

Likewise, construct curves $\alpha_1', \beta_1'$ satisfying $i(\alpha_1', \beta_1') = 1$ and $\psi(\alpha_1') = \psi(\beta_1') = 0$. Take (open) regular neighborhoods $T_1$ and $T_1'$ of $\alpha_1\cup \beta_1$ and $\alpha_1' \cup \beta_1'$, respectively. There is a diffeomorphism $f_1: T_1 \to T_1'$ for which $f_1(\alpha_1) = \alpha_1'$ and $f_1(\beta_1) = \beta_1'$. Define $c_1 = \partial \overline{T_1}$ and $c_1' = \partial \overline{T_1'}$. Then $\phi(c_1) = 1$ when $c_1$ is oriented with $T_1$ on the right, and similarly for $c_1'$. The curve $c_1$ is therefore a boundary component of $S \setminus T_1$ with signature $\phi(c_1) = 1$, and likewise for $c_1'$. This shows that the inductive hypothesis is satisfied, and so there exists a diffeomorphism $f_2: S \setminus T_1 \to S\setminus T_1'$ taking $c_1$ to $c_1'$ and fixing each remaining mutual boundary component, such that 
\[
f_2^*(\psi\mid_{S \setminus T_1'}) = \phi\mid_{S \setminus T_1.}
\]
The diffeomorphisms $f_1$ and $f_2$ can be chosen in such a way as to extend to a diffeomorphism $f: S \to S$. Let $\mathcal B = \{\alpha_1, \beta_1, \dots, \alpha_g, \beta_g\}$ be a geometric symplectic basis for $S$, with $\alpha_1, \beta_1$ the same curves as above. Necessarily $\alpha_k, \beta_k$ are curves on $S \setminus T_1$ for $k \ge 2$. By construction, the spin structures $\phi$ and $f^*(\psi)$ take the same values on each element of $\mathcal B$, and $\sig(S, \phi) = \sig(S, f^*(\psi))$. It then follows from Proposition \ref{theorem:gsb} that $\phi = f^*(\psi)$ as claimed.
\end{proof}

Proposition \ref{proposition:one} has several corollaries that will be used extensively in the remainder of the paper. These play the role of a change-of-coordinates principle for surfaces in the presence of a $\Z/r\Z$-valued spin structure. The first of these was established in the second paragraph of the proof of Proposition \ref{proposition:one}. We remark that the assumption that $r$ is odd played no role in the argument.
\begin{corollary}\label{corollary:admissible}
Let $r$ be an integer and let $\phi$ be a $\Z/r\Z$-valued spin structure on a surface $S$. Let $S' \subset S$ be a subsurface of genus $h \ge 1$. Then there is some admissible curve $a \subset S'$ that is not parallel to a boundary component.
\end{corollary}

This in turn leads to another useful result that will allow us to construct curves with prescribed intersection properties and arbitrary $\phi$-values.

\begin{corollary}\label{corollary:intersections}
Let $r$ be an integer and let $\phi$ be a $\Z/r\Z$-valued spin structure on a surface $S$. Let $a, c_1, \dots, c_k$ be a collection of simple closed curves. Assume that there is some connected subsurface $T \subset S$ of positive genus disjoint from $a, c_1, \dots, c_k$, and that there is an arc $\epsilon$ connecting $a$ to $\partial T$ that is disjoint from all $c_i$. Then for $\ell \in \Z/r\Z$ arbitrary, there is a simple closed curve $a_\ell$ for which $i(a_\ell, c_i) = i(a,c_i)$ for $i = 1, \dots, k$, and for which $\phi(a_\ell) = \ell$.
\end{corollary}
\begin{proof}
By Corollary \ref{corollary:admissible}, there exists an admissible curve $b \subset T$ that is not boundary-parallel. The arc $\epsilon$ can be concatenated with an arc joining $\partial T$ to $b$; denote this extended arc by $\epsilon'$. Set $\ell_0 = \phi(a)$ (where $a$ is oriented with $\epsilon'$ lying to the left), and define $a_{\ell_0}: = a$. Define $a_{\ell_0+1} := a_{\ell_0} +_{\epsilon'} b$. By Lemma \ref{lemma:curvearcsum}, $\phi(a_{\ell_0+1}) = \phi(a_{\ell_0}) + 1 = \ell_0 + 1$. 

To see that $i(a_{\ell_0+1}, c_i) = i(a, c_i)$, we appeal to the {\em bigon criterion} of \cite[Proposition 1.7]{FM}. Choose representative curves for $a, c_1, \dots, c_k$, pairwise in minimal position. The bigon criterion asserts that $a, c_i$ are in minimal position if and only if the configuration $a \cup c_i$ does not bound any {\em bigons}, i.e. an embedded disk whose boundary is the union of an arc of $a$ and an arc of $c_i$ meeting in exactly two points. The curve-arc sum $a_{\ell_0 +1}$ meets each $c_i$ in exactly the same set of points as $a_{\ell_0}$. To conclude, it thus suffices to see that no bigons were introduced by the summing procedure. The only arc of $a_{\ell_0+1}$ that is not also an arc of $a_{\ell_0}$ is the one along which the summing procedure is performed; denote the original arc of $\alpha_{\ell_0}$ by $\alpha$ and the modified arc by $\alpha'$. Suppose that there is an arc $\gamma$ of $c_i$ such that $\alpha' \cup \gamma$ bounds a bigon. As $\alpha' = \alpha+_{\epsilon'} b$, it must be the case that the curve $\alpha \cup \gamma$ is isotopic to $b$. But by assumption, $b \subset T$ is not boundary-parallel, so this cannot happen.

To construct $a_\ell$ for $\ell \in \Z/r\Z$ arbitrary, one simply repeats the above construction, producing, for any $t \ge 0$, a curve $a_{\ell_0+t}$ with the same intersection properties as $a_{\ell_0}$ and satisfying $\phi(a_{\ell_0 + t}) = \ell_0 + t$.
\end{proof}

For the remaining corollaries of Proposition \ref{proposition:one}, we re-instate the requirement that $r$ be odd.

\begin{corollary}\label{corollary:allvalues}
Let $r$ be an odd integer and let $\phi$ be a $\Z/r\Z$-valued spin structure on a surface $S$. Let $S' \subset S$ be a subsurface of genus $h \ge 1$, and suppose that if $h = 1$, then $S'$ includes some boundary component of signature $k$ such that $k+1$ generates $\Z/r \Z$.
\begin{enumerate}
\item\label{item:single} For all $x \in \Z /r \Z$, there exists some nonseparating curve $c$ supported on $S'$ satisfying $\phi(c) = x$,
\item\label{item:oddgsb} For any $2h$-tuple $(i_1, j_1, \dots, i_h, j_h)$ of elements of $\Z / r \Z$, there is some geometric symplectic basis $\mathcal B = \{a_1, b_1, \dots, a_h, b_h\}$ for $S'$ with $\phi(a_\ell) = i_\ell$ and $\phi(b_\ell) = j_\ell$ for all $1 \le \ell \le h$,
\item\label{item:oddchain} For any $2h$-tuple $(k_1, \dots, k_{2h})$ of elements of $\Z/r\Z$, there is some chain $(a_1, \dots, a_{2h})$ of curves on $S'$ such that $\phi(a_\ell) = k_\ell$ for all $1 \le \ell \le 2h$.
%
\end{enumerate}
\end{corollary}

\begin{proof}
Certainly \eqref{item:single} follows from \eqref{item:oddgsb}. To establish \eqref{item:oddgsb}, choose any geometric symplectic basis $\mathcal B = \{a_\ell', b_\ell'\}$ on $S'$. There is {\em some} spin structure $\psi$ on $S'$ for which $\psi(a_\ell') = i_\ell$ and $\psi(b_\ell') = j_\ell$. By Proposition \ref{proposition:one}, there is a diffeomorphism $f$ of $S'$ such that $f^*(\psi) = \phi$. Then $\mathcal B = f^{-1}(\mathcal B')$ has the required properties.

We will deduce \eqref{item:oddchain} from \eqref{item:oddgsb}. Given the $2h$-tuple $(k_1, \dots, k_{2h})$, define a $2h$-tuple $(i_1, j_1, \dots, i_h, j_h)$ as follows: set $i_\ell =1 -\ell + \sum_{t = 1}^\ell k_{2t - 1}$, and set $j_\ell = k_{2\ell}$. By \eqref{item:oddgsb}, there exists a geometric symplectic basis $\mathcal B = \{c_\ell, d_\ell\}$ on $S'$ whose $\phi$-values realize the tuple $(i_1, j_1, \dots, i_h, j_h)$. Any geometric symplectic basis can be ``completed'' into a chain as follows: for $\ell = 1, \dots, h-1$, let $f_\ell$ be a simple closed curve satisfying $i(f_\ell, d_{\ell}) = i(f_\ell, d_{\ell+1}) = 1$ and $i(f_\ell,x) = 0$ for all other elements $x \in \mathcal B$.  As $\mathcal B$ is a geometric symplectic basis, this imposes the homological relation $[f_\ell] = [c_{\ell+1}]-[c_{\ell}]$, and the intersection conditions imposed on the set of curves $\{f_\ell\}$ imply that this homology is realized geometrically: $c_\ell \cup f_\ell \cup c_{\ell+1}$ must bound a pair of pants $P_\ell$ for each $\ell = 1, \dots, h-1$. The orientations can be arranged so that $P_\ell$ lies to the right of $c_\ell$ and $f_\ell$ and to the left of $c_{\ell +1}$.

Applying the homological coherence property to each $P_\ell$, it follows that $\phi(f_\ell) = k_{2 \ell + 1}$. By construction, the curves $c_1, d_1, f_1, d_2, f_2, d_3,\dots, f_{h-1}, d_h$ form a chain of length $2h$; denote this chain by $C$. By construction, $\phi(c_1) = i_1 = k_1$, and $\phi(d_\ell) = k_{2 \ell}$. Altogether, this shows that $C$ has the required properties.
\end{proof}

\subsection{Even \boldmath${r}$} Following the discussion in Section \ref{subsection:classical}, we see that the Arf invariant distinguishes at least two orbits of $\Mod(\Sigma_g)$ on the set of $\Z/r\Z$-valued spin structures. To see that there are {\em exactly} two orbits, in Definition \ref{definition:modelspin} we formulate two ``model'' $\Z/r\Z$-valued spin structures $\phi^{\mathcal B}_{even}$ and $\phi^{\mathcal B}_{odd}$ of prescribed Arf invariant, and in Proposition \ref{proposition:onlyarf} we show that every $\Z/r\Z$-valued spin structure is equivalent to one of $\phi^{\mathcal B}_{even}$ or $\phi^{\mathcal B}_{odd}$. We restrict attention here to the case where the surface $S$ has at most one boundary component. The general setting of multiple boundary components introduces considerable subtlety owing to the failure for the intersection pairing to determine a symplectic form, and our results require only the case of at most one boundary component.

\begin{definition}\label{definition:modelspin}
Let $S$ be a surface of genus $g \ge 1$ with at most one boundary component. Fix a geometric symplectic basis $\mathcal B = \{\alpha_1, \beta_1, \dots, \alpha_g, \beta_g\}$. Define $\phi_{even}^{\mathcal B}$ and $\phi_{odd}^{\mathcal B}$ as the $\Z/r\Z$-valued spin structures such that $\phi_{even}^{\mathcal B}(\gamma) = \phi_{odd}^{\mathcal B}(\gamma)= 0$ for all $\gamma \in \mathcal B \setminus \{\beta_g\}$, and where $\phi_{even}^{\mathcal B}(\beta_g)$ and $\phi_{odd}^{\mathcal B}(\beta_g)$ are chosen to be $0$ or $1$ as necessary so that $\Arf(\phi_{even}^{\mathcal B}) = 0$ and $\Arf(\phi_{odd}^{\mathcal B}) = 1$.
\end{definition}

In spite of the evident dependence on geometric symplectic basis, as $\mathcal B$ ranges over the set of all geometric symplectic bases, the elements $\phi^{\mathcal B}_{odd}$ lie in a single orbit of $\Mod(S)$ (and the same is also true of $\phi_{even}^{\mathcal B}$). The following is immediate via the change-of-coordinates principle.

\begin{lemma}\label{lemma:evenoddorbit}
Let $\mathcal B$ and $\mathcal B'$ be geometric symplectic bases. Then there is a diffeomorphism $f: S \to S$ such that $f(\mathcal B) = \mathcal B'$. Consequently, $f^*(\phi_{even}^{\mathcal B'}) = \phi_{even}^{\mathcal B}$ and $f^*(\phi_{odd}^{\mathcal B'}) = \phi_{odd}^{\mathcal B}$.
\end{lemma}

\begin{definition}\label{definition:evenodd}
Let $S$ be a surface of genus $g \ge 1$ with at most one boundary component endowed with a $\Z/r\Z$-valued spin structure $\phi$. We say that $\phi$ is {\em even} if there is a geometric symplectic basis $\mathcal B$ such that $\phi = \phi_{even}^{\mathcal B}$, and we say that $\phi$ is {\em odd} if $\phi = \phi_{odd}^{\mathcal B}$. 
\end{definition}

\begin{proposition}\label{proposition:onlyarf}
Fix an even integer $r$. Let $S$ be a surface of genus $g \ge 2$ with at most one boundary component. Let $\phi$ be a $\Z/r \Z$-valued spin structure on $S$. Then in the sense of Definition \ref{definition:evenodd}, either $\phi$ is even, or else $\phi$ is odd.
\end{proposition}

\begin{proof}
The argument makes use of the techniques of the proof of Proposition \ref{proposition:one}. Let $\mathcal B = \{\alpha_1, \beta_1, \dots, \alpha_g, \beta_g\}$ be a geometric symplectic basis, and let $S_i$ denote the genus-$1$ subsurface determined by $\alpha_i, \beta_i$; define $D_i$ as the boundary curve of $S_i$. Exactly as in Proposition \ref{proposition:one}, each pair $\alpha_i, \beta_i$ can be replaced by new curves $\alpha_i', \beta_i'$ supported on $S_i$ and satisfying $i(\alpha_i', \beta_i') = 1$, such that $\alpha_i'$ is admissible. Denote the corresponding geometric symplectic basis by $\mathcal B'$. For an arc $\epsilon$ connecting $\beta_1'$ to $D_2$ and disjoint from all other $D_i$, the curve-arc sum $\beta_1' +_\epsilon D_2$ satisfies $\phi(\beta_1'+_\epsilon D_2) = \phi(\beta_1') + 2$. By repeatedly performing this curve-arc sum using an arc $\epsilon$ disjoint from $\mathcal B' \setminus \{\beta_2'\}$ (as in Proposition \ref{proposition:one}), $\beta_2'$ can be replaced with a curve $\beta_2''$ such that $\phi(\beta_2'') = 0 \mbox{ or }1$. By performing an analogous operation on all $\beta_i'$, one obtains a geometric symplectic basis $\mathcal B'' = \{\alpha_1', \beta_1'', \dots, \alpha_g', \beta_g''\}$ such that $\phi(\alpha_i') = 0$ and $\phi(\beta_i'') = 0 \mbox{ or }1$. 

It remains to further alter each $\beta_1'', \dots, \beta_{g-1}''$ so that $\phi(\beta_i'') = 0$ in this range. For $1 \le i \le g-1$, let $\gamma_i$ be a collection of disjoint curves such that $\beta_1, \gamma_1, \dots, \beta_{g-1}, \gamma_{g-1}, \beta_g$ forms a chain of length $2g-1$, and such that each $\gamma_i$ is disjoint from all $\alpha_j'$. Then necessarily $\alpha_i, \gamma_i, \alpha_{i+1}$ forms a pair of pants, and so $\phi(\gamma_i) = -1$. If $\phi(\beta_1'') = 1$, then $\phi(T_{\gamma_1}(\beta_1'')) = 0$. Replace $\beta_1'', \beta_2''$ by $T_{\gamma_1}(\beta_1''), T_{\gamma_1}(\beta_2'')$, respectively. Repeat, applying $T_{\gamma_2}^k$ to $T_{\gamma_1}(\mathcal B'')$ for $k$ such that $\phi(T_{\gamma_2}^k T_{\gamma_1}(\beta_2'')) = 0$. Proceed in this way, taking each $\beta_i''$ for $i \le g-1$ to some $\beta_i'''$ with $\phi(\beta_i''') = 0$. At the end, the geometric symplectic basis $\mathcal B''' = \{\alpha_1', \beta_1''', \dots, \alpha_g', \beta_g'''\}$ will satisfy $\phi(\gamma) = 0$ for all $\gamma \in \mathcal B'''$ except possibly $\gamma = \beta_g'''$. By repeating the curve-arc summing procedure, $\beta_g'''$ can be altered to satisfy $\psi(\beta_g''') = 0 \mbox{ or }1$ as required. Define $\widetilde{\mathcal B}$ to be this geometric symplectic basis. Applying Theorem \ref{theorem:gsb}, we see that $\phi = \phi_{even}^{\widetilde{\mathcal B}} \mbox{ or }\phi_{odd}^{\widetilde{\mathcal B}}$ as required.
\end{proof}

There is an analogue of Corollary \ref{corollary:allvalues} for $r$ even, although the Arf invariant provides an obstruction that was not present in the case of odd $r$.
\begin{corollary}\label{corollary:allvalueseven}
Let $r$ be an even integer, and let $S' \subset S$ be a subsurface of genus $h \ge 2$ endowed with a $\Z/r\Z$-valued spin structure $\phi$. Then the following assertions hold:
\begin{enumerate}
\item\label{item:singleeven} For all $x \in \Z /r \Z$, there exists some nonseparating curve $c$ supported on $S'$ satisfying $\phi(c) = x$.
\item\label{item:gsbeven} For a given $2h$-tuple $(i_1, j_1, \dots, i_h, j_h)$ of elements of $\Z / r \Z$, there is some geometric symplectic basis $\mathcal B = \{a_1, b_1, \dots, a_h, b_h\}$ for $S'$ with $\phi(a_\ell) = i_\ell$ and $\phi(b_\ell) = j_\ell$ for $1 \le \ell \le h$ if and only if the parity of the spin structure defined by these conditions agrees with the parity of the restriction $\phi\vline_{S'}$ to $S'$. 
\item\label{item:gsbevenrestricted} For {\em any} $(2h-2)$-tuple $(i_1, j_1, \dots,i_{h-1}, j_{h-1})$ of elements of $\Z / r \Z$, there is some geometric symplectic basis $\mathcal B = \{a_1, b_1, \dots, a_h, b_h\}$ for $S'$ with $\phi(a_\ell) = i_\ell$ and $\phi(b_\ell) = j_\ell$ for $1 \le \ell \le h-1$.
\item \label{item:chaineven} For a given $2h$-tuple $(k_1, \dots, k_{2h})$ of elements of $\Z/r\Z$, there is some chain $(a_1, \dots, a_{2h})$ of curves on $S'$ such that $\phi(a_\ell) = k_\ell$ for all $1 \le \ell \le 2h$ if and only if the parity of the spin structure defined by these conditions agrees with the parity of the restriction $\phi\vline_{S'}$ to $S'$. 
\item \label{item:chainevenrestricted} For {\em any} $(2h-2)$-tuple $(k_1, \dots, k_{2h-2})$ of elements of $\Z/r\Z$, there is some chain $(a_1, \dots, a_{2h-2})$ of curves on $S'$ such that $\phi(a_\ell) = k_\ell$ for all $1 \le \ell \le 2h-2$.
\end{enumerate}
\end{corollary}
\begin{proof}
The proof is essentially identical to that of Corollary \ref{corollary:allvalues}. The arguments for \eqref{item:gsbeven} and \eqref{item:gsbevenrestricted} are slightly novel; the remaining points follow their counterparts in Corollary \ref{corollary:allvalues} verbatim. To establish \eqref{item:gsbeven}, let $\mathcal B' = \{a_\ell', b_\ell'\}$ be a geometric symplectic basis on $S'$. Let $S''$ be a subsurface of $S'$ containing each curve in $\mathcal B'$ that has only one boundary component. Given $(i_1, j_1, \dots, i_h, j_h)$, there is some spin structure $\psi$ on $S''$ for which $\psi(a'_\ell) = i_\ell$ and $\psi(b'_\ell) = j_\ell$ for $1 \le \ell \le h$. By Proposition \ref{proposition:onlyarf}, there is an element $f \in \Mod(S'')$ for which $f^*(\psi) = \phi$ if and only if the Arf invariants of $\phi$ and $\psi$ agree; if they do, then $\mathcal B = f^{-1}(\mathcal B')$ has the required properties.

\eqref{item:gsbevenrestricted} will be obtained from \eqref{item:gsbeven}. Let $\epsilon \in \Z/2\Z$ denote the Arf invariant of $\phi$, and define the quantity
\[
\eta = \sum_{\ell = 1}^{h-1} (i_\ell + 1)(j_\ell +1) \pmod 2.
\]

As the formula \eqref{equation:arf} for the Arf invariant shows, given any $(2h-2)$-tuple $(i_1, j_i, \dots,i_{h-1}, j_{h-1})$ and any value $\epsilon \in \Z/2\Z$, there is a choice of $i_h, j_h \in \Z/r\Z$ for which $\eta + (i_h +1)(j_h +1) \equiv \epsilon \pmod 2$. The result now follows by applying \eqref{item:gsbeven} to the tuple $(i_1, j_1, \dots, i_h, j_h)$. 
\end{proof}

We will also require a result establishing the existence of configurations $\mathscr D_n$ as in the $D_n$ relation (Proposition \ref{proposition:dnrel}).
\begin{corollary}\label{corollary:dneven}
Let $r = 2d$ be an even integer, and let $\Sigma_g$ be a closed surface endowed with a $\Z/2d\Z$-valued spin structure $\phi$. Let $\Delta$ be a curve on $\Sigma_g$ that separates $\Sigma_g$ into subsurfaces $S_1, S_2$ for which the genus $g(S_1) \ge d+1$. Set $n = 2g(S_1) - 2d + 1$. Then there exists a configuration $a, a', c_1, \dots, c_{n-2}$ of curves on $S_1$ arranged in the $\mathscr D_n$ configuration, such that the elements $a, a',$ and $c_i$ are admissible for all $i$, and such that $\Delta = \Delta_2$ as in Figure \ref{figure:dnrel}. 
\end{corollary}
\begin{proof}
By Corollary \ref{corollary:allvalueseven}.\ref{item:chainevenrestricted}, there exists a chain $a, c_1, \dots, c_{n-2}$ of admissible curves on $S_1$. Let $a' \subset S_1$ be chosen so that $a \cup a' \cup \Delta$ bounds a subsurface of genus $g(S_1) - d - 1$ containing $c_i$ for $i \ge 2$, and such that $i(a',c_1) = 1$. The other side of $a \cup a'$ bounds a subsurface of genus $d$, and so the homological coherence property implies that $a'$ is admissible. By construction, the curves $a, a', c_1, \dots, c_{n-2}$ form the configuration $\mathscr D_{n}$ of the $D_n$ relation, and the boundary component $\Delta_2$ of Figure \ref{figure:dnrel} is given here by $\Delta$.
\end{proof}

\section{$r$ odd: generating $\Mod(\Sigma_g)[\phi]$ by Dehn twists}\label{section:dodd}

Let $\phi$ be a $\Z/r \Z$-valued spin structure on a closed surface $\Sigma_g$. Throughout this section we assume that $r \mid (2g-2)$ (so that, following Remark \ref{remark:admits}, $\Sigma_g$ admits a $\Z/r \Z$-valued spin structure) and that $r$ is {\em odd}. Recall from Definition \ref{definition:admissible} that the admissible subgroup is defined via
\[
\mathcal T_\phi = \pair{T_a \mid a \mbox{ nonseparating curve, } \phi(a) = 0}.
\]
By construction, $\mathcal T_\phi \leqslant \Mod(\Sigma_g)[\phi]$. The main result of this section is that for $r$ odd, this containment is an equality. 

\begin{proposition}\label{proposition:twistgen}
For any $g \ge 3$ and for any odd integer $r$ satisfying $r < g - 1$, there is an equality 
\[
\mathcal T_\phi = \Mod(\Sigma_g)[\phi].
\]
\end{proposition}

Before beginning with the proof, we will first establish some properties of the group $\mathcal T_\phi$ which will be used throughout this section and the next.

\begin{lemma}\label{item:allcurves} Let $\phi$ be a $\Z/r\Z$-valued spin structure on a surface $\Sigma_g$ with $r < g-1$ and $g \ge 5$. Let $c$ be any nonseparating simple closed curve on $\Sigma_g$. Suppose that $r$ is odd, or else that $r$ is even and $\phi(c) \equiv 1 \pmod 2$. Then $T_c^r \in \mathcal T_\phi$. 
\end{lemma}
\begin{proof}
Let $c$ be as in the statement of Lemma \ref{item:allcurves}. Our first objective is to construct a configuration of admissible curves $\mathscr D_{2r+3}$ as in Corollary \ref{corollary:dn} for which $c = C_k$. By hypothesis, there is an expression of the form $\phi(c) = 2k-1 \pmod r$ for some integer $1 \le k \le r$. Invoking Corollary \ref{corollary:allvalues}.\ref{item:oddchain} or \ref{corollary:allvalueseven}.\ref{item:chainevenrestricted} as appropriate, the hypothesis $r< g-1$ implies that there is a chain $a, c_1, \dots, c_{2k-1}$ of admissible curves disjoint from $c$, and there is a chain $c_{2k+1}, \dots, c_{2r+1}$ of admissible curves disjoint from $c$ and from $a, c_1, \dots, c_{2k-1}$. Let $a'$ be a curve such that $a \cup a' \cup c$ bounds a surface of genus $k-1$ containing $c_2, \dots, c_{2k-1}$, and satisfying $i(a', c_1) = 1$ and $i(a', c_i) = 0$ for $2k+1 \le i \le 2r+1$. The homological coherence property implies that $a'$ is admissible.  

To complete the construction, it remains only to find the curve $c_{2k}$. Such a curve $c_{2k}$ must be admissible, and $c_{2k}$ must have the following intersection properties: 
\begin{equation}\label{equation:c2k}
i(c_{2k}, c_{2k \pm 1}) = 1, \quad i(c_{2k}, a) = i(c_{2k}, a') = i(c_{2k}, c_i) = 0 \mbox{ for }\abs{i-2k} > 1, \quad i(c_{2k}, c) = 2.
\end{equation}
Let $c_{2k}'$ be any curve satisfying the intersection properties \eqref{equation:c2k}. If we can show that the complement of a regular neighborhood of the configuration $\mathscr D'_{2r+3} := a, a', c_1, \dots, c_{2k-1}, c_{2k'}, c_{2k+1}, \dots, c_{2r+1}$ is a surface of positive genus, then the existence of $c_{2k}$ will follow from Corollary \ref{corollary:intersections}.

The configuration $\mathscr D'_{2r+3}$ is contained in a surface of genus $r+1$ with two boundary components. Each boundary component is homologous to the nonseparating curve $c$, so the complement has genus $g - r - 2$. We must show that this quantity is positive. Establishing $g -r - 2 \ge 1$ is a matter of simple arithmetic. Writing $r = \frac{2g-2}{m}$ for some $m \ge 3$, we have 
\[
g -r - 2=  \frac{m-2}{m}(g-1) -1 \ge \frac{g}{3}-1>0,
\]
since $g \ge 5$ by hypothesis.

Recalling that the group $H_{2r + 3}^+$ from Corollary \ref{corollary:dn} is defined to be the group generated by the Dehn twists about the elements of $\mathscr D_{2 r + 3} \cup \{\Delta_1\}$, it follows that if each element of $\mathscr D_{2 r +3}$ is admissible, then $H_{2r+3}^+ \leqslant \mathcal T_\phi$. We have constructed the curves $a, a', c_1, \dots, c_{2r+1}$ so as to be admissible; homological coherence implies that also $\Delta_1$ is admissible. Corollary \ref{corollary:dn}.2 then implies that $T_{C_k}^r \in \mathcal T_\phi$ for any $1 \le k \le r+1$.
\end{proof}

\begin{lemma}\label{item:homology}
Let $\phi$ be a $\Z/r\Z$-valued spin structure on a surface $\Sigma_g$, and let $v \in H_1(\Sigma_g; \Z)$ be any primitive homology class. If $r$ is odd, then for any $k \in \Z/r \Z$, there exists a curve $c$ for which $[c] = v$ and $\phi(c) = k$. If $r$ is even, then for any $k \in \Z/r \Z$ such that $\phi \pmod 2(v) \equiv k \pmod 2$, there exists a curve $c$ for which $[c] = v$ and $\phi(c) = k$. 
\end{lemma}

\begin{proof}
Let $c_0$ be any (oriented) curve on $\Sigma_g$ with $[c_0] = v$; set $\phi(c_0) = k_0$. Let $c_1$ be a curve disjoint from $c_0$ such that $c_0 \cup c_1$ bounds a subsurface of genus $1$, oriented to the left of $c_0$. Then $\phi(c_1) = k_0+2$ when oriented with the subsurface to the right, and $[c_0] = [c_1]$. This construction can be repeated, giving rise to curves $c_m$ with $\phi(c_m) = k_0+2m$. If $r$ is odd, then the set of values $k_0 + 2m$ for various values of $m$ exhausts $\Z/r\Z$, and if $r$ is even, then the set of values $k_0 + 2m$ exhausts the coset $k_0 + 2 \Z/r\Z$. The claim follows by taking $c = c_m$ for the appropriate value of $m$.
\end{proof}

\begin{proof} {\em (of Proposition \ref{proposition:twistgen})}
The method is to compare the intersections of $\mathcal T_\phi$ and $\Mod(\Sigma_g)[\phi]$ with $\mathcal I_g$ and $\mathcal K_g$. We first present a high-level overview of the logical structure of the proof that explains how Proposition \ref{proposition:twistgen} follows from ancillary results; these results are then obtained in Steps 1--4.

\para{Overview} Recall from \eqref{equation:sprep} the symplectic representation $\Psi: \Mod(\Sigma_g) \to \Sp(2g, \Z)$ with kernel given by the Torelli group $\mathcal I_g$. To show that $\mathcal T_\phi = \Mod(\Sigma_g)[\phi]$, it suffices to show that (I) $\Psi(\mathcal T_\phi) = \Psi(\Mod(\Sigma_g)[\phi])$ and that (II) $\mathcal T_\phi \cap \mathcal I_g = \Mod(\Sigma_g)[\phi] \cap \mathcal I_g$. 

The equality of (I) is obtained in Step 1 as Lemma \ref{lemma:spimage}. The proof of (II) is carried out in Steps 2-4. The method is to study the restriction of the Johnson homomorphism to the groups $\mathcal T_\phi \cap \mathcal I_g$ and $\Mod(\Sigma_g)[\phi] \cap \mathcal I_g$. Recall from \eqref{equation:johnsonhom} that the Johnson homomorphism is the surjective homomorphism
\[
\tau: \mathcal I_g \to \wedge^3 H_\Z/H_\Z,
\]
and that the kernel is written $\mathcal K_g$. To establish (II), it suffices to show that (i) $\tau(\mathcal T_\phi \cap \mathcal I_g) = \tau(\Mod(\Sigma_g)[\phi] \cap \mathcal I_g)$ and that (ii) $\mathcal T_\phi \cap \mathcal K_g = \Mod(\Sigma_g)[\phi] \cap \mathcal K_g$. The equality of (i) is carried in Steps 2 and 3. The main result of Step 2, Lemma \ref{lemma:tauimage}, establishes an upper bound on the image $\tau(\Mod(\Sigma_g)[\phi] \cap \mathcal I_g)$, and the main result of Step 3, Lemma \ref{lemma:converse}, shows that the subgroup $\tau(\mathcal T_\phi \cap \mathcal I_g)$ realizes this upper bound. Finally (ii) is established in Step 4: Lemma \ref{lemma:containsjk} shows that there is a containment $\mathcal K_g \leqslant \mathcal T_\phi$.

\para{Step 1: The symplectic quotient} The first step is to understand the image of $\mathcal T_\phi$ and $\Mod(\Sigma_g)[\phi]$ in the symplectic group $\Sp(2g,\Z)$. 

\begin{lemma}\label{lemma:spimage}
For $r$ odd, the symplectic representation $\Psi: \Mod(\Sigma_g) \to \Sp(2g,\Z)$ restricts to a {\em surjection}
\[
\Psi: \mathcal T_{\phi} \twoheadrightarrow \Sp(2g,\Z).
\]
It follows that also $\Psi: \Mod(\Sigma_g)[\phi] \twoheadrightarrow \Sp(2g,\Z)$ is a surjection.
\end{lemma}

\begin{proof}
Let $v \in H_1(\Sigma_g; \Z)$ be a primitive element. By Lemma \ref{item:homology}, there is some curve $c$ with $[c] = v$ and $\phi(c) = 0$. The result follows from this, since $\Sp(2g, \Z)$ is generated by the set of transvections $T_v$ given by $x \mapsto x + \pair{x,v}v$ for $v \in H_1(\Sigma_g; \Z)$ primitive, and $\Psi(T_c) = T_{[c]}$. 
\end{proof}

\para{Step 2: \boldmath$\Mod(\Sigma_g)[\phi]$ and the Johnson homomorphism} Our next objective is Lemma \ref{lemma:tauimage} below. This concerns the image of $\Mod(\Sigma_g)[\phi] \cap \mathcal I_g$ under the Johnson homomorphism. In order to formulate the result, it is necessary to first study a different quotient of $\mathcal I_g$ first constructed by Chillingworth in \cite{chillingworth} and \cite{chillingworth2}. Chillingworth's work is formulated using the notion of a ``winding number function''; as explained in \cite[Introduction]{HJ}, a winding number function is a particular instance of a spin structure. The properties of a winding number function that Chillingworth exploits in his work are common to all spin structures, and so we formulate his results in this larger context. See also \cite[Section 5]{johnsonhom} for a brief summary of Chillingworth's work. Recall in the statement below that $\mathcal S$ is defined to be the set of isotopy classes of oriented simple closed curves on a surface $\Sigma_g$.

\begin{theorem}[Chillingworth]\label{theorem:chillingworth}
Let $\phi$ be a $\Z/r\Z$-valued spin structure on a closed surface $\Sigma_g$. Let $c$ be the function $c: \mathcal I_g \times \mathcal S \to \Z/r \Z$ defined by the formula
\[
c(f, \gamma) = \phi(f(\gamma)) - \phi(\gamma).
\]
Then the value $c(f, \gamma)$ depends only on the homology class $[\gamma] \in H_\Z$, and $c$ descends to a homomorphism
\[
c: \mathcal I_g \to \Hom(H_\Z, \Z/r \Z) \cong H^1(\Sigma_g; \Z/r \Z).
\] 
In particular, $c$ does not depend on the choice of $\Z/r \Z$-valued spin structure.
\end{theorem}

In \cite{johnsonhom}, Johnson related Chillingworth's homomorphism to the Johnson homomorphism. To formulate the precise connection, we require the following well-known lemma; see e.g \cite[Sections 5,6]{johnsonhom}.
\begin{lemma}\label{lemma:contraction}
There is a $\Sp(2g,\Z)$-equivariant surjection
\[
C: \wedge^3 H_{\Z} / H_{\Z} \to H_{\Z/(g-1)\Z}
\]
given by the ``contraction''
\begin{equation}\label{equation:contract}
C(x\wedge y \wedge z) = \pair{x,y}z + \pair{y,z} x + \pair{z,x} y \pmod{g-1}.
\end{equation}
\end{lemma}

 It follows that for any $r \mid (g-1)$, there is a $\Sp(2g,\Z)$-equivariant surjection 
 \[
 C_r: \wedge^3 H_{\Z} / H_{\Z} \to H_{\Z/r\Z}
 \]
 given by post-composing $C$ with the reduction mod $r$. We can now formulate the main result of Step 3.

\begin{lemma}\label{lemma:tauimage}
Let $\phi$ be a $\Z/r\Z$-valued spin structure on a surface of genus $g$, with $g \ge 3$ and $r$ odd. Then $C_r \circ \tau = 0$ on $\Mod(\Sigma_g)[\phi] \cap \mathcal I_g$.
\end{lemma}

\begin{proof}
According to \cite[Theorem 3]{johnsonhom}, the composition $C_r \circ \tau$ coincides (up to an application of Poincar\'e duality) with the mod-$r$ Chillingworth homomorphism $c: \mathcal I_g \to H^1(\Sigma_g;{\Z/r \Z})$. The formula for $c$ given in Theorem \ref{theorem:chillingworth} shows that $c$ measures how $f \in \mathcal I_g$ alters the set of values $\{\phi(\gamma)\mid \gamma \in \mathcal S\}$; it therefore follows immediately that the restriction of $c$ to $\Mod(\Sigma_g)[\phi] \cap \mathcal I_g$ is trivial.
\end{proof}

\para{Step 3: \boldmath$\mathcal T_\phi$ and the Johnson homomorphism} In the previous step, we showed that there is a containment
\[
\tau(\Mod(\Sigma_g)[\phi] \cap \mathcal I_g) \leqslant \ker(C_r \circ \tau). 
\]
Our next result establishes that this containment is an equality, even when restricted to the subgroup $\tau(\mathcal T_\phi \cap \mathcal I_g)$.
\begin{lemma}\label{lemma:converse}
For $r< g-1$ odd and for $g \ge 3$, the Johnson homomorphism $\tau$ gives a surjection
\[
\tau: \mathcal T_\phi \cap \mathcal I_g \twoheadrightarrow \ker(C_r \circ \tau).
\]
\end{lemma}

\begin{proof} Define $K := \ker(C_r)$. We must show that $\mathcal T_\phi\cap \mathcal I_g$ surjects onto $K$ under $\tau$. The first step will be to determine a generating set for $K$, and then we will exhibit each generator within $\tau(\mathcal T_\phi \cap \mathcal I_g)$.  

To determine a generating set for $K$, we consider the short exact sequence
\[
1 \to K \to \wedge^3 H_\Z / H_\Z \to H_{\Z/r \Z} \to 1.
\]
determined by $C_r$. By lifting a set of relations $\{r_i\}$ for $H_{\Z/r\Z}$ to $\wedge^3 H_\Z/H_\Z$, we will obtain a set of generators $\{\tilde r_i\}$ for $K$. Let $\mathcal B = \{x_1, y_1, \dots, x_g, y_g\}$ be a symplectic basis for $H_\Z$. There is an associated basis $\wedge^3 \mathcal B \subset \wedge^3 H_\Z$ given by
\[
\wedge^3 \mathcal B := \{z_1 \wedge z_2 \wedge z_3 \mid z_i \in \mathcal B \mbox{ distinct}\}
\]
Thus also $\wedge^3H_\Z /H_\Z$ is generated by the image of $\wedge^3 \mathcal B$. 

To determine the relations $r_i$, we must understand $C_r(z_1 \wedge z_2 \wedge z_3)$ for the various possibilities for $\{z_1, z_2, z_3\}$. There are two orbits of generators under the action of $\Sp(2g, \Z)$. The first orbit consists of elements of the form $z \wedge x_i \wedge y_i$ (necessarily with $z \ne x_i, y_i$), and the second orbit consists  of elements of the form $z_i \wedge z_j \wedge z_k$ with each $z_\ell \in \{x_\ell, y_\ell\}$ and with $i,j, k$ mutually distinct. 

The image of $z \wedge x_i \wedge y_i$ in $H_{\Z/r \Z}$ is 
\[
C_r(z \wedge x_i \wedge y_i) = z,
\]
while $C_r(z_i \wedge z_j \wedge z_k) = 0$ for elements of the second type. Define $A$ to be the abelian group generated by the symbols $C_r(z_1 \wedge z_2 \wedge z_3)$ for $z_1 \wedge z_2 \wedge z_3 \in \wedge^3 \mathcal B$, subject to the relations (R1)-(R3) below:
\begin{enumerate}[(R1)]
\item $r C_r(z \wedge x_i \wedge y_i) = 0$
\item $C_r(z \wedge x_i \wedge y_i) - C_r(z \wedge x_j \wedge y_j) = 0$
\item $C_r(z_i \wedge z_j \wedge z_k) = 0$ for $\{i,j,k\} \subset \{1, \dots, g\}$ distinct.
\end{enumerate}
It can be easily verified that there is an isomorphism $A \cong H_{\Z/r \Z}$, so that the relations (R1) - (R3) can be lifted to $\wedge^3 H_\Z/ H_\Z$ to give a generating set for $K$ as desired. The corresponding generators are given below.
\begin{enumerate}[(G1)]
\item\label{item:deltatwist} $r z \wedge x_i \wedge y_i$
\item\label{item:difference} $z \wedge (x_i \wedge y_i - x_j\wedge y_j)$
\item\label{item:lagrangian} $z_i \wedge z_j \wedge z_k$ for $\{i,j,k\} \subset \{1, \dots, g\}$ distinct. 
\end{enumerate}

Having determined a generating set for $K$, it remains to exhibit each such generator in the form $\tau(f)$ for $f \in \mathcal T_\phi \cap \mathcal I_g$. These will be handled on a case-by-case basis. We start with  (G\ref{item:deltatwist}). By Lemma \ref{lemma:jhom}, there exist curves $c,d$ that determine a genus-$1$ bounding pair map with $\tau(T_c T_d^{-1}) = z \wedge x_i \wedge y_i$. By Lemma \ref{item:allcurves}, $T_c^r, T_d^r \in \mathcal T_\phi$, so that $T_c^r T_d^{-r} \in \mathcal T_\phi$ is an element with the required properties.

\begin{figure}
\labellist
\small
\pinlabel $a$ [bl] at 24 102
\pinlabel $b$ [bl] at 211.2 102
\pinlabel $c$ [bl] at 120 102
\pinlabel $x_i$ [bl] at 72 73.4
\pinlabel $y_i$ [l] at 65 26.8
\pinlabel $x_j$ [bl] at 170 72.4
\pinlabel $y_j$ [l] at 161.6 26.8
\endlabellist
\includegraphics{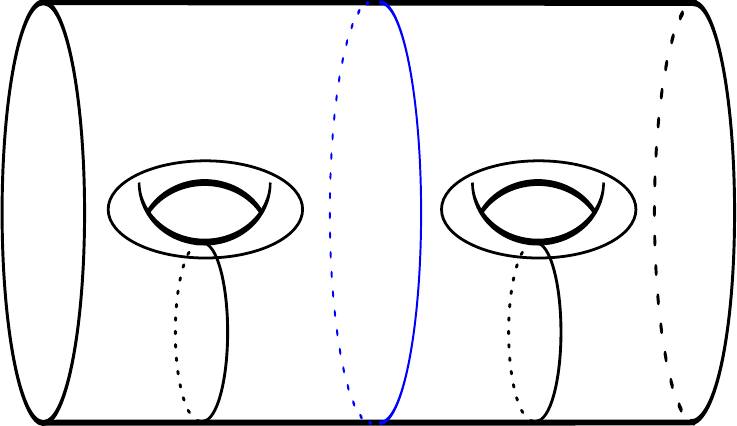}
\caption{The configuration of curves used to exhibit (G\ref{item:difference}).}
\label{figure:g2}
\end{figure}
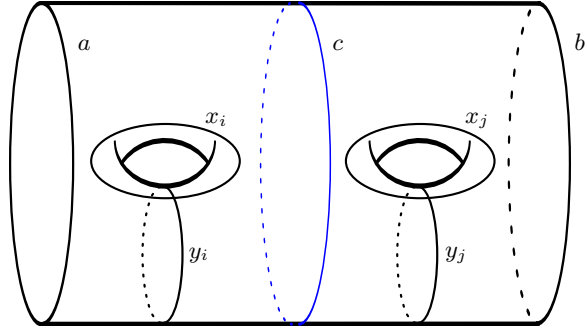

Next we consider (G\ref{item:difference}). Let $c$ be a curve with $[c] = z$ and $\phi(c) = 0$. By the change-of-coordinates principle, there exist curves $a,b$ with the following properties: (1) $a \cup b$ bounds a subsurface $S$ of genus $2$, (2) $c \subset S$, (3) $[a] = [b] = [c]$, and $c$ separates $S$ into two subsurfaces $S_1, S_2$ each of genus $1$, (4) $x_i, y_i$ determine a symplectic basis for $S_1$ and $x_j, y_j$ determine a symplectic basis for $S_2$. Such a configuration is shown in Figure \ref{figure:g2}. By homological coherence, $\phi(a) = \phi(b) = -2$ when $a,b$ are oriented with $S$ to the left. 
By Lemma \ref{lemma:jhom},
\[
\tau(T_a T_c^{-1}) = z \wedge x_i \wedge y_i
\]
and
\[
\tau(T_b T_c^{-1}) = -z \wedge x_j \wedge y_j.
\]
Therefore, it is necessary to show $T_a T_b T_c^{-2} \in \mathcal T_\phi$. By hypothesis, $T_c \in \mathcal T_\phi$. By Corollary \ref{corollary:allvalues}.\ref{item:oddchain}, there exists a chain $a_1, \dots, a_5$ of curves on $S$ for which $\phi(a_i) = 0$. By the chain relation (Proposition \ref{proposition:chain}), $T_a T_b \in \mathcal T_\phi$, and the result follows. 

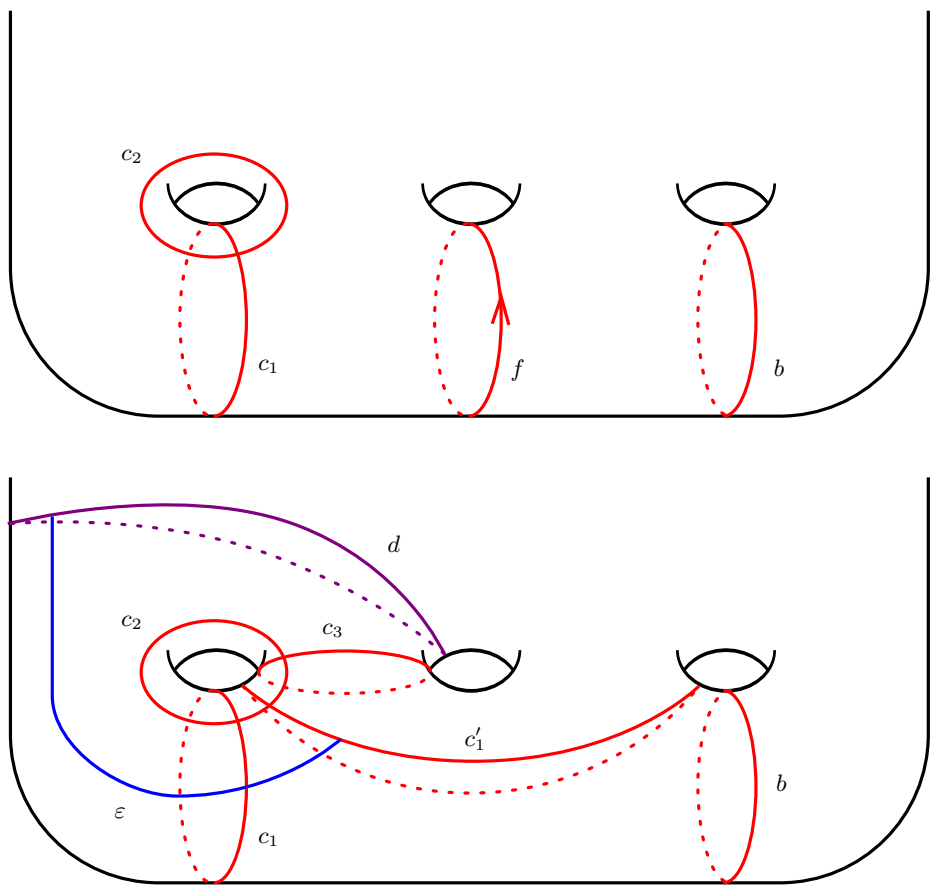
\begin{figure}
\labellist
\small
\pinlabel $f$ [tl] at 188 200
\pinlabel $b$ [tl] at 287.2 200
\pinlabel $c_2$ [br] at 53.6 270
\pinlabel $c_1$ [tl] at 92.8 200
\pinlabel $c_1$ [tl] at 92.8 21.6
\pinlabel $c_2$ [br] at 53.6 94.4
\pinlabel $c_3$ [b] at 123.2 90.4
\pinlabel $c_1'$ [t] at 176.8 62.4
\pinlabel $\epsilon$ [tr] at 47.2 31.2
\pinlabel $d$ [bl] at 141.6 123.2
\pinlabel $b$ [l] at 287.2 38.4
\endlabellist
\includegraphics{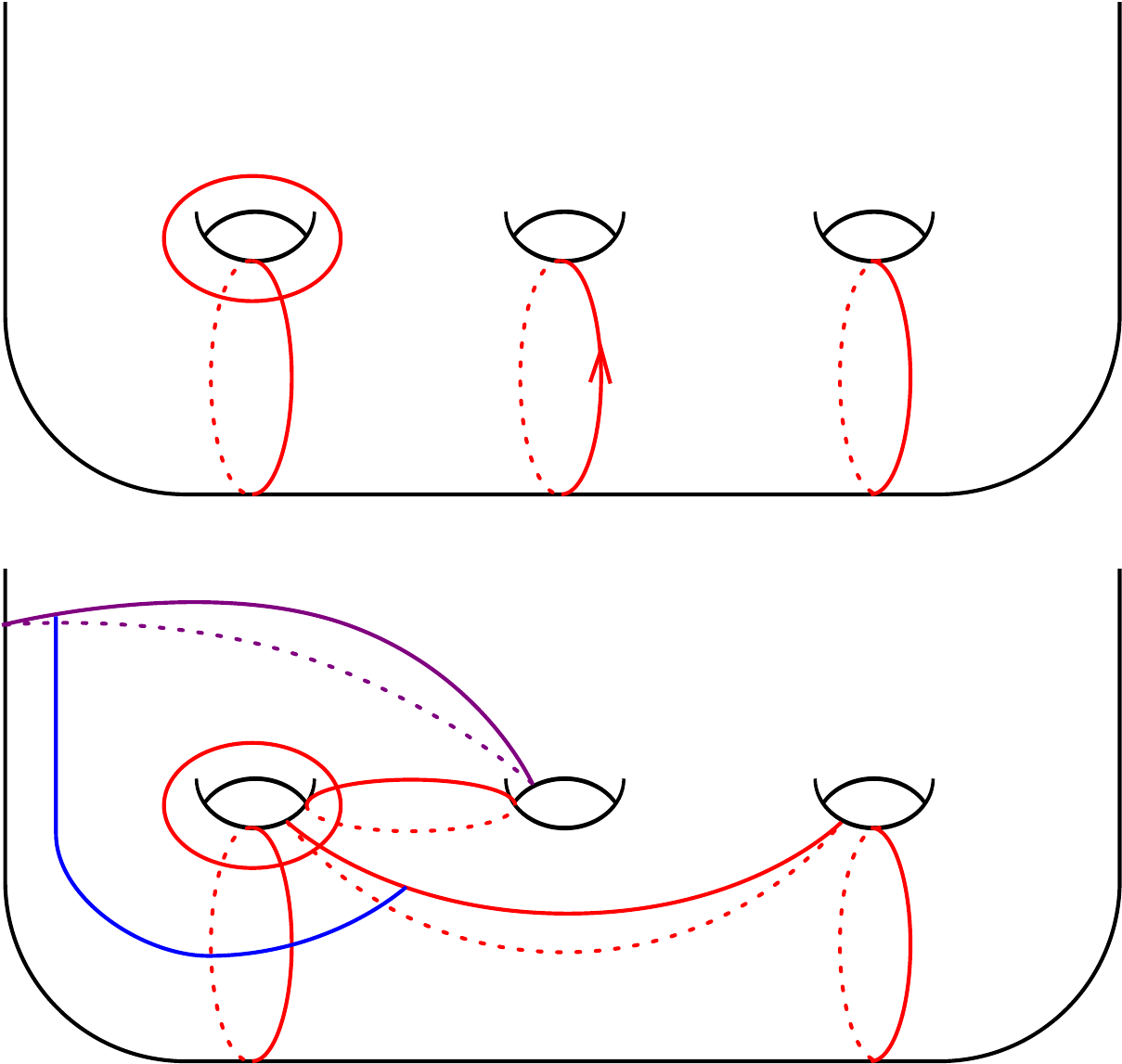}
\caption{Top: The relevant portion of the geometric symplectic basis $\mathcal B$. Bottom: The configuration of curves used to exhibit (G\ref{item:lagrangian}). Orientations have been suppressed wherever possible.}
\label{figure:g3}
\end{figure}

It remains to exhibit generators of type (G\ref{item:lagrangian}). Any such generator is equivalent under the action of $\Sp(2g, \Z)$ to $y_1 \wedge y_2 \wedge y_3$. By combining the $\Sp(2g,\Z)$-equivariance of $\tau$ (Lemma \ref{lemma:jhom}.1) with the result of Lemma \ref{lemma:spimage}, it suffices to exhibit only $y_1 \wedge y_2 \wedge y_3$. Figure \ref{figure:g3} shows the two $3$-chains $C_1 = (c_1,c_2,c_3)$ and $C_2 = (c_1',c_2,c_1'+_\epsilon d)$. Observe that $d$ is a boundary component for regular neighborhoods of both $C_1$ and $C_2$; let $e_1, e_2$ denote the other boundary component of $C_1, C_2$, respectively. 

By Corollary \ref{corollary:allvalues}.\ref{item:oddgsb}, there exists a geometric symplectic basis $\mathcal B$ that contains the elements $c_1, c_2, b, f$ as depicted in the top portion of Figure \ref{figure:g3}, with homology classes and $\phi$-values given in the table below.  The remaining entries in the table have been filled in using the homological coherence property. (A value of $*$ indicates that the value is irrelevant and/or underdetermined, and if an orientation is left unspecified, this is in accordance with Convention \ref{convention:orientation}).
\[
\begin{array}{|c|cccccccc|}
\hline
\mbox{Curve:}			& c_1	& c_2	& c_3		& c_1'		& b		& f		& d		& c_1'+_\epsilon d \\ \hline
\mbox{Homology class:}	& x_1	&y_1		&y_2-x_1		&x_1-y_3		& y_3	& -y_2	&y_2		&y_2 + y_3-x_1		\\
\mbox{$\phi$-value:}		& 0 		& *		& 0			& *			& *		& -1		&-2		& *				\\ \hline
\end{array}
\]

By Lemma \ref{lemma:jhom}, 
\[
\tau(T_d T_{e_1}^{-1})= x_1 \wedge y_1 \wedge y_2,
\]
\[
\tau(T_d T_{e_2}^{-1}) = (x_1 - y_3) \wedge y_1 \wedge y_2.
\]
It follows that $\tau(T_{e_1}^{-1}T_{e_2}) = y_3 \wedge y_1 \wedge y_2$. As $d \cup e_1$ and $d \cup e_2$ each bound subsurfaces of genus $1$ and $\phi(d) = -2$ when $d$ is oriented with these subsurfaces to the left, the homological coherence property implies that $e_1$ and $e_2$ are admissible. The result follows.\end{proof}

\para{Step 4: The Johnson kernel} The final piece of the analysis concerns the relationship between $\Mod(\Sigma_g)[\phi]$ and the Johnson kernel $\mathcal K_g$.

\begin{lemma}\label{lemma:containsjk}
Let $\phi$ be a $\Z/r\Z$-valued spin structure with $r$ odd. If $g \ge 3$, then $\mathcal T_\phi$ contains the Johnson kernel $\mathcal K_g$. It follows that also
\[
\mathcal K_g \leqslant \Mod(\Sigma_g)[\phi].
\]
\end{lemma}

\begin{proof}
According to Johnson's Theorem \ref{theorem:johnson2}, $\mathcal K_g$ has a generating set consisting of the set of all $T_c$ for $c$ a separating curve. Each such $c$ divides $\Sigma_g$ into two subsurfaces $S, S'$, and since $g \ge 3$, without loss of generality we can assume that $g(S) > 1$. By Corollary \ref{corollary:allvalues}.\ref{item:oddchain}, there exists a chain $a_1, \dots, a_{2g(S)}$ of curves on $S$ such that $\phi(a_i) = 0$ for all $i$. By hypothesis, $T_{a_i} \in \mathcal T_\phi$ for all $i$. By the chain relation (Proposition \ref{proposition:chain}), it follows that $T_c \in \mathcal T_\phi$ as required. 
\end{proof}
This concludes the proof of Proposition \ref{proposition:twistgen}.
\end{proof}

\section{$r$ even: $\mathcal T_\phi$ has finite index in $\Mod(\Sigma_g)$}\label{section:deven}
We continue to assume that $r \mid (2g-2)$, but now we take $r = 2d$ to be even. For $r$ even, we cannot give a complete characterization of $\mathcal T_\phi$ as in Proposition \ref{proposition:twistgen}, but we will show that $\mathcal T_\phi$ has finite index in $\Mod(\Sigma_g)$. The minimal genus for which the ensuing arguments apply has a rather intricate dependence on $r$, encapsulated in the definition below.

\begin{definition}\label{definition:gd}
For an integer $d \ge 1$, define $k(d)$ as follows:
\[
k(d) = \begin{cases}
2 & d \mbox{ odd or }d \ge 6 \mbox{ even}\\
6 & d =2\\
5 & d = 4.
\end{cases}
\]
Suppose $r = 2d$ is an even integer. Define 
\[
g(r) = k(d) d + 1.
\]
\end{definition}

\begin{proposition}\label{proposition:twistgeneven}
Let $r = 2d$ be an even integer. Suppose $g \ge g(r)$ and that $r < g - 1$. Then $\mathcal T_\phi$ is a finite-index subgroup of $\Mod(\Sigma_g)$. 
\end{proposition}

 The presence of an underlying $\Z/2\Z$ spin structure makes proving the analogues of Lemma \ref{lemma:converse} and Lemma \ref{lemma:containsjk} substantially more difficult. At present, we do not know how to establish the analogue of Lemma \ref{lemma:converse}, owing to the fact that the Arf invariant provides an obstruction to finding the configurations of curves on subsurfaces needed for the arguments therein. Thus we content ourselves with showing that $\mathcal T_\phi \leqslant \Mod(\Sigma_g)$ is finite-index.

\begin{proof}{\em (of Proposition \ref{proposition:twistgeneven})} The proof of Proposition \ref{proposition:twistgeneven} follows a similar outline to that of Proposition \ref{proposition:twistgen}. We begin with an overview of the proof.

\para{Overview} To establish finiteness of the index $[\Mod(\Sigma_g): \mathcal T_\phi]$, it suffices to show that the indices $[\Sp(2g,\Z): \Psi(\mathcal T_\phi)]$ and $[\mathcal I_g: \mathcal T_\phi \cap \mathcal I_g]$ are both finite. Finiteness of $[\Sp(2g,\Z): \Psi(\mathcal T_\phi)]$ is established in Lemma \ref{lemma:spimageeven} of Step 1, which moreover gives a complete description of the subgroup $\Psi(\mathcal T_\phi)$. 

Finiteness of $[\mathcal I_g: \mathcal T_\phi \cap \mathcal I_g]$ is obtained in Steps 2 and 3, again by using the Johnson homomorphism to analyze the intersection $\mathcal T_\phi \cap \mathcal I_g$ as in Steps 2-4 of the proof of Proposition \ref{proposition:twistgen}. The main result of Step 2 is Lemma \ref{lemma:taufindex}, which shows that $\tau(\mathcal T_\phi \cap \mathcal I_g)$ has finite index in $\wedge^3 H_\Z/H_\Z$. Step 3 completes the argument by showing the containment $\mathcal K_g \leqslant \mathcal T_\phi$; this is obtained as Lemma \ref{lemma:containsjkeven}. We advise the reader that Step 3 is substantially more complicated than its counterpart Step 4 of the proof of Proposition \ref{proposition:twistgen}, and will require an explanatory outline of its own.

\para{Step 1: The symplectic quotient} 
The case of $r$ even is no more difficult than for $r$ odd. Let $q$ be a $\Z/2\Z$-valued spin structure. An {\em anisotropic transvection} is a transvection
\[
T_v(w) = w + \pair{w,v}v
\]
for a primitive $v \in H_1(\Sigma_g; \Z)$ such that $q(v) = 0$. 

The following theorem is surely well-known to experts but we were unable to find a reference. A special case is treated in \cite[Proposition 14]{dieudonne}. 
\begin{theorem}[Folklore]\label{theorem:aniso}
Let $q$ be a $\Z/2\Z$-valued spin structure on $\Sigma_g$ for $g \ge 3$, and let $\Sp(2g, \Z)[q]$ denote the subgroup of $\Sp(2g, \Z)$ that fixes $q$. Then $\Sp(2g, \Z)[q]$ is generated by the collection of anisotropic transvections 
\[
\{T_v \mid v \in H_1(\Sigma_g; \Z) \mbox { primitive, }q(v) = 0\}.
\]
\end{theorem}
\begin{proof}
The action of $\Sp(2g,\Z)$ on the set of spin structures factors through the quotient $f: \Sp(2g, \Z) \to \Sp(2g, \Z/2\Z)$. Define $\Sp(2g,\Z)[2] := \ker(f)$.  Thus, there is a short exact sequence
\[
1 \to \Sp(2g, \Z)[2] \to \Sp(2g, \Z)[q] \to \Sp(2g, \Z/2\Z)[q] \to 1,
\]
with $\Sp(2g, \Z/2\Z)[q]$ denoting the stabilizer of $q$ in $\Sp(2g, \Z/2\Z)$. According to \cite[Theorem 14.16]{grove}, the group $\Sp(2g, \Z/2\Z)[q]$ is generated by the images of all anisotropic transvections. So it remains to see only that the subgroup of $\Sp(2g, \Z)[q]$ generated by anisotropic transvections contains $\Sp(2g, \Z)[2]$. According to \cite[Lemma 5]{johnsoniii}, the group $\Sp(2g, \Z)[2]$ is generated by the collection of ``square transvections'' $T_w^2$, where $w$ ranges over all primitive $w \in H_1(\Sigma_g; \Z)$. 

If $q(w) = 0$ then $T_w \in \Sp(2g, \Z)[q]$ and so there is nothing to show. Assume now that $q(w) = 1$. It is easy to produce (e.g. by the change-of-coordinates principle on $\Sigma_g$) vectors $v_1, v_2, v_3 \in H_1(\Sigma_g; \Z)$ with the following properties:
\begin{enumerate}
\item $q(v_i) = 0$ for all $i$,
\item $\pair{v_1,v_2} = \pair{v_2,v_3} = 1$ and $\pair{v_1, v_3} = 0$,
\item $\pair{v_i, w} = 0$ for all $i$,
\item $v_1 + v_3 = w$.
\end{enumerate}
The chain relation in $\Mod(\Sigma_g)$ (Proposition \ref{proposition:chain}) descends to show the relation
\[
(T_{v_1}T_{v_2}T_{v_3})^4 = T_w^2.
\]
Since the left-hand side is a product of anisotopic transvections, it follows that for $w$ arbitrary, $T_w^2 \in \Sp(2g, \Z)[q]$ as required. \end{proof}

The following is the main result of Step 1.
\begin{lemma}\label{lemma:spimageeven}
Let $\phi$ be a $\Z/r \Z$-valued spin structure for $r$ an even integer, and let 
\[
q := \phi \pmod 2
\]
 denote the associated $\Z/2\Z$-valued spin structure. The symplectic representation $\Psi: \Mod(\Sigma_g) \to \Sp(2g, \Z)$ restricts to a surjection
\[
\Psi: \mathcal T_\phi \twoheadrightarrow \Sp(2g, \Z)[q],
\]
where $\Sp(2g, \Z)[q]$ denotes the stabilizer of $q$ in $\Sp(2g,\Z)$. 
\end{lemma}

\begin{proof}
As $\mathcal T_\phi$ preserves the $\Z/r\Z$-valued spin structure $\phi$, it also preserves the mod-$2$ reduction $q$. Thus $\mathcal T_\phi \leqslant \Mod(\Sigma_g)[q]$ and so $\Psi(\mathcal T_\phi) \leqslant \Sp(2g, \Z)[q]$. Let $v \in H_1(\Sigma_g; \Z)$ be a primitive element satisfying $q(v) = 0$. By Lemma \ref{item:homology}, there is some curve $c$ with $[c] = v$ and $\phi(c) = 0$. As $T_c \in \Mod(\Sigma_g)[\phi]$ and $\Psi(T_c) = T_v$, the result now follows from Theorem \ref{theorem:aniso}.
\end{proof}

\para{Step 2: The Johnson homomorphism} 
We remind the reader that while the value $\phi(c)$ on a simple closed curve depends on more than the homology class $[c] \in H_1(\Sigma_g;\Z)$, the discussion of Section \ref{subsection:classical} establishes that the mod-$2$ reduction $q(c)$ {\em does} depend only on the homology class $[c]$ (indeed, the coefficients here can be taken to be $\Z/2\Z$). Thus the arguments in Step 3 can be carried out entirely in the homological setting.

For the duration of Step 2, we adopt the following notation. As usual, define 
\[
q:= \phi \pmod 2.
\]
 There exists a symplectic basis $\{x_1, y_1, \dots, x_g, y_g\}$ for $H_1(\Sigma_g;\Z)$ such that $q(x_i) = 0$ for $1 \le i \le g$ and $q(y_j) = 0$ for $1 \le j \le g-1$; with such a basis, $\Arf(q)$ depends only on $g$ and on $q(y_g)$. 

Before proceeding to the main result of Step 2 (Lemma \ref{lemma:taufindex}), we begin with an algebraic lemma.
\begin{lemma}\label{lemma:next}
Set $v:= x_1 \wedge y_1 \wedge x_4$. Let $V \leqslant \wedge^3 H_\Z$ denote the submodule generated by the set
\[
\{gv \mid g \in \Sp(2g, \Z)[q]\}.
\]
Then $V = \wedge^3 H_\Z$ for $g \ge 5$. 
\end{lemma}
\begin{proof}
As remarked in Lemma \ref{lemma:converse}, $\wedge^3H_\Z$ is generated by elements of the form $z_i \wedge z_j \wedge z_k$ with each $z_i \in \{x_1, y_1,\dots, x_g, y_g\}$. To begin with, we will exhibit generators for the submodule of $\wedge^3 H_\Z$ spanned by generators $z_i\wedge z_j \wedge z_k$ for which $z_i,z_j,z_k \in \{x_1,y_1, \dots,x_{g-1}, y_{g-1}\}$. The restriction of $\Sp(2g;\Z)[q]$ to this submodule is independent of the parity of $q$. For $i \ne j \le g-1$, define $S_{i,j} \in \Sp(2g, \Z)$ via
\begin{align*}
S_{i,j}(x_i) &= x_j, &  S_{i,j}(y_i) &= y_j,\\
 S_{i,j}(x_j) &= x_i, & S_{i,j}(y_j) & = y_i,
\end{align*}
with all other generators fixed. As $q(x_k) = q(y_k) = 0$ for $k \le g-1$, in fact $S_{i,j}$ is an element of $\Sp(2g, \Z)[q]$. Applying $S_{i,j}$ for $i,j \ne 4$ to $v$ shows that $V$ contains all generators of the form $x_i \wedge y_i \wedge x_4$. Applying $S_{i,4}$ to $x_j \wedge y_j \wedge x_4$ for $i \ne j$ shows that $V$ contains all generators of the form $x_j \wedge y_j \wedge x_i$ for $j \ne 4$; then applying $S_{j,4}$ to $x_j \wedge y_j \wedge x_i$ shows that $V$ contains all elements of the form $x_j \wedge y_j \wedge x_i$.

For $1 \le i \le g-1$, define $R_i \in \Sp(2g, \Z)$ via
\[
R_i(x_i) = y_i,\quad R_i(y_i)= -x_i
\]
with all other generators fixed. Again, the condition $q(x_k) = q(y_k) = 0$ for $k \le g-1$ implies that $R_i$ is an element of $\Sp(2g,\Z)[q]$. Applying $R_i$ to $x_j \wedge y_j \wedge x_i$ shows that also $V$ contains all elements of the form $x_j \wedge y_j \wedge y_i$. 

It remains to exhibit generators of the form $z_i \wedge z_j \wedge z_k$ with $z_\ell \in \{x_\ell,y_\ell\}$ and $i,j,k<g$ all distinct. Consider the transvection $T_{x_4 - x_1} \in \Sp(2g,\Z)[q]$. Applied to $x_1 \wedge y_1 \wedge x_2$, this shows that
\[
x_1 \wedge(y_1 + x_4) \wedge x_2 \in V,
\]
hence also $x_1 \wedge x_2 \wedge x_4 \in V$. Now by repeated applications of the elements $S_{i,j}$ and $R_i$, one can produce all remaining generators. 

In the case $q(y_g) = 0$, the elements $S_{i,g}$ and $R_g$ are contained in $\Sp(2g,\Z)[q]$, and so the above argument extends to complete this case. It remains to consider the case where $q(y_g) = 1$. In this case, the formula \eqref{equation:q} defining a $\Z/2\Z$-valued quadratic form shows that $q(y_{g-1} + y_{g}) = 0$. Applying $T_{y_{g-1} + y_{g}}$ to the elements $x_1 \wedge x_2 \wedge x_{g-1}$ and $x_1 \wedge y_1 \wedge x_{g-1}$ shows that $x_1 \wedge x_2 \wedge y_g$ and $x_1 \wedge y_1 \wedge x_g$ are elements of $V$. Applying $S_{i,j}$ and $R_i$ for $i,j \le g-1$ produces all elements of the form $z_i \wedge z_j \wedge y_g$ with $z_\ell \in \{x_\ell,y_\ell\}\ (i,j \le g-1)$. Then applying $T_{x_g}$ to these elements shows that also each $z_i \wedge z_j \wedge x_g \in V$. 

By \eqref{equation:q}, we have $q(x_1 + x_g - y_g) = 0$. Applying $T_{x_1 + x_g - y_g}^{-1}$ to $y_1 \wedge y_2 \wedge y_g$ gives
\[
w = (y_1 + x_1 + x_g - y_g) \wedge y_2 \wedge (x_1 + x_g);
\] 
expanding this product yields the expression $w = -y_2 \wedge x_g \wedge y_g + w'$, with $w'$ expressed entirely in terms of generators already known to be elements of $V$. Applying $S_{i,j}$ and $R_{i}$ as in the above paragraph shows that all the remaining generators $z_i \wedge x_g \wedge y_g$ are elements of $V$.
\end{proof}

The following is the main result of Step 2.
\begin{lemma}\label{lemma:taufindex}
For $g \ge 5$, the image $\tau(\mathcal T_\phi \cap \mathcal I_g)$ under the Johnson homomorphism is a finite-index subgroup of $\wedge^3 H_\Z/H_\Z$.
\end{lemma}

\begin{proof}
As stated in Lemma \ref{lemma:jhom}.1, the homomorphism $\tau: \mathcal I_g \to \wedge^3 H_\Z/H_\Z$ is $\Sp(2g, \Z)$-equivariant. The strategy for the proof of Lemma \ref{lemma:taufindex} is to first exhibit a single nonzero element of $\tau(\mathcal T_\phi \cap \mathcal I_g)$, and then to exploit this equivariance.

By Corollary \ref{corollary:allvalueseven}.\ref{item:chainevenrestricted}, there exists a $3$-chain of admissible curves $a_1, a_2, a_3$ such that
\[
[a_1] = x_1,\quad [a_2] = y_1, \quad [a_3] = x_4- x_1.
\]
Let $\nu$ be a regular neighborhood of this chain, and denote the boundary curves as $b, b'$. As $a_1$ and $a_2$ are admissible, homological coherence implies that $\phi(b) = \phi(b') = -1$ when oriented so that $\nu$ lies to the left of both $b, b'$. By Lemma \ref{item:allcurves}, $T_b^r$ is an element of $\mathcal T_\phi$. It follows by the chain relation (Proposition \ref{proposition:chain}) that the bounding pair map $T_b^r T_{b'}^{-r} \in \mathcal T_\phi$.  One sees that $[b] = [a_1]+ [a_3] =  x_4$. By Lemma \ref{lemma:jhom}.3, 
\[
\tau(T_b^r T_{b'}^{-r}) = r (x_1 \wedge y_1 \wedge x_4).
\]
By Lemma \ref{lemma:spimageeven} and the equivariance of $\tau$ with respect to $\Sp(2g,\Z)$ (and {\em a fortiori} with respect to $\Sp(2g,\Z)[q]$), it follows that $\tau(\mathcal T_\phi \cap \mathcal I_g)$ contains the $\Z$-span of the entire $\Sp(2g, \Z)[q]$-orbit of $v:= r (x_1 \wedge y_1 \wedge x_4)$. Lemma \ref{lemma:taufindex} now follows from Lemma \ref{lemma:next}. 
\end{proof}

\para{Step 3: The Johnson kernel} In this section, we establish the following result.
\begin{lemma}\label{lemma:containsjkeven}
Let $\phi$ be a $\Z/2d\Z$-valued spin structure on $\Sigma_g$. Assume that $g$ satisfies the hypotheses of Proposition \ref{proposition:twistgeneven}. Then $\mathcal T_\phi$ contains the Johnson kernel $\mathcal K_g$.
\end{lemma}

Before beginning the proof, we explain the difficulties imposed by the assumption that $r = 2d$ is even.

\para{The Arf invariant as obstruction} The mechanism of proof for Lemma \ref{lemma:containsjk} was the chain relation (Proposition \ref{proposition:chain}): if $S\subset \Sigma_g$ has one boundary component, we exploited Corollary \ref{corollary:allvalues} to produce a maximal chain $\{a_i\}$ of curves on $S$ with $\phi(a_i) = 0$, and then used the chain relation to express $T_{\partial S}$ in terms of the admissible twists $\{T_{a_i}\}$. Now suppose $\phi$ is a $\Z/r\Z$-valued spin structure for $r$ even, and let $q = \phi \pmod 2$ denote the mod-$2$ reduction. For any subsurface $S\subset \Sigma_g$ with one boundary component, $q$ restricts to give a $\Z/2\Z$-valued spin structure $q\,\vline_S$ on $S$. The Arf invariant of $q\,\vline_S$, written here as $\epsilon(S)$, provides an obstruction to the existence of a maximal chain $\{a_i\}$ of admissible curves on $S$, since such a chain determines the value $\epsilon(S)$ solely as a function of $g(S)$. 

Suppose $c \subset \Sigma_g$ is a separating curve that divides $\Sigma_g$ into disjoint surfaces $S, S'$. Such a $c$ is called {\em easy} if at least one of $S, S'$ supports a maximal chain of admissible curves, and is {\em hard} otherwise. By Corollary \ref{corollary:allvalueseven}.\ref{item:chaineven} and the chain relation (Proposition \ref{proposition:chain}), if $c$ is easy, then $T_c \in \mathcal T_\phi$. 

\para{Outline of proof of Lemma \ref{lemma:containsjkeven}} We begin with Lemma \ref{lemma:easy}, which characterizes those subsurfaces supporting a maximal chain of admissible curves in terms of the Arf invariant. This in particular shows the relevance of the genus of the subsurface mod $4$, which in turn forces us to treat the cases $r \equiv 0, r \equiv 2 \pmod 4$ separately. We therefore establish Lemma \ref{lemma:containsjkeven} by combining Lemma \ref{lemma:containsjkeven2} and \ref{lemma:containsjkeven0}, which treat the cases of $r \equiv 2 \pmod 4$ and $r \equiv 0 \pmod 4$, respectively. 

These are handled in Substeps 1 and 2, respectively. In each case, we first show that all separating twists of particular genera are elements of $\mathcal T_\phi$. In Substep 1, Lemma \ref{lemma:dodd} shows that all separating twists of genus $d$ lie in $\mathcal T_\phi$. In Substep 2, Lemma \ref{lemma:deven} shows that all separating twists of genus $h \equiv d+2 \pmod{2d}$ lie in $\mathcal T_\phi$, and Lemma \ref{lemma:deven2} establishes the same result for separating twists of genus $h \equiv d+4 \pmod{2d}$. Lemmas \ref{lemma:containsjkeven2} and \ref{lemma:containsjkeven0} then follow from these preliminary results and an application of the $D_n$ relation (Proposition \ref{proposition:dnrel}).

\begin{lemma}\label{lemma:easy}
Let $S \subset \Sigma$ be a subsurface with single boundary component. Assume the genus $g(S) \ge 2$. Then there is a maximal chain of admissible curves on $S$ if and only if one of the following conditions hold:
\begin{itemize}
\item $g(S) \equiv 1\mbox{ or }2 \pmod 4$ and $\epsilon(S) = 1$,
\item $g(S) \equiv 3\mbox{ or }0 \pmod 4$ and $\epsilon(S) = 0$.
\end{itemize}
\end{lemma}
\begin{proof}
Suppose $S$ supports a maximal chain $a_1, \dots, a_{2g(S)}$ of admissible curves. Since the chain determines a basis for $H_1(S; \Z)$, the conditions $\phi(a_i) = 0$ completely determine $\phi$. One can easily compute $\epsilon(S)$ from this and see that the above conditions are necessary. Sufficiency follows from Corollary \ref{corollary:allvalueseven}.\ref{item:chaineven}. 
\end{proof}

\para{Substep 1: \bm{$d$} odd} The objective of Substep 1 is Lemma \ref{lemma:containsjkeven2} below. The first step is to see that all separating twists $T_c$ of genus $d$ are elements of $\mathcal T_\phi$, regardless of whether $c$ is easy. 
\begin{lemma}\label{lemma:dodd}
Let $S \subset \Sigma_g$ be a subsurface of genus $d$ with a single boundary component $c$. If $\phi$ is a $\Z/2d\Z$-valued spin structure with $d$ odd, then $T_c \in \mathcal T_\phi$.
\end{lemma}
\begin{proof}
If $c$ is easy then there is nothing to show. Assume therefore that $c$ is hard. If $c$ is oriented so that $S$ lies to the right, then $\phi(c) = -(1-2d) \equiv -1 \pmod {2d}$. The assumption that $r = 2d < g-1$ implies that $\Sigma_g \setminus S$ has genus at least $2$. We claim that there exists a $3$-chain of admissible curves $x,y,z$ on $\Sigma_g \setminus S$ such that $c \cup x \cup z$ forms a pair of pants. To see this, we invoke Corollary \ref{corollary:admissible} to let $x \subset \Sigma_g \setminus S$ be an admissible curve. Let $z \subset \Sigma_g \setminus S$ be any curve such that $c \cup x \cup z$ bounds a pair of pants; admissibility of $z$ follows by the homological coherence property, as $c$ is oriented with $\Sigma_g \setminus S$ to the left. To construct $y$, let $y' \subset \Sigma_g \setminus S$ be any curve such that $x, y', z$ forms a chain. By Corollary \ref{corollary:intersections}, $y'$ can be replaced with an admissible curve $y$ with the same intersection properties. 

 Let $S'$ denote the connected surface of genus $d+1$ containing $S$ and $x \cup y \cup z$.  If $\mathcal B$ is a basis for $H_1(S; \Z)$, then $\mathcal B \cup \{x,y\}$ forms a basis for $H_1(S';\Z)$. Applying the formula \eqref{equation:arf} for the Arf invariant, it follows that $\epsilon(S') = \epsilon(S)+1$. 

Since $c$ is hard and $d=g(S)$ is odd, Lemma \ref{lemma:easy} implies that $\epsilon(S) = 0$ if $g(S) \equiv 1 \pmod 4$ and that $\epsilon(S) = 1$ otherwise. Recalling that $\epsilon(S') = \epsilon(S) + 1$, in the first case, $g(S') \equiv 2 \pmod 4$ and $\epsilon(S') = 1$, and in the second case, $g(S') \equiv 0 \pmod 4$ and $\epsilon(S') = 0$. Lemma \ref{lemma:easy} then implies that $c':= \partial S'$ must be easy, and so $T_{c'} \in \mathcal T_\phi$. Applying the chain relation (Proposition \ref{proposition:chain}) to $x,y,z$ shows that $T_c T_{c'} \in \mathcal T_\phi$; this implies that also $T_c \in \mathcal T_\phi$.
\end{proof}

\begin{lemma}\label{lemma:containsjkeven2}
Let $\phi$ be a $\Z/2d\Z$-valued spin structure on $\Sigma_g$ with $d$ odd. Assume that $g$ satisfies the hypotheses of Proposition \ref{proposition:twistgeneven}. Then $\mathcal T_\phi$ contains the Johnson kernel $\mathcal K_g$.
\end{lemma}

\begin{proof}
By Theorem \ref{theorem:johnson2}, it suffices to show that $T_c \in \mathcal T_\phi$ for all separating curves $c$ of arbitrary genus. To do this, we combine Lemma \ref{lemma:dodd} with the $D_n$ relation (Proposition \ref{proposition:dnrel}). Suppose $c$ is a separating curve on $\Sigma_g$. Since $g = kd+1$ with $k \ge 2$, at least one side of $c$ must be a subsurface $S$ of genus $g(S) \ge d+1$. Set $n := 2g(S) -2d +1$. By Corollary \ref{corollary:dneven}, there is a configuration $\mathscr D_n$ of admissible curves as in the $D_n$ relation for which $\Delta_2 = c$. The other boundary component $\Delta_0$ bounds a subsurface of genus $d$. Applying the $D_n$ relation, we have $T_{\Delta_0}^{n-1} T_c \in \mathcal T_\phi$. But since $\Delta_0$ bounds a surface of genus $d$, also $T_{\Delta_0} \in \mathcal T_\phi$ by Lemma \ref{lemma:dodd}. Thus $\mathcal K_g \leqslant \mathcal T_\phi$ in this case.
\end{proof}

\para{Substep 2: \bm{$d$} even} The objective is to establish Lemma \ref{lemma:containsjkeven0}. The argument here follows a similar outline to that of Substep 1 but now requires the two preliminary Lemmas \ref{lemma:deven} and \ref{lemma:deven2}. 
\begin{lemma}\label{lemma:deven}
Let $S \subset \Sigma_g$ be a subsurface of genus $g(S) \ge 5$ with a single boundary component $c$, such that $g(S) \equiv d + 2 \pmod {2d}$. If $\phi$ is a $\Z/2d\Z$-valued spin structure with $d$ even, then $T_c \in \mathcal T_\phi$.
\end{lemma}
\begin{proof}
Orient $c$ so that $S$ lies to the left. Then 
\[
\phi(c) = 1-2g(S) \equiv 1 - 2(d+2) \equiv -3 \pmod {2d}.
\]
By Corollary \ref{corollary:allvalueseven}.\ref{item:chainevenrestricted}, there exists a chain $a_1, \dots, a_6$ of admissible curves on $S$. Let $a_7$ be any curve on $S$ such that $i(a_7, a_k) = 1$ for $k = 6$ and is zero for $k \le 5$, and such that $c \cup a_1 \cup a_3 \cup a_5 \cup a_7$ bounds a subsurface of $S$ homeomorphic to $\Sigma_{0,5}$. By homological coherence, $a_7$ is admissible. 

Let $S'$ denote the subsurface of $S$ homeomorphic to $\Sigma_{g(S)-3, 1}$ determined by the complement of the chain $a_1, \dots, a_7$. Applying the formula \eqref{equation:arf} for the Arf invariant, one finds that $\epsilon(S') = \epsilon(S)$. On the other hand, $g(S') \equiv g(S)+1 \pmod{4}$. By hypothesis, $g(S)$ is even, and so referring to Lemma \ref{lemma:easy}, if $c$ is hard, then $c':= \partial S'$ must be easy. The arguments given at the conclusion of Lemma \ref{lemma:dodd} now apply to give the result. 
\end{proof}

\begin{lemma}\label{lemma:deven2}
Let $S \subset \Sigma$ be a subsurface of genus $g(S) \ge 9$ with a single boundary component $c$, such that $g(S) \equiv d+4 \pmod{2d}$. If $\phi$ is a $\Z/2d\Z$-valued spin structure with $d$ even, then $T_c \in \mathcal T_\phi$. 
\end{lemma}
\begin{proof}
This is proved along similar lines to Lemma \ref{lemma:deven}. Arguing as in the first paragraph of the proof of Lemma \ref{lemma:deven}, there exists a chain $a_1, \dots, a_{15}$ of admissible curves on $S$ such that $c \cup a_1 \cup a_3 \cup \dots \cup a_{15}$ bounds a subsurface of $S$ homeomorphic to $\Sigma_{0,9}$. Let $S'$ denote the subsurface of $S$ homeomorphic to $\Sigma_{g(S) - 7,1}$ determined by the complement of the chain $a_1, \dots, a_{15}$. The rest of the argument proceeds as in Lemma \ref{lemma:deven}: one shows that if $c$ is hard, necessarily $c' := \partial S'$ must be easy, and the result follows as before by the chain relation (Proposition \ref{proposition:chain}).
\end{proof}

\begin{lemma}\label{lemma:containsjkeven0}
Let $\phi$ be a $\Z/2d\Z$-valued spin structure on $\Sigma_g$ with $d$ even. Assume that $g$ satisfies the hypotheses of Proposition \ref{proposition:twistgeneven}. Then $\mathcal T_\phi$ contains the Johnson kernel $\mathcal K_g$.
\end{lemma}
\begin{proof}
According to Johnson's Theorem \ref{theorem:johnson2}, in order to show that $\mathcal K_g \leqslant \mathcal T_\phi$, it suffices to exhibit all separating twists of genus $1$ and $2$ as elements of $\mathcal T_\phi$. To do this, we again appeal to the $D_n$ relation (Proposition \ref{proposition:dnrel}). Suppose $c$ is a separating curve on $\Sigma_g$ with $g(c) \le 2$. By hypothesis, $g \ge kd+1$ with $d$ even and $k \ge 2$. Since the genus of one side of $c$ is at most $2$, the genus $h$ of the other side of $c$ is at least $kd - 1 \ge 2d-1$. If $d \ge 6$, then $2d-1 \ge 11$. If $d = 4$, then by assumption $k \ge 5$, and so $h \ge 19$. If $d = 2$ then we assume $k \ge 6$, so that $h \ge 11$.

In all three of these cases, Corollary \ref{corollary:dneven} implies that there exists an $n \ge 4$ and a configuration $a, a', c_1, \dots, c_{2n-1}$ of admissible curves in the configuration of the $D_{2n+1}$ relation, with $\Delta_2 = c$ and $C_4$ bounding a subsurface of genus $g(S) \equiv d+4 \pmod{2d}$ disjoint from $S$, such that the hypotheses of Lemmas \ref{lemma:deven} and \ref{lemma:deven2} hold.

By the $D_{k}$ relation (for $k = 2n+1,5,9$ respectively), $T_{\Delta_0}^{2n-1}T_c$ and $T_{\Delta_0}^3 T_{C_2}$ and $T_{\Delta_0}^7 T_{C_4}$ are all elements of $\mathcal T_\phi$. By Lemma \ref{lemma:deven}, $T_{C_2} \in \mathcal T_\phi$ as well, hence $T_{\Delta_0}^3 \in \mathcal T_\phi$. Likewise, Lemma \ref{lemma:deven2} shows that $T_{C_4} \in \mathcal T_\phi$, hence $T_{\Delta_0}^7 \in \mathcal T_\phi$. Combining these last two results shows that $T_{\Delta_0} \in \mathcal T_\phi$, and ultimately that $T_c \in \mathcal T_\phi$ as required.
\end{proof}
This concludes the proof of Proposition \ref{proposition:twistgeneven}. \end{proof}

\section{Connectivity of some complexes}\label{section:trick}
This section is devoted to establishing the connectivity of the simplicial complexes $\mathcal C_{sep,2}(\Sigma_g)$ and $\mathcal C_\phi^1(\Sigma_g)$ to be defined below. The first of these will be an important ingredient in the proof of Proposition \ref{proposition:submakesT}, and the second will feature in the proof of Theorem \ref{theorem:toric}. The mechanism by which these will be seen to be connected is the so-called {\em Putman trick}. The version given below is slightly less general than the full theorem as stated in \cite{putmantrick}, but will suffice for our purposes.

\begin{theorem}[The Putman trick]\label{theorem:putman}
Let $X$ be a simplicial graph, and let $G$ act on $X$ by simplicial automorphisms. Suppose that the action of $G$ on the set of vertices $X^{(0)}$ is transitive. Fix some base vertex $v \in X^{(0)}$. Let $\Sigma = \Sigma^{-1}$ be a symmetric set of generators for $G$, and suppose that for each $s \in \Sigma$, there is a path in $X$ connecting $v$ to $s\cdot v$. Then $X$ is connected.
\end{theorem}

\begin{definition}
$\mathcal{C}_{sep, 2}(\Sigma_g)$ is the simplicial graph where vertices correspond to (isotopy classes of) separating curves $c$ bounding a subsurface homeomorphic to $\Sigma_{2,1}$, and where $c$ and $d$ are adjacent in $\mathcal{C}_{sep,2}(\Sigma_g)$ whenever $c$ and $d$ are disjoint in $\Sigma_g$. 
\end{definition}

\begin{figure}
\labellist
\small 
\pinlabel $c$ at 114 80
\pinlabel $d$ at 247 80
\endlabellist
\includegraphics{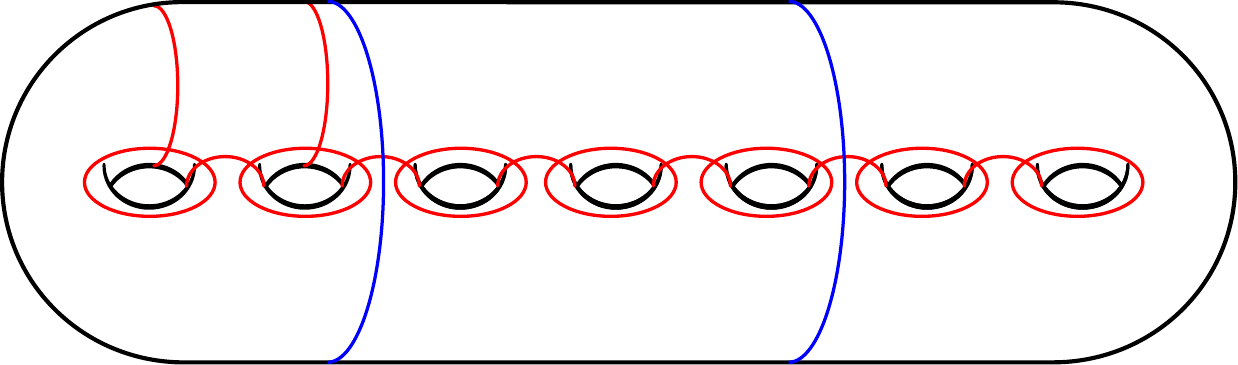}
\caption{The configuration of curves needed for Lemma \ref{lemma:connected1}.}
\label{figure:putmantrick}
\end{figure}

\begin{lemma}\label{lemma:connected1} $\mathcal{C}_{sep,2}(\Sigma_g)$ is connected for $g \ge 5.$
\end{lemma}

\begin{proof}
This is a straightforward consequence of Theorem \ref{theorem:putman}. With reference to Figure \ref{figure:putmantrick} and the standard generating set of Figure \ref{figure:humph}, observe that only the generator $T_{c_2}^{\pm}$ does not fix the base vertex $c$. In this case, the genus $2$ subsurface determined by $d$ is disjoint from both $c$ and $T_{c_2}^\pm(c)$, and so there is a path $c,d,T_{c_2}^{\pm}(c)$ in $\mathcal C_{sep, 2}(\Sigma_g)$. 
\end{proof}

\begin{definition}
Let $\phi$ be a $\Z/r \Z$-valued spin structure on a surface $\Sigma_g$. The graph $\mathcal C_\phi(\Sigma_g)$ has vertices consisting of the admissible curves for $\phi$, where $a$ and $b$ are adjacent whenever $i(a,b) = 0$. The graph $\mathcal C_\phi^1(\Sigma_g)$ has the same vertex set as $\mathcal C_\phi(\Sigma_g)$, but vertices $a,b$ are adjacent whenever $i(a,b) = 1$.
\end{definition}

\begin{lemma}\label{lemma:c1connected}
$\mathcal C_\phi^1(\Sigma_g)$ is connected for $g \ge 5$.
\end{lemma}
\begin{proof}
The first step is to establish the connectivity of $\mathcal C_\phi(\Sigma_g)$. Let $a,b$ be vertices. Choose subsurfaces $S_{a}, S_{b}$ containing $a,b$ respectively, each homeomorphic to $\Sigma_{2,1}$. By Lemma \ref{lemma:connected1}, there is a path $S_{a_0}, \dots, S_{a_n}$ in $\mathcal C_{sep,2}(\Sigma_g)$ with $a \subset S_{a_0}$ and $b\subset S_{a_n}$, with each $S_{a_i}$ disjoint from $S_{a_{i+1}}$. By Corollary \ref{corollary:admissible}, on each $S_{a_i}$ there exists some admissible curve $a_i$. By construction, $a = a_0, a_1, \dots, a_n = b$ is a path in $\mathcal C_\phi(\Sigma_g)$ connecting $a$ to $b$. 

The connectivity of $\mathcal C_\phi^1(\Sigma_g)$ now follows readily. Given a path $a = a_0, \dots, a_n = b$ in $\mathcal C_\phi(\Sigma_g)$, Corollary \ref{corollary:intersections} implies that for each $i$, there exists some admissible curve $c_i$ such that $i(a_i, c_i) = i(a_{i+1},c_i) = 1$. The path $a_0, c_0, a_1, c_1, \dots, c_{n-1}, a_n$ connects $a$ to $b$ in $\mathcal C_\phi^1(\Sigma_g)$. 
\end{proof}

\section{Subsurface push subgroups and $\mathcal T_\phi$}\label{section:push}

As discussed in the introduction, the main technical result on the groups $\Mod(\Sigma_g)[\phi]$ and $\mathcal T_\phi$ that we require is a criterion for a collection of Dehn twists to generate $\mathcal T_\phi$, given below as Theorem \ref{theorem:networkgenset}. This is the first of two sections dedicated to proving Theorem \ref{theorem:networkgenset}. Here, we formulate and prove the intermediate result Proposition \ref{proposition:submakesT}, which gives a generating set for $\mathcal T_\phi$ {\em not} consisting entirely of Dehn twists. The results here concern a class of subgroups known as {\em spin subsurface push subgroups}; these are introduced in Sections \ref{subsection:subpush} and \ref{subsection:spinsubpush}.

\subsection{Subsurface push subgroups}\label{subsection:subpush} Recall the classical {\em inclusion map}, as discussed in \cite[Theorem 3.18]{FM}. Let $S' \subset S$ be a subsurface either of genus $g(S') \ge 2$ with $n \ge 1$ boundary components, or else of genus $g(S') = 1$ with $n \ge 2$ boundary components. Assume that no component of $\partial S'$ bounds a closed disk in $S$. Let $a_1, \dots, a_k$ denote the boundary components of $S'$ that bound punctured disks in $S$, let $b_1, b_1', \dots, b_\ell, b_\ell'$ denote the pairs of boundary components of $S'$ that cobound an annulus in $S$, and $c_1, \dots, c_m$ denote the remaining boundary components. Let $i_*: \Mod(S') \to \Mod(S)$ denote the map on mapping class groups arising from the inclusion $i: S' \into S$. Then 
\[
\ker(i_*) = \pair{T_{a_1}, \dots, T_{a_k}, T_{b_1} T_{b_1'}^{-1}, \dots, T_{b_\ell}T_{b_\ell'}^{-1}}.
\]

Let $\Delta$ be a boundary component of $S'$, and suppose that $\Delta$ does not bound a punctured disk in $S$. Let $\overline{S'}$ denote the surface obtained from $S'$ by capping off $\Delta$ with a closed disk. According to (\ref{equation:birman2}), there is a subgroup of $\Mod(S')$ isomorphic to $\pi_1(UT\overline{S'})$. The {\em subsurface push subgroup} for $(S', \Delta)$ is defined to be the image of $\pi_1(UT\overline{S'})$ under the inclusion $i_*: \Mod(S') \to \Mod(S)$. This will be written $\Pi(S',\Delta)$, or simply $\Pi(S')$ if the boundary component does not need to be emphasized.

We remark here that $i_*$ restricts to an {\em injection} $\pi_1(UT\overline{S'}) \into \Mod(S)$, even when there exists some other boundary component $\Delta'$ of $S'$ such that $\Delta \cup \Delta'$ cobounds an annulus on $S$. To see this, observe that $\pi_1(UT\overline{S'})\leqslant \Mod(S')$ is characterized by the property that $f \in \pi_1(UT\overline{S'})$ if and only if $f$ becomes isotopic to the identity when extended to $\overline{S'}$. It is easy to see that no element of $\ker(i_*)$ has this property. 

\subsection{Spin subsurface push subgroups}\label{subsection:spinsubpush} Let $S' \subset S$ be a subsurface with some boundary component $\Delta$ satisfying $\phi(\Delta) = -1$. The following lemma shows that $\Mod(S)[\phi]$ contains a finite-index subgroup of $\Pi(S', \Delta)$. This subgroup, written $\widetilde{\Pi}(S', \Delta)$, is called a {\em spin subsurface push subgroup}. Before proceeding with the rest of the section, the reader may wish to review the notion of a {\em fundamental multitwist} defined in Section \ref{subsection:firstproperties}.

\begin{lemma}\label{lemma:spinpush}
Let $S' \subset S$ be a subsurface with some boundary component $\Delta$ satisfying $\phi(\Delta) = -1$. Then there is a finite-index subgroup $\widetilde \Pi(S', \Delta) \leqslant \Mod(S)[\phi] \cap \Pi(S', \Delta)$ characterized by the diagram given below, whose rows are short exact sequences:
\begin{equation}\label{equation:diagram}
\xymatrix{
1 \ar[r]	& \pair{T_\Delta^r} \ar[r] \ar[d]	& \tilde{\Pi}(S', \Delta) \ar[r] \ar[d]	& \pi_1(\overline{S'}) \ar[r] \ar@{=}[d]	&1\\
1 \ar[r]	& \pair{T_\Delta} \ar[r]			& \Pi(S', \Delta) \ar[r]				& \pi_1(\overline{S'})	\ar[r]		&1	.
}
\end{equation}
The subgroup $\widetilde \Pi(S', \Delta)$ contains all fundamental multitwists for pairs of pants $P \subset S'$ of the form $P = a \cup b \cup \Delta$.
\end{lemma}

\begin{proof}Following the discussion of Section \ref{subsection:birman}, there exists a ``geometric'' generating set for $\pi_1(UT\overline{S'}, \Delta)$ of the following form:
\begin{equation}\label{equation:utgens}
\pi_1(UT\overline{S'})= \pair{\tilde \alpha_1, \dots, \tilde \alpha_k, \zeta}.
\end{equation}
Here $\alpha_i$ is some simple closed curve on $\overline{S'}$ based at $\Delta$, and $\tilde \alpha_i$ denotes the Johnson lift to $\pi_1(UT\overline{S'})$. As before, $\zeta$ denotes the loop around the fiber. As an element of $\Mod(S')$, each $\tilde \alpha_i$ is of the form $T_{\alpha_{i,L}} T_{\alpha_{i,R}}$, where $\alpha_{i,L}$ denotes the curve on $S'$ lying to the left of $\alpha_i$ and $\alpha_{i,R}$ lies to the right. It follows that $P_i = \alpha_{i,L} \cup \alpha_{i,R} \cup \Delta$ forms a pair of pants on $S'$. Following Lemma \ref{lemma:fundtwist}, the fundamental multitwist
\[
T_{P_i} = T_{\alpha_{i,L}} T_{\alpha_{i,R}}^{-1} T_\Delta^{\phi(\alpha_{i,R})}
\]
lies in $\Mod(S)[\phi] \cap \Pi(S', \Delta)$. Embedding $\pi_1(UT\overline{S'})$ into $\Mod(S')$, the generating set of (\ref{equation:utgens}) can be replaced by the following generating set for $\Pi(S', \Delta)$:
\[
\Pi(S', \Delta) = \pair{T_{P_1}, \dots, T_{P_k}, T_\Delta}.
\]
Define
\[
\widetilde{\Pi}(S', \Delta) = \pair{T_{P_1}, \dots, T_{P_k}, T_\Delta^r}.
\]
By construction, $\widetilde{\Pi}(S', \Delta) \leqslant \Mod(S)[\phi]$. Under the projection $\Pi(S', \Delta) \to \pi_1(\overline{S'})$, the set $\{T_{P_i}\}$ maps onto a generating set for $\pi_1(\overline{S'})$. It follows that $\widetilde\Pi(S', \Delta)$ surjects onto $\pi_1(\overline{S'})$. As $T_\Delta^m \in \Mod(S)[\phi]$ if and only if $r \mid m$, it follows that $\widetilde \Pi(S', \Delta)$ is indeed characterized by the diagram (\ref{equation:diagram}) as claimed. 

For the second claim, let $P = a \cup b \cup \Delta$ be a pair of pants on $S'$. The curves $a,b$ are isotopic on $\overline{S'}$ and cobound an annulus containing the basepoint. It follows that $T_a T_b^{-1} \in \pi_1(\overline{S'})$. Via (\ref{equation:diagram}), there is some lift $T_a T_b^{-1} T_\Delta^k \in \widetilde \Pi(S',\Delta)$, and as $T_\Delta^r \in \widetilde \Pi(S', \Delta)$ as well, it follows that all fundamental multitwists for $P$ are elements of $\widetilde \Pi(S', \Delta)$ as claimed.
\end{proof}

For the purposes of this paper, we will most often be concerned with subsurface push subgroups for a special class of subsurfaces. Let $b \subset \Sigma_g$ be a nonseparating closed curve satisfying $\phi(b) = -1$. The boundary component $\Delta$ of $\Sigma_g \setminus\{b\}$ corresponding to the left side of $b$ satisfies $\phi(\Delta) = -1$, and to ease notation, we write $\widetilde{\Pi}(\Sigma_g \setminus\{b\})$ to refer to this spin subsurface push subgroup.

\subsection{Generating admissible twists}
We have arrived at the key result of the section.
\begin{proposition}\label{proposition:submakesT}
Let $\phi$ be a $\Z/r\Z$-valued spin structure on a closed surface $\Sigma_g$ for $g \ge 5$ and any integer $r$. Let $(a_0,a_1,b)$ be an ordered $3$-chain of curves with $\phi(a_0) = \phi(a_1) = 0$ and $\phi(b) = -1$. Let $H \leqslant \Mod(\Sigma_g)$ be a subgroup containing $T_{a_0}, T_{a_1}$ and the spin subsurface push group $\widetilde\Pi(\Sigma_g \setminus\{b\})$. Then $H$ contains $\mathcal T_\phi$.
\end{proposition}

The proof will require the preliminary Lemma \ref{lemma:creep}, for which we introduce some terminology. For a subgroup $H \leqslant \Mod(\Sigma_g)$, we say that a simple closed curve $a$ is an {\em $H$-curve} if $T_a \in H$. We also say that curves $a,b$ are {\em $H$-equivalent} if there exists some $f \in H$ with $f(a) = b$. If $a$ and $b = f(a)$ are $H$-equivalent and $\widetilde \Pi(\Sigma_g \setminus\{a\}) \leqslant H$, then also $\widetilde\Pi(\Sigma_g\setminus\{b\}) = f\widetilde \Pi(\Sigma_g\setminus\{a\}) f^{-1}$ is a subgroup of $H$.

The following lemma establishes some sufficient conditions for $H$-equivalence of curves. 
\begin{lemma}\label{lemma:creep}
Let $\Sigma_g$ be a surface of genus $g \ge 5$. Let $a_0,a_1,b$ be an ordered $3$-chain of curves with $\phi(a_0) = \phi(a_1) = 0$ and $\phi(b) = -1$. Let $H \leqslant \Mod(\Sigma_g)$ be a subgroup containing $T_{a_0}, T_{a_1}$ and $\widetilde\Pi(\Sigma_g \setminus \{b\})$. 
\begin{enumerate}
\item\label{item:b} Let $b'$ be an oriented curve satisfying $\phi(b') = -1$ such that $i(b, b') = 0$ and $i(a_1, b') = 1$. Then $b$ and $b'$ are $H$-equivalent. It follows that $\widetilde\Pi(\Sigma_g \setminus \{b'\}) \leqslant H$.
\item \label{item:a} Let $a'$ be any nonseparating curve satisfying $\phi(a')= 0$ such that $a'$ is disjoint from the configuration $a_0 \cup a_1 \cup b$. Then $a'$ is an $H$-curve.
\item \label{item:b2} Let $b'$ be any nonseparating curve satisfying $\phi(b') = -1$ such that $b'$ is disjoint from the configuration $a_0 \cup a_1 \cup b$. Then $b$ and $b'$ are $H$-equivalent, and hence $\widetilde\Pi(\Sigma_g \setminus \{b'\}) \leqslant H$.
\end{enumerate}
\end{lemma}

\begin{proof}
(\ref{item:b}): If $b = b'$ there is nothing to prove. Otherwise, given $a_1, b, b'$, we define a curve $b''$ as follows. Let $\epsilon$ be the portion of $a_1$ connecting the left side of $b$ to one of the sides of $b'$; then $b''$ is defined as the curve-arc sum $b'' := b +_\epsilon b'$. By construction $b\cup b' \cup b''$ bounds a pair of pants $P$ lying to the left of $b$, and $i(a_1, b'') = 0$. By Lemma \ref{lemma:curvearcsum}, there exists an orientation of $b''$ such that $\phi(b'') = -1$. This can be determined as follows: $b''$ is oriented with $P$ lying to the right if and only if $P$ lies to the left of $b'$. If $P$ lies to the left of $b'$, then the element $T_b T_{b'} T_{b''}^{-1}$ is a fundamental multitwist and hence an element of $\widetilde\Pi(\Sigma_g \setminus\{b\}) \leqslant H$. Otherwise, $T_b T_{b'}^{-1} T_{b''}$ is a fundamental multitwist.
In the first case, the braid relation implies that
\[
T_{a_1}(T_b T_{b'} T_{b''}^{-1})T_{a_1}(b) = b',
\]
while in the second case,
\[
T_{a_1}^{-1}(T_b T_{b'}^{-1} T_{b''})T_{a_1}(b) = b'.
\]
In either case, the indicated element lies in $H$, showing the $H$-equivalence between $b, b'$. \\

(\ref{item:a}): Let $\epsilon$ be an arc connecting $a_0$ to $a'$ that is disjoint from $a_1 \cup b$, and define $b' := a_0+_\epsilon a'$. It is possible that $b' = b$, but this will not pose any difficulty. Then $a_0 \cup a' \cup b'$ forms a pair of pants and $b'$ satisfies the intersection conditions $i(b,b') = 0$ and $i(a_1, b') = 1$. By the homological coherence property, $\phi(b') = -1$. By the second assertion of (\ref{item:b}), $\widetilde\Pi(\Sigma_g \setminus \{b'\}) \leqslant H$. As $a_0 \cup a' \cup b'$ forms a pair of pants, it follows that $T_{a_0} T_{a'}^{-1}$ is a fundamental multitwist, and so $T_{a_0}T_{a'}^{-1} \in \widetilde\Pi(\Sigma_g \setminus \{b'\}) \leqslant H$. As $T_{a_0} \in H$ by hypothesis, this shows that $T_{a'} \in H$ as desired.\\

(\ref{item:b2}): Given $b'$, Corollary \ref{corollary:intersections} implies that there exists an admissible curve $a'$ that is disjoint from $a_0 \cup a_1 \cup b$ and for which $i(a',b') = 1$. Corollary \ref{corollary:intersections} also establishes the existence of a curve $b''$, satsifying $\phi(b'') = -1$, with the following intersection properties:
\[
i(b, b'') = i(b', b'') = i(a_0, b'') = 0; \qquad i(a', b'') = i(a_1, b'') = 1.
\]
By (\ref{item:b}), $b$ and $b''$ are $H$-equivalent. By (\ref{item:a}), $a'$ is an $H$-curve, so that by (\ref{item:b}) again, $b''$ and $b'$ are $H$-equivalent, showing the result. 
\end{proof}

\begin{proof}{\em (of Proposition \ref{proposition:submakesT})} Let $a$ be any admissible curve. There is some genus $2$ subsurface $S' \cong \Sigma_{2,1}$ containing $a$, and there is also some genus $2$ subsurface $S \cong \Sigma_{2,1}$ that contains the curves $a_0,a_1, b$. By Lemma \ref{lemma:connected1}, there is a path $S_0 = S - S_1 - \dots - S_n = S'$ of subsurfaces homeomorphic to $\Sigma_{2,1}$ with boundary components $\partial S_i$ and $\partial S_{i+1}$ disjoint for $i = 1, \dots, n-1$, hence $S_i \cap S_{i+1} = \emptyset$ for $i = 1, \dots, n-1$.

For $i = 1, \dots, n$, let $a_{2i}$ be an admissible curve contained in $S_i$; we take $a_{2n} = a$. We claim that there exist curves $a_{2i+1}$ and $b_i$ on $S_i$ such that $a_{2i}, a_{2i+1}, b_i$ forms a chain, and $\phi(a_{2i+1}) = 0, \phi(b_i) = -1$. To see this, let $T \subset S_i$ be a subsurface of genus $1$ that does not contain $a_{2i}$. By Corollary \ref{corollary:admissible}, there is an admissible curve $a'$ contained in $T$. Let $\epsilon$ be an arc connecting $a_{2i}$ and $a'$; then $b_i:= a_{2i} +_\epsilon a'$ satisfies $\phi(b) = -1$ for a suitable choice of orientation. Let $c$ be any curve on $S_i$ such that $i(c,a_{2i}) = i(c, b_i) = 1$. Then $a_{2i+1} := T_{b_i}^{\phi(c)}(c)$ is admissible, and $a_{2i}, a_{2i+1}, b_i$ forms a chain as required.

We assume for the sake of induction that $a_{2i}, a_{2i+1}$ are $H$-curves and that $\widetilde\Pi(\Sigma_{g} \setminus \{b_i\}) \leqslant H$. Then by Lemma \ref{lemma:creep}.\ref{item:a}, also $a_{2i+2}, a_{2i+3}$ are $H$-curves, and $\widetilde\Pi(\Sigma_{g} \setminus \{b_{i+1}\}) \leqslant H$. The base case $i = 0$ holds by hypothesis, taking $b_0 = b$. The claim now follows by induction.
\end{proof}

\section{Networks}\label{section:network}
In this section we deduce Theorem \ref{theorem:networkgenset} from Proposition \ref{proposition:submakesT}. The key notion is that of a {\em network} of curves. In Section \ref{subsection:networkbasics}, we establish the basic theory of networks, and in Section \ref{subsection:networkproof}, we state and prove Theorem \ref{theorem:networkgenset}.  Departing from our conventions elsewhere in the paper, in this section we work with individual curves and {\em not} merely their isotopy classes.

\subsection{Networks and their basic theory}\label{subsection:networkbasics}

\begin{definition}\label{definition:network}
Let $S = \Sigma_{g,b}^n$ be a surface, viewed as a compact surface with marked points. A {\em network} on $S$ is any collection $\mathcal N = \{a_1,\dots, a_n\}$ of simple closed curves on $S$, disjoint from any marked points, such that $\#(a_i \cap a_j) \le 1$ for all pairs of curves $a_i, a_j \in \mathcal N$, and such that there are no triple intersections. A network $\mathcal N$ has an associated {\em intersection graph} $\Gamma_{\mathcal N}$, whose vertices correspond to curves $x \in \mathcal N$, with vertices $x,y$ adjacent if and only if $\#(x\cap y) = 1$.  A network is said to be {\em connected} if $\Gamma_{\mathcal N}$ is connected, and {\em arboreal} if $\Gamma_{\mathcal N}$ is a tree. A network is {\em filling} if 
\[
S \setminus \bigcup_{a \in \mathcal N} a
\]
is a disjoint union of disks and boundary-parallel annuli; each component is allowed to contain at most one marked point of $S$. 
\end{definition}

The data of a network encodes both an abstract finite set of curves as well as a topological subspace of the surface $S$. To avoid confusing these, let the symbol $\mathcal N$ denote this finite set, and let $\widehat{\mathcal N}$ denote the space. When $\mathcal N$ is arboreal, there is a simple generating set for $\pi_1(\widehat{\mathcal N})$. To describe it, endow $\widehat{\mathcal N}$ with the structure of a $CW$ complex, and let $\mathcal T$ be a spanning tree for this $CW$ complex.

\begin{lemma}\label{lemma:graphs}
Let $\mathcal N$ be an arboreal network. Then there is a 1--1 correspondence between the set of edges $\widehat{\mathcal N}\setminus \mathcal T$, and the set $\mathcal N$. 
\end{lemma}
\begin{proof}
Each edge of $\widehat{\mathcal N}$ is contained in a unique element of $\mathcal N$. For a given $a \in \mathcal N$, let $a_1, \dots, a_{n(a)}$ denote these edges, ordered so that adjacent edges are numbered consecutively. For each $a \in \mathcal N$, the sequence $a_1, \dots, a_{n(a)}$ forms a cycle in $\widehat{\mathcal N}$. Thus for each $a \in \mathcal N$, there is at least one edge $a_1$ (without loss of generality) that is not contained in $\mathcal T$. 

It remains to show that for each $a \in \mathcal N$, there is {\em exactly} one edge not contained in $\mathcal T$. Equivalently, we must show that the intersection $a \cap \mathcal T$ is connected as a topological space. The assumption that $\mathcal N$ is arboreal implies that $\Gamma_{\mathcal N}$ has the following property: let $\Gamma_{\mathcal N}(a)$ be the graph obtained from $\Gamma_{\mathcal N}$ by removing all edges incident to $a$. Then each vertex $b$ adjacent to $a$ in $\Gamma_{\mathcal N}$ determines a {\em distinct} component of $\Gamma_{\mathcal N}(a)$.

Let $v,w \in a$ be vertices of $\widehat{\mathcal N}$, and let $e_1, \dots, e_n$ be the unique geodesic path in $\mathcal T$ connecting $v$ to $w$. It suffices to show that this path is contained in $a$. If this is not the case, let $k_1$ (resp. $k_2$) be the minimal (resp. maximal) integer such that $e_{k_1}$ (resp. $e_{k_2}$) is not contained in $a$. Then exactly one vertex $v_1$ of $e_{k_1}$ (resp. $v_2$ of $e_{k_2}$) lies on $a$, and the other lies on some adjacent curve $b_1$ (resp. $b_2$). As $i(a,b_1) = i(a,b_2) =1$ and the path $e_1, \dots, e_n$ visits each vertex in $\widehat{\mathcal N}$ at most once, it follows that $b_1$ and $b_2$ are distinct elements of $\mathcal N$. As explained in the above paragraph, the arboreality assumption implies that every path in $\widehat{\mathcal N}$ connecting a point in $b_1$ to a point in $b_2$ must pass through $a$. Any such path must pass through $v_1$ and $v_2$: this shows that if the path $e_1, \dots, e_n$ enters $b_1$, it must pass through $v_1$ at least twice, contrary to assumption.
\end{proof}
Via Lemma \ref{lemma:graphs}, each $a \in \mathcal N$ determines a {\em unique} based loop $P(a)$ by following the unique path in $\mathcal T$ from the basepoint to $a$. Lemma \ref{lemma:pi1genset} below now follows by a standard application of the Seifert-Van Kampen theorem.

\begin{lemma}\label{lemma:pi1genset}
Let $\mathcal N$ be an arboreal network. Then $\pi_1(\widehat{\mathcal N})$ is generated by the set of loops $\{P(a) \mid a \in \mathcal N\}$. If $\mathcal N$ is moreover filling, then the map $\pi_1(\widehat{\mathcal N}) \to \pi_1(S)$ is a surjection, and so $\pi_1(S)$ is also generated by this collection of loops. 
\end{lemma}

$\pi_1(S)$ is a normal subgroup of $\Mod(S)$: if $\alpha \in \pi_1(S)$ is a mapping class corresponding to a based loop and $f \in \Mod(S)$ is arbitrary, then conjugation by $f$ takes $\alpha$ to the mapping class corresponding to the based loop $f(\alpha)$. In the context of the ``network presentation'' of $\pi_1(S)$ arising from the surjection $\pi_1(\widehat{\mathcal N}) \to \pi_1(S)$, this means that $\pi_1(S)$ has a very simple {\em normal} generating set as a subgroup of $\Mod(S)$, as the following makes precise.

\begin{lemma}\label{lemma:makepushes}
Let $\widehat{\mathcal N} \subset S$ be an arboreal filling network. Let $H \leqslant \Mod(S)$ be a subgroup containing $T_a$ for each $a \in \mathcal N$. If $H$ also contains $P(a_1) \in \pi_1(S)$ for some $a_1 \in \mathcal N$, then $H$ contains the entire point-pushing subgroup $\pi_1(S)$.  
\end{lemma}

\begin{proof}
As recorded in Lemma \ref{lemma:pi1genset}, $\pi_1(\widehat{\mathcal N})$, and hence also $\pi_1(S)$, is generated by the collection of elements $P(a)$ for $a \in \mathcal N$. We will proceed by induction. Define connected subnetworks 
\[\mathcal N_0 \subset \mathcal N_1 \subset \dots \subset \mathcal N_n = \mathcal N\] as follows: $\mathcal N_k$ consists of all those curves $a$ at a distance of at most $k$ from the base vertex $a_1 \in \Gamma_{\mathcal N}$ (viewing $\Gamma_{\mathcal N}$ as a metric space for which each edge has length $1$). We suppose that $\pi_1(\widehat{\mathcal N}_k) \leqslant H$; the base case $k = 0$ holds by hypothesis.

Let $a \in \mathcal N_{k+1}\setminus \mathcal N_k$ be arbitrary. Let $a' \in \mathcal N_k$ be adjacent to $a$. By the braid relation, 
\[
T_{a'} T_{a} (a') = a,
\]
and hence $P(a) = (T_{a'} T_{a}) P(a') (T_{a'} T_{a}) ^{-1} \in H$. This completes the inductive step.
\end{proof}

\subsection{Network generating sets for \boldmath$\mathcal T_\phi$}\label{subsection:networkproof} Having established some of the basic theory of networks, we can now formulate and prove the key technical result of the paper. For hypotheses \eqref{item:config} and \eqref{item:intersect}, the reader may wish to consult Figure \ref{figure:dnrel} and the surrounding discussion of the $D_n$ relation (Proposition \ref{proposition:dnrel}) and the configuration $\mathscr D_n$. For an example of a network satisfying the hypotheses of Theorem \ref{theorem:networkgenset}, see Figure \ref{figure:example}.
\begin{theorem}\label{theorem:networkgenset}
Let $\phi$ be a $\Z/r\Z$-valued spin structure on a closed surface $\Sigma_g$, with $1 \le r < g - 1$. Let $\mathcal N = \{a_n\}$ be a connected filling network of curves on $\Sigma_g$ with the following properties:
\begin{enumerate}
\item \label{item:zero} Every element $a_n$ is admissible,
\item \label{item:config} There is a collection $a_1, \dots, a_{2r + 4}$ of elements of $\mathcal N$ such that $a_1, \dots, a_{2r + 3}$ are arranged in the configuration of the curves $\mathscr D_{2r+3}$ of the $D_{2 r + 3}$ relation, and $a_{2r + 4}$ corresponds to the boundary component $\Delta_1$ associated to the subconfiguration $\mathscr D_{2 r + 2}$. 
\item \label{item:intersect} Let $b \subset \Sigma_g$ correspond to the curve $\Delta_0$ of the $D_{2 r + 3}$ relation, as appearing in Figure \ref{figure:dnrel}. Then $\mathcal N$ must contain some curve $d$ with $i(d,b) = 1$. 
\item \label{item:restrict} Let $\mathcal N' \subset \mathcal N$ be the subnetwork consisting of all curves in $\mathcal N$ disjoint from $b$. Then $\mathcal N'$ must be an arboreal filling network for $\Sigma_g \setminus \{b\}$.
\end{enumerate}
If $g \ge 5$, then $\pair{T_{a_i},\ a_i \in \mathcal N}$ contains the admissible subgroup $\mathcal T_\phi$. 

Moreover, if $r$ is odd, then
\[
 \pair{T_{a_i},\ a_i \in \mathcal N} = \Mod(\Sigma_g)[\phi].
\]
If $r=2d$ is even and $g \ge g(r)$ for the function $g(r)$ of Definition \ref{definition:gd}, then $\pair{T_{a_i},\ a_i \in \mathcal N}$ is a subgroup of finite index in $\Mod(\Sigma_g)[\phi]$.
\end{theorem}

\begin{proof}
Define
\[
H = \pair{T_{a_i},\ a_i \in \mathcal N}.
\]
By hypothesis (\ref{item:zero}), $H \leqslant \mathcal T_\phi$. We establish the opposite containment $\mathcal T_\phi \leqslant H$. The remaining assertions in the statement of Theorem \ref{theorem:networkgenset} follow by an appeal to Proposition \ref{proposition:twistgen} or Proposition \ref{proposition:twistgeneven} as appropriate. The containment $\mathcal T_\phi \leqslant H$ will follow from Proposition \ref{proposition:submakesT}. To see that the hypotheses of Proposition \ref{proposition:submakesT} are satisfied by $H$, it is necessary to establish a containment $\widetilde \Pi(\Sigma_g \setminus\{b\})\leqslant H$, and to find suitable curves corresponding to $a_0, a_1$ in the statement of Proposition \ref{proposition:submakesT}.

Consider the curves $\{a_1, \dots, a_{2r + 4}\} \subset \mathcal N$ corresponding to $\mathscr D_{2 r + 3} \cup \{\Delta_1\}$ as in Corollary \ref{corollary:dn}, as posited by hypothesis (\ref{item:config}). Without loss of generality, assume that $a_1, a_2, a_3 \in \mathcal N$ correspond to the curves $a, c_1, a'$ of $\mathscr D_{2 r +3}$, so that $b \subset \Sigma_g$ is one of the boundary components of the chain $a_1, a_2, a_3$. Let $d$ be the curve with $i(d,b) = 1$ posited by hypothesis (\ref{item:intersect}), and let $P$ be the pair of pants bounded by $a_1, a_3, b$.  The intersection $d \cap P$ must be a single arc, since $d \in \mathcal N$ and so $\#(d\cap a_1) \le 1$ and $\#(d \cap a_3) \le 1$. Without loss of generality, assume $\#(d\cap a_1) = 1$ and $\#(d \cap a_3) = 0$. Then the $3$-chain $a_1, d, b$ on $\Sigma_g$ corresponds to the $3$-chain $a_0, a_1, b$ of Proposition \ref{proposition:submakesT}, since $\phi(b) = -1$ by the homological coherence property. By assumption, $T_d, T_{a_1} \in H$, so it remains only to establish $\widetilde \Pi(\Sigma_g \setminus\{b\})\leqslant H$.

By hypothesis (\ref{item:restrict}), the restriction $\mathcal N'$ determines an arboreal filling network on $\Sigma_g\setminus\{b\}$. The same is therefore true on the surface $\overline{\Sigma_g \setminus \{b\}}$ obtained by filling in the boundary component corresponding to the left side of $b$ (where $b$ is oriented so that $a_1, a_3$ lie to the left). The surface $\overline{\Sigma_g \setminus \{b\}}$ is connected, since the hypothesis $i(d,b) =1$ implies that $d$ is nonseparating. We treat $\overline{\Sigma_g \setminus \{b\}}$ as a surface $\Sigma_{g-1, 1}^1$, with the marked point corresponding to the filled-in left side of $b$. 

We will show that $\widetilde \Pi(\Sigma_g \setminus \{b\}) \leqslant H$ by appealing to Lemma \ref{lemma:spinpush}. We must therefore show that  $T_b^r$ is an element of $H$, and show that the image of the map 
\[
\pair{T_a, a \in \mathcal N'} \to \Mod(\overline{\Sigma_g \setminus \{b\}})
\]
contains the point-pushing subgroup $\pi_1(\overline{\Sigma_g \setminus \{b\}})$. Applying Corollary \ref{corollary:dn}.2, we obtain $T_b^r \in H$. To exhibit $\pi_1(\overline{\Sigma_g \setminus \{b\}})$, we will appeal to Lemma \ref{lemma:makepushes}.  The element $T_{a_1} T_{a_3}^{-1}$ corresponds to an element $P(a_1) \in \pi_1(\overline{\Sigma_g \setminus \{b\}})$. By Lemma \ref{lemma:makepushes}, it follows that the entire point-pushing subgroup $\pi_1(\overline{\Sigma_g \setminus \{b\}})$ is contained in the subgroup $\pair{T_a, a \in \mathcal N'} \leqslant H$. 
\end{proof}

\section{Linear systems in toric surfaces}\label{section:CL}
The purpose of this section is to give a minimal account of the work of  Cr\'etois--Lang in \cite{CL}. We do not attempt to give a detailed summary of the theory of toric surfaces; the interested reader is referred to \cite{CL} and the references therein. 

Consider the integer lattice $\Z^2 \subset \R^2$. A {\em lattice polygon} $\Delta$ is the convex hull of a finite collection $\{v_1, \dots, v_n\}$ of $n \ge 3$ elements $v_i \in \Z^2$, not all collinear. Given a polygon $\Delta$ which contains at least one lattice point in the interior $\operatorname{int}(\Delta)$, the {\em adjoint polygon} $\Delta_a$ is defined to be the convex hull of $\operatorname{int}(\Delta) \cap \Z^2$. 

The following proposition is a concise summary of the correspondence between line bundles on toric surfaces and polygons. For details, see \cite[Section 3]{CL}. In item (\ref{item:defined}), a {\em unimodular transformation} of $\R^2$ is an affine map $A: \R^2 \to \R^2$ (necessarily invertible) such that $A \Z^2 = \Z^2$.

\begin{proposition}\label{proposition:polygonfacts}
Let $X$ be a smooth toric surface. 
\begin{enumerate}
\item\label{item:defined} Associated to any nef line bundle $\mathcal L$ on $X$ is a convex lattice polygon $\Delta_{\mathcal L}$, well-defined up to unimodular transformations.
\item\label{item:dilates} If $\mathcal L$ is nef, then the roots of $\mathcal L$ (i.e. the line bundles $\mathcal S$ for which $n\mathcal S = \mathcal L$ for some integer $n$) are in correspondence with the dilates $\frac{1}{n} \Delta_{\mathcal L}$ for which $\frac{1}{n} \Delta_{\mathcal L}$ is a {\em lattice} polygon. 
\item Suppose that $\mathcal L$ is ample and that $\operatorname{int}(\Delta_{\mathcal L}) \cap \Z^2$ is nonempty. Then the adjoint line bundle $\mathcal L \otimes K_X$ is nef, and $\Delta_{\mathcal L \otimes K_X} = (\Delta_{\mathcal L})_a$.
\item Let $\mathcal L$ be ample. The genus $g(\mathcal L)$ of a smooth $C \in \abs{\mathcal L}$ is given by the formula
\[
g(\mathcal L) = \# (\operatorname{int}(\Delta_{\mathcal L}) \cap \Z^2) = \#((\Delta_{\mathcal L})_a \cap \Z^2).
\]
\item \label{item:hyperelliptic} Let $\mathcal L$ be ample. A generic fiber $C \in \abs{\mathcal L}$ is hyperelliptic if and only if $(\Delta_{\mathcal{L}})_a$ is a line segment.
\end{enumerate}
\end{proposition}

The following proposition indicates the connection between the divisibility properties of $\mathcal L \otimes K_X$ as an element of $\Pic(X)$ (or after Proposition \ref{proposition:polygonfacts}.\ref{item:dilates}, the divisibility of $(\Delta_\mathcal L)_a$), and the presence of invariant higher spin structures. It is a folklore theorem; see \cite[Theorem 1.1]{saltermonodromy} and \cite[Proposition 2.7]{CL} for written accounts. 

\begin{proposition}\label{proposition:invtspin}
Let $\mathcal L$ be an ample line bundle on a smooth toric surface $X$. For any $r$ such that the adjoint line bundle $\mathcal L \otimes K_X$ admits a $r^{th}$ root in $\Pic(X)$, there exists a (unique) $\Z/r \Z$-valued spin structure $\phi$ preserved by the monodromy $\mu_\mathcal L$:
\[
\Gamma_\mathcal L \leqslant \Mod(\Sigma_{g(\mathcal L)})[\phi].
\]
\end{proposition}

Proposition \ref{proposition:polygonfacts} suggests that it might be profitable to ``model'' a smooth $C \in \abs{\mathcal L}$ on the lattice polygon $\Delta_\mathcal L$. 
\begin{construction}[Inflation procedure]\label{construction:inflation}
Let $\Delta$ be a lattice polygon. Let $B(r,x)$ denote the open ball of radius $r$ centered at $x \in \R^2$. Define the surface with boundary
\[
\Delta^\circ := \Delta \setminus \bigcup_{v \in \operatorname{int}(\Delta) \cap \Z^2} B(v, 1/4).
\]
The {\em inflation} of $\Delta$ is the surface $C_{\Delta}$ obtained as the double of $\Delta^\circ$ along its boundary. It is a closed oriented surface of genus $g =  \# (\operatorname{int}(\Delta) \cap \Z^2)$. In particular, for $\Delta = \Delta_{\mathcal L}$ for some ample $\mathcal L$, the inflation $C_\Delta$ has genus $g(\mathcal L)$. See Figure \ref{figure:example} for the example of $\mathcal O(6)$ on $\CP^2$.
\end{construction}

The first indication of the utility of the inflation procedure is provided by the following theorem of  Cr\'etois--Lang. For an inflation $C_\Delta$, define an {\em $A$-curve} to be any simple closed curve on $C_\Delta$ that corresponds to the circle of radius $1/4$ centered at an interior lattice point of $\Delta$.

\begin{theorem}[\cite{CL}, Theorem 3]\label{theorem:Acurves}
Let $\mathcal L$ be an ample line bundle on a smooth toric surface $X$. There is a homeomorphism $f: C_0 \to C_{\Delta_{\mathcal L}}$ identifying a smooth $C_0 \in \abs{\mathcal L}$ with $C_{\Delta_{\mathcal L}}$, such that every $A$-curve $a \subset C_{\Delta_{\mathcal L}}$ is a vanishing cycle, and
\[
T_a \in \Gamma_{\mathcal L}.
\]
\end{theorem}

 Cr\'etois--Lang also determine a second family of elements of $\Gamma_{\mathcal L}$ arising from the combinatorics of $\Delta$. A {\em primitive integer segment} is a line segment $\sigma \subset \R^2$ whose endpoints lie on $\Z^2$ and whose interior is disjoint from $\Z^2$. A primitive integer segment determines a line in $\R^2$ in the obvious way. When a lattice polygon $\Delta$ is fixed, it will be understood that a primitive integer segment connects lattice points $v,w \in \Delta \cap \Z^2$, and such that $v$ and $w$ do not lie along the same edge of $\Delta$. Under the inflation procedure, a primitive integer segment corresponds to a simple closed curve. For a primitive integer segment $\sigma$, we write $T_\sigma$ for the corresponding Dehn twist.

Suppose that $\Delta$ is a lattice polygon, let $d \ge 1$ be an integer. We say that $\Delta$ is {\em divisible by $d$} if the dilate $\frac{1}{d} \Delta$ is again a lattice polygon. If $\Delta$ is divisible by $d$, then after translating $\Delta$ so that one vertex lies in the sublattice $d\Z^2$, the remaining vertices do also. We write
\[
\Delta(d) := \Delta \cap d \Z^2,
\]
relative to any such embedding. The following is a combination of Propositions 7.13 and 7.16 of \cite{CL}.

\begin{theorem}[Cr\'etois--Lang]\label{theorem:piseg}
Let $\mathcal L$ be an ample line bundle on a smooth toric surface $X$. Suppose that the adjoint polygon $(\Delta_\mathcal L)_a$ is divisible by $d$. Suppose that $\sigma$ is a primitive integer segment such that the line it generates intersects $(\Delta_\mathcal L)_a(d)$. Then, with respect to the identification $f: C_0 \to C_{\Delta_\mathcal L}$ of Theorem \ref{theorem:Acurves}, we have that $\sigma$ is a vanishing cycle and $T_\sigma \in \Gamma_\mathcal L$. 
\end{theorem}

Taken together, Theorems \ref{theorem:Acurves} and \ref{theorem:piseg} produce a large family of Dehn twists in $\Gamma_\mathcal L$. In the next section, we will see that they provide sufficiently many elements to satisfy the hypotheses of Theorem \ref{theorem:networkgenset}, which will lead to a proof of Theorem \ref{theorem:toric}.

\section{Proof of Theorem \ref{theorem:toric}}\label{section:proof}
Fix a toric surface $X$ and an ample line bundle $\mathcal L$. This determines the polygons $\Delta_\mathcal L$ and $(\Delta_\mathcal L)_a$, as well as the monodromy group $\Gamma_\mathcal L$. For convenience, we will drop reference to $\mathcal L$ from the notation, and speak of $\Delta, \Delta_a, \Gamma$, etc. We also shorten notation for the inflation curve $C_{\Delta_\mathcal L}$, and refer simply to $C$ instead.

By hypothesis, $r$ is the highest root of the line bundle $\mathcal L \otimes K_X$. Proposition \ref{proposition:polygonfacts}.\ref{item:dilates} implies that $\Delta_a$ is $r$-divisible. Our first objective is to find a network $\mathcal N$ satisfying the hypotheses of Theorem \ref{theorem:networkgenset}. This will show all but the last assertion of Theorem \ref{theorem:toric}. Once this has been accomplished, we will see that the answer to Question \ref{question:donaldson} readily follows. 

\para{Genus hypotheses} We first address the genus assumptions of Theorem \ref{theorem:networkgenset}. Recalling that $\Delta_a$ is assumed to be $r$-divisible, a calculation using Pick's formula implies that for $r > 1$,
\[
g \ge \frac{(r+1)(r+2)}{2}.
\]
This shows that $g \ge 5$ for all $r>1$ and that $g \ge r + 1$ for $r = 2d$ even. For $r = 2d = 4$, this gives $g \ge 15$, and for $r = 2d = 8$ this gives $g \ge 45$. In all cases, the hypothesis $g \ge g(r)$ of Theorem \ref{theorem:networkgenset} holds.

The remaining assumption to be addressed is the requirement that $r < g -1$. As noted in Remark \ref{remark:admits}, $r$ must divide $2g -2$, so we must only show that the cases $r = 2g-2$ and $r = g-1$ do not occur in the study of linear systems on toric surfaces. Suppose first that $r = 2g-2$. This implies that the adjoint polygon $\Delta_a$ contains precisely $g$ lattice points, but is also $2g-2$-divisible. This is an absurdity: let $e$ be an edge of the lattice polygon $\frac{1}{2g-2} \Delta_a$; then the dilate $(2g-2)e$ contains at least $2g-1 > g$ lattice points. In the case $r = g -1$, a similar analysis shows that in fact $\Delta_a$ must {\em equal} the $g-1$-fold dilation of a primitive integer segment. By Proposition \ref{proposition:polygonfacts}.\ref{item:hyperelliptic}, this implies that the general fiber of the linear system is hyperelliptic, which we have excluded from consideration.

\para{Constructing the network $\bm{\mathcal N}$} Recall that according to Theorem \ref{theorem:Acurves}, each integer point $v \in \Delta_a \cap \Z^2$ determines a vanishing cycle in $\Gamma$; we introduce the notation $A(v)$ to refer to the curve associated to $v$. When we have a specific identification of $\Delta$ with a lattice polygon, we will use the notation $A(x,y)$ to refer to the $A$-curve at the integer point $(x,y)$. Similarly, given a primitive integer segment $\sigma$, we let $B(\sigma)$ denote the associated simple closed curve on $C$. When $\Delta$ is identified with a lattice polygon, we write $B((x,y), (z,w))$ for the $B$-curve associated to the primitive integer segment connecting $(x,y)$ and $(z,w)$. We refer to these as $A$-curves and $B$-curves, respectively. 

To define the network $\mathcal N$, it will be useful to introduce some terminology. Let $\sigma$ be a primitive integer segment, and let $L(\sigma)$ be the line determined by $\sigma$. For an integer point $v$, we say that $\sigma$ {\em points towards $v$} if $v \in L(\sigma)$. We also introduce the notion of a {\em $\kappa$-standard embedding}. Let $\kappa$ be a vertex of $\Delta_a$. A $\kappa$-standard embedding is an embedding of $\Delta$ into $\R^2$ such that $\kappa$ corresponds to $(0,0)$ and such that the edges of $\Delta_a$ incident to $\kappa$ lie along the $x$ and $y$ axes. Any embedding $\Delta \subset \R^2$ can be made $\kappa$-standard by applying a suitable unimodular transformation. Following Proposition \ref{proposition:polygonfacts}.\ref{item:defined}, we are free to apply unimodular transformations as needed.

Let $\kappa$ be a vertex of $\Delta_a$. Let $\mathcal N$ be the network consisting of the following curves:
\begin{enumerate}
\item All $A$-curves.
\item The curve $B(\sigma)$, where $\sigma$ is defined as follows. Let $\kappa'$ be a vertex of $\Delta_a$ adjacent to $\kappa$. Let $e'$ be the edge of $\Delta_a$ containing $\kappa'$ and not containing $\kappa$, and let $w \in \partial \Delta_a$ be the integer point lying on $e'$ that is connected to $\kappa'$ by a primitive integer segment $\sigma$. 
\item The curve $B(\tau)$ defined as follows. Under a $\kappa$-standard embedding of $\Delta$, necessarily $(0,-1) \in \partial \Delta$. Since $\Delta_a$ is assumed to be $d$-divisible, the edge of $\Delta_a$ lying along the $x$-axis extends at least as far as $(d,0)$. We take $\tau$ to be the primitive integer segment identified with $B((d,0),(0,-1))$ in this embedding of $\Delta$.
\item All $B$-curves associated to primitive integer segments pointing towards $\kappa$, but such that the associated line does not pass through the interior of the segments $\sigma$ or $\tau$ or the primitive integer segment connecting $(-1,1)$ to $(0,1)$.
\end{enumerate}
See Figure \ref{figure:example} for a picture of $\mathcal N$ in the case of the line bundle $\mathcal O(6)$ on $\CP^2$.
\begin{remark}
As can be seen in Figure \ref{figure:example}, certain elements of $\mathcal N$ are mutually isotopic. This harmless excess is introduced only to make the definition of $\mathcal N$ more tidy.
\end{remark}

In anticipation of an appeal to Theorem \ref{theorem:networkgenset}, we also define the subnetwork
\begin{equation}\label{equation:nprime}
\mathcal N' := \mathcal N \setminus \{A(0,1)\}.
\end{equation}

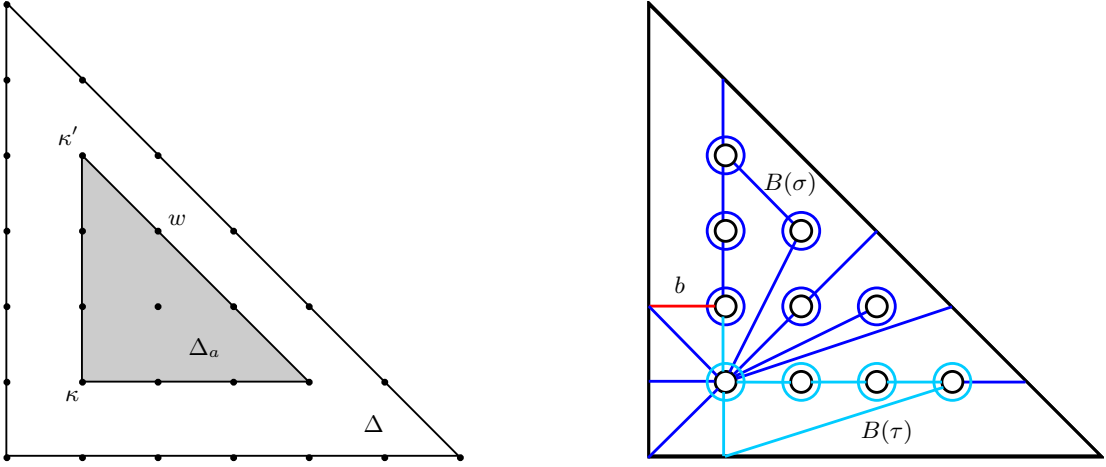
\begin{figure}
\labellist
\small
\pinlabel $\Delta$ [bl] at 150 10
\pinlabel $\Delta_a$ [br] at 94.4 41.6
\pinlabel $\kappa$ [tr] at 34.4 32
\pinlabel $\kappa'$ [br] at 34.4 131.4
\pinlabel $w$ [bl] at 67.4 97.8
\pinlabel $B(\sigma)$ [bl] at 319.4 109.8
\pinlabel $B(\tau)$ [tl] at 360.6 19.8
\pinlabel $b$ [b] at 286.4 68
\endlabellist
\includegraphics[scale = 0.9]{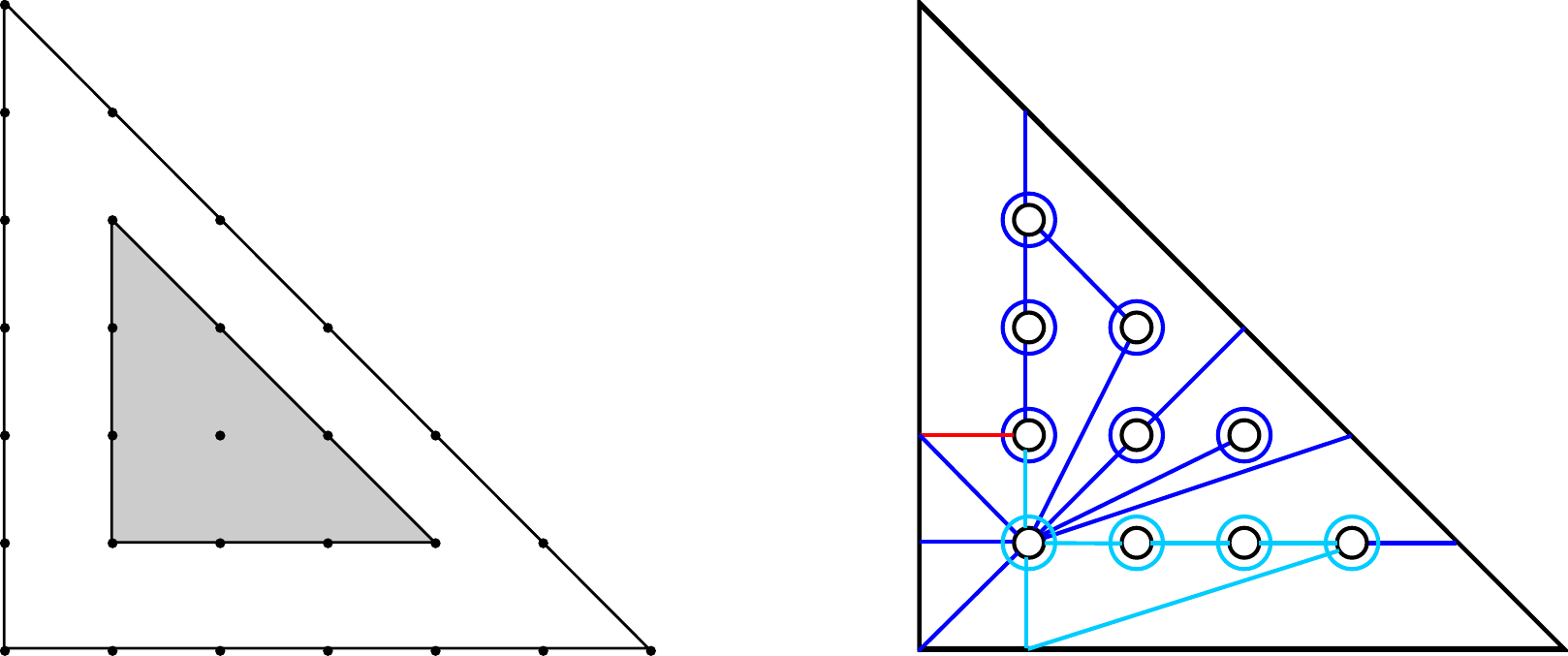}
\caption{Example: $(X, \mathcal L) = (\CP^2, \mathcal O(6))$; here $r = 3$. Left: the lattice polygons $\Delta$ and $\Delta_a$ (shaded). Right: the inflation construction, and the network $\mathcal N$, depicted in both shades of blue. The curves $a_1, \dots, a_{10}$ of Theorem \ref{theorem:networkgenset}.\ref{item:config} and \ref{theorem:networkgenset}.\ref{item:intersect} are shown in the light shade of blue. Note the curve $b$ (in red) is not part of the network, but does correspond to the curve $b$ of Theorem \ref{theorem:networkgenset}.\ref{item:intersect}. }
\label{figure:example}
\end{figure}

\para{First properties of $\bm{\mathcal N}$} We first claim that $\mathcal N$ is a network. Indeed, all $A$-curves are mutually disjoint. The set of primitive integer segments under consideration meet only at integer points in $\Delta_a$, and hence the associated $B$-curves are also mutually disjoint. Suppose $\sigma$ has endpoints $v,w$. Then $i(A(v), B(\sigma)) = i(A(w), B(\sigma)) = 1$, and $i(A(u), B(\sigma)) = 0$ for any other integer point $u$. Thus $\mathcal N$ is a network.

Indeed, $\mathcal N$ is a connected network, as follows from the description of $\Gamma_{\mathcal N}$ and $\Gamma_{\mathcal N'}$ given below.
\begin{lemma}\label{lemma:Nprops}
The graph $\Gamma_{\mathcal N}$ has the homotopy type of $S^1$, and $\Gamma_{\mathcal N'}$ is a tree, i.e. $\mathcal N'$ is arboreal.
\end{lemma}
\begin{proof}
We first establish that $\Gamma_{\mathcal N}$ is connected. It suffices to show that every $c \in \mathcal N$ is connected to $A(\kappa)$. We first consider the case of an $A$-curve $A(v)$. If $v \in \Delta_a$ is some other integer point, there is a line segment $L$ connecting $v$ to $\kappa$. This decomposes as a union of primitive integer segments $\sigma_i$ based at the integer points $v_j$ lying on $L$. Each such segment determines a $B$-curve in $\mathcal N$, and it is clear that there is a path from $A(v)$ to $A(\kappa)$ alternating between $B(\sigma_i)$ and $A(v_j)$. The argument for a $B$-curve (including the exceptional elements $B(\sigma)$ and $B(\tau)$) is similarly straightforward.

We next claim that the subnetwork 
\[
\mathcal N'' := \mathcal N \setminus{B(\sigma)}
\]
is arboreal. The curves $B(\sigma)$ and $B(\tau)$ are the only $B$-curves in $\mathcal N$ that do not lie on a line passing through $\kappa$. Thus the network consisting only of curves of type (1) and (4) is arboreal by construction. As $B(\tau)$ intersects only $A(d,0)$, this shows that the network consisting of curves of type (1),(3), and (4) is also arboreal, but this network is $\mathcal N ''$ by definition. 

The curve $B(\sigma)$ intersects only the $A$-curves $A(\kappa')$ and $A(w)$. Thus $\Gamma_{\mathcal N}$ is obtained from the tree $\Gamma_{\mathcal N''}$ by adding one new vertex that is connected to two edges, so that $\Gamma_{\mathcal N} \simeq S^1$ as claimed. 

It will follow from this that $\mathcal N'$ is also arboreal. The path in $\Gamma_{\mathcal N''}$ connecting $\kappa'$ to $w$ follows the $y$-axis down to $\kappa$, then proceeds out along the line connecting $\kappa$ to $w$; in particular, it passes through the vertex $(0,1)$. Thus, removing $A(0,1)$ to create the network $\mathcal N'$ removes the single circuit in $\Gamma_{\mathcal N}$, so that $\Gamma_{\mathcal N'}$ is a tree as claimed.
\end{proof}

We claim that $\mathcal N$ is filling. This will be established in the next two lemmas. Recall the definition of $\Delta^\circ$ from the definition of the inflation procedure in Construction \ref{construction:inflation}. 

\begin{lemma}\label{lemma:simpcon}
Let $S \subset \Delta$ denote the union of all primitive integer segments associated to $B$-curves in $\mathcal N$. Then
\begin{enumerate}
\item Each component of $\Delta^\circ \setminus S$ is simply-connected.
\item For each component $D$ of $\Delta^\circ \setminus S$, the intersection $D \cap \partial \Delta$ has at most one component. 
\end{enumerate}
\end{lemma}
\begin{proof}
We begin by observing that there are homotopy equivalences $\Delta^\circ \simeq \Delta \setminus (\Delta_a \cap \Z^2)$ and $\Delta^\circ \setminus S \simeq \Delta \setminus (S \cup (\Delta_a \cap \Z^2))$. It will be tidier to work with this latter space, and so we formulate our arguments in this setting. 

Embed $\Delta$ into $\R^2$ and consider $S$ as a planar graph contained in $\Delta$. Basic properties of convexity imply that for any integer point $v \in \Z^2 \cap \Delta_a$, the line connecting $v$ and $\kappa$ does not intersect either of $B(\sigma)$ or $B(\tau)$. Hence this line determines a union of primitive integer segments in $\mathcal N$, and upon the removal of these segments over all $v$, there is an equality 
\[
\Delta \setminus (S \cup (\Z^2 \cap \Delta_a)) = \Delta \setminus S.
\]
To prove (1), it therefore suffices to show that $H_1(\Delta \setminus S; \Z) = 0$. There is a map of pairs $f: (\Delta, \partial \Delta) \to (S^2, *)$, where $*\in S^2$ is an arbitrary basepoint. $f$ induces a homeomorphism 
\[
f: \Delta \setminus \partial \Delta \to S^2 \setminus \{*\}.
\]
Since the segment $B(\tau)$ (among many others) intersects $\partial \Delta$, it follows that $f$ induces a homotopy equivalence 
\[
f: \Delta \setminus S \to S^2 \setminus f(S), 
\]
and hence there is an isomorphism
\[
f_*: \tilde H_1(\Delta\setminus S; \Z) \to \tilde H_1(S^2 \setminus f(S); \Z).
\]
 By Alexander duality, $\tilde H_1(S^2 \setminus f(S);\Z) \cong \tilde H^0(f(S);\Z) = 0$, the latter holding because $S$ is connected by construction. This proves (1). 
 
 For (2), consider the subconfiguration $S' \subset S$ consisting of all primitive integer segments lying along a line connecting $\kappa$ to any integer point $v \in \partial \Delta$. This provides a subdivision of $\Delta$ into convex sets, each of which has a vertex at $\kappa$. Convexity then implies that each component $D_i'$ of $\Delta \setminus S'$ intersects $\partial \Delta$ in at most one component. The subdivision of $\Delta$ induced by $S$ is a refinement of that induced by $S'$. Since all segments in $S$ that intersect $\partial \Delta$ are elements of $S'$, there is an equality
 \[
 \partial \Delta \setminus (S \cap \partial \Delta) = \partial \Delta \setminus (S' \cap \partial \Delta).
 \]
 Thus each component of $\partial \Delta \setminus (S \cap \partial \Delta)$ corresponds to a {\em distinct} component of $\Delta \setminus S$, and (2) follows. 
\end{proof}

\begin{lemma}\label{lemma:simpcon2}
Each component of $C \setminus \mathcal N$ is simply-connected, i.e. $\mathcal N$ is filling.
\end{lemma}
\begin{proof}
By construction, the ``deflation'' map $p: C \to \Delta$ takes components of $C\setminus \mathcal N$ to components of $\Delta \setminus S$, where $S$ continues to denote the union of all primitive integer segments associated to $B$-curves in $\mathcal N$. This map on components is at most 2-to-1, and is {\em exactly} 2-to-1 in the case where the component $D$ of $\Delta \setminus S$ does not contain a lattice point in its interior and does not intersect $\partial \Delta$. In this 2-to-1 case, each component $\tilde D_1, \tilde D_2$ of $C \setminus \mathcal N$ is mapped homeomorphically by $p$ onto the component of $\Delta^\circ \setminus S$ corresponding to $D$. Since $D$ is assumed to contain no lattice points in the interior, it follows that the corresponding component of $\Delta^\circ \setminus S$ is simply-connected, and hence $\tilde D_1, \tilde D_2$ are as well.

Lemma \ref{lemma:simpcon}.1 implies that no component of $\Delta \setminus S$ contains an interior lattice point, and so it remains only to be seen that every component of $C \setminus \mathcal N$ that corresponds to a component of $\Delta \setminus S$ intersecting $\partial \Delta$ is simply-connected. Let $\tilde D\subset C\setminus \mathcal N$ be such a component, and let $D$ be the corresponding component of $\Delta^\circ \setminus S$. Observe that $\tilde D$ is constructed by attaching two copies of $D$ along $D \cap \partial \Delta$. It follows that $\tilde D$ is simply-connected if and only if $D \cap \partial \Delta$ is connected. The result now follows by Lemma \ref{lemma:simpcon}.2.
\end{proof}

\para{Applicability of Theorem \ref{theorem:networkgenset}} It remains to verify the properties $(1)-(4)$ of Theorem \ref{theorem:networkgenset}. By Theorem \ref{theorem:Acurves}, for an $A$-curve $A(v)$, the associated Dehn twist $T_{A(v)} \in \Gamma$. By Theorem \ref{theorem:piseg}, any curve $B(\xi) \in \mathcal N$ arising from a primitive integer segment $\xi$ also satisfies $T_{B(\xi)} \in \Gamma$. It follows from the definitions that any curve $c$ in any connected network is necessarily non-separating. For a nonseparating curve $c \subset C$, the Dehn twist $T_c \in \Gamma$ only if the associated spin structure satisfies $\phi(c) = 0$. Hence $(1)$ holds. 

For $(2)$, we take a $\kappa$-standard embedding of $\Delta$. It is now easy to find a collection of curves $\mathscr{S}_{2r +4}$ determining the configuration $\mathscr D_{2 r + 3} \cup \{\Delta_1\}$ of Corollary \ref{corollary:dn}. We take $a= B((0,0),(0,-1))$ and $a' = B((0,0),(0,1))$. Since $\Delta_a$ is assumed to be $r$-divisible, the edge of $\Delta_a$ lying along the $x$-axis extends at least as far as $(r,0)$. For $1 \le k \le r+1$, we can therefore take $c_{2k-1}$ to be $A(0,k-1)$, and for $1 \le k \le r$, we take $c_{2k}$ to be $B((k-1,0),(k,0))$. We take $a_{r+1}= B(\tau) = B((r,0),(0,-1))$. The segments connecting $(0,-1),(0,0), (1,0), \dots, (d,0), (0,-1)$ separate $\Delta$ into two components, hence under the inflation procedure, the associated $B$-curves separate $C$. From the construction it is clear that the curves bound a subsurface of genus $0$ with $r+2$ boundary components, as required for the configuration $\mathscr D_{2r+3} \cup \{\Delta_1\}$ of Corollary \ref{corollary:dn}.

For (3), we observe that from the construction, the $\Delta_0$ curve of the configuration $\mathscr D_{2r+3}$ corresponds to $b:= B((-1,1),(0,1))$ on $\Delta$. One sees that $A(0,1)$ intersects this curve, and is an element of $\mathcal N$ as needed. 

For (4), we begin by observing that {\em only} the element $A(0,1) \in \mathcal N$ intersects $B((-1,1),(0,1))$. Enumerate the components of $C \setminus \mathcal N$ as $\{D_i\}$. We claim that there are exactly three disks $D_1, D_2, D_3$ in $C \setminus \mathcal N$ with boundary lying on $A(0,1)$, and that $b \subset \overline{D_1} \subset C$. Indeed, using the notation of item (2) in the definition of $\mathcal N$, the disks $D_2$ and $D_3$ arise via inflation from the component of $\Delta \setminus S$ bounded by the triangle formed by $e', \sigma$, and the primitive integer segment(s) connecting $\kappa$ to $w$. Neither $D_2$ nor $D_3$ intersects $b$, and the only curve in $\mathcal N$ intersecting $b$ is $A(0,1)$; this implies that $b \subset \overline{D_1}$ as claimed. 
 
Thus, in $(C \setminus \mathcal N) \cup A(0,1)$, the disks $D_2$ and $D_3$ are joined into a single disk $D^+$, while $D_1$ has two portions of its boundary joined to create an annulus with core curve $b$. Upon passing to $(C\setminus \{b\}) \setminus \mathcal N'$, this annulus is cut open to create two annular regions bounded by $b$, while the disks $D^+$ and $D_i$ for $i \ge 4$ are unaffected. Thus $\mathcal N'$ does determine a filling network on $C \setminus \{b\}$ as required. Arboreality of $\mathcal N'$ was established in Lemma \ref{lemma:Nprops}.\\
 
\para{From admissible twists to vanishing cycles} In order to address Question \ref{question:donaldson}, it is necessary to better understand the relationship between admissible twists and vanishing cycles. A first remark is that any vanishing cycle is necessarily an admissible curve, so it remains only to show the converse. We observe that if $\alpha$ is a loop in $\mathcal M(\mathcal L)$ based at $C_0$ that determines a vanishing cycle, then any conjugate $\beta \alpha \beta^{-1}$ also determines a vanishing cycle. To complete the argument, it therefore suffices to establish the following claim.
 
\begin{lemma}
Let $a$ be any admissible curve on $C_0$. Then $T_a$ is conjugate in $\Gamma$ to some twist $T_c$ for $c$ a vanishing cycle.
\end{lemma}
\begin{proof}
An admissible curve $a$ determines a vertex in the graph $\mathcal C_\phi^1(C_0)$ of Section \ref{section:trick}. Theorems \ref{theorem:Acurves} and \ref{theorem:piseg} together imply that $\Gamma$ has a generating set consisting entirely of vanishing cycles. Thus the set of vertices in $\mathcal C_\phi^1(C_0)$ corresponding to vanishing cycles is nonempty. 

We claim that if $a \in \mathcal C_\phi^1(C_0)$ is adjacent to some vanishing cycle $c$, then $a$ is also a vanishing cycle. Indeed, the condition that $a$ and $c$ are adjacent in $\mathcal C_\phi^1(C_0)$ is equivalent to $i(a,c) = 1$, and hence by the braid relation, 
\[
T_a = (T_c T_a) T_c (T_c T_a)^{-1}.
\]
As $T_c, T_a \in \Gamma$ by the first part of Theorem \ref{theorem:toric}, the above observation implies that $T_a$ is a vanishing cycle. The claim now follows from the connectivity of $\mathcal C_\phi^1(C_0)$ established in Lemma \ref{lemma:c1connected}.
\end{proof}

This concludes the proof of Theorem \ref{theorem:toric}.\qed

\bibliography{planecurvemonodromy}{}
\bibliographystyle{alpha}

\end{document}